\documentclass[11pt,reqno]{amsart}

\usepackage{amsthm}
\usepackage{amsmath}
\usepackage{amssymb}
\usepackage[mathscr]{euscript}
\usepackage{mathrsfs}
\usepackage{hyperref}
\usepackage{framed}

\usepackage[all]{xy}

\usepackage{latexsym}
\usepackage{palatino}

\title[Ringel duality for perverse sheaves on hypertoric varieties]{Ringel duality for perverse sheaves on hypertoric varieties}

\author{Tom Braden}
\author{Carl Mautner}

\address{Tom Braden \\
Dept.\ of Mathematics and Statistics\\
         University of Massachusetts, Amherst}
\email{braden@math.umass.edu}

\address{Carl Mautner \\
Department of Mathematics\\
University of California, Riverside}
\email{carl.mautner@ucr.edu}

\DeclareFontFamily{U}{mathx}{\hyphenchar\font45}
\DeclareFontShape{U}{mathx}{m}{n}{
      <5> <6> <7> <8> <9> <10>
      <10.95> <12> <14.4> <17.28> <20.74> <24.88>
      mathx10
      }{}
\DeclareSymbolFont{mathx}{U}{mathx}{m}{n}
\DeclareFontSubstitution{U}{mathx}{m}{n}
\DeclareMathAccent{\widecheck}{\mathord}{mathx}{"71}

\newtheorem{theorem}{Theorem}[section]
\newtheorem*{theorem*}{Theorem}
\newtheorem{lemma}[theorem]{Lemma}
\newtheorem{proposition}[theorem]{Proposition}
\newtheorem{corollary}[theorem]{Corollary}

{
\theoremstyle{definition}
\newtheorem{definition}[theorem]{Definition}
\newtheorem{example}[theorem]{Example}

\newtheorem{remark}[theorem]{Remark}
\newtheorem*{remark*}{Remark}
}

\newcommand{\excise}[1]{}

\newcommand{\rank}{\operatorname{rk}}
\newcommand{\rk}{\operatorname{rk}}

\newcommand{\im}{\operatorname{im}}
\newcommand{\id}{\operatorname{id}}

\renewcommand{\dim}{\operatorname{dim}}
\newcommand{\codim}{\operatorname{codim}}

\newcommand{\Sym}{\operatorname{Sym}}
\newcommand{\Hom}{\operatorname{Hom}}
\newcommand{\End}{\operatorname{End}}
\newcommand{\RHom}{\operatorname{\mathcal{RH}\mathit{\!om}}}

\newcommand{\C}{{\mathbb{C}}}
\newcommand{\Z}{{\mathbb{Z}}}

\renewcommand{\H}{{\mathbb{H}}}
\newcommand{\R}{{\mathbb{R}}}

\newcommand{\half}{\frac{1}{2}}

\renewcommand{\iff}{\Leftrightarrow}

\newcommand{\M}{\mathfrak{M}}

\newcommand{\B}{\mathcal{B}}

\newcommand{\D}{\mathbb D}

\renewcommand{\a}{\alpha}
\renewcommand{\b}{\beta}

\newcommand{\sig}{\sigma}

\newcommand{\F}{\mathcal{F}}

\renewcommand{\mod}{\mathbin{\!/\!\!/\!}}

\renewcommand{\t}{\mathfrak{t}}

\newcommand{\IC}{\mathop{\mathbf{IC}}}

\newcommand{\cal}{\mathcal}

\newcommand{\cB}{{\cal B}}
\newcommand{\cC}{{\cal C}}

\newcommand{\cE}{{\cal E}}
\newcommand{\cF}{{\cal F}}
\newcommand{\cH}{{\cal H}}

\newcommand{\cO}{{\cal O}}

\newcommand{\cS}{{\mathscr S}}
\newcommand{\cV}{{\cal V}}

\newcommand{\ol}{\overline}

\newcommand{\wt}{\widetilde}
\newcommand{\wh}{\widehat}
\newcommand{\md}{\mathrm{-mod}}

\newcommand{\be}{{\mathbf{e}}}
\renewcommand{\setminus}{\smallsetminus}

\newcommand{\kk}{{ k}}

\renewcommand{\emptyset}{\varnothing}
\DeclareMathOperator{\Perv}{Perv}
\DeclareMathOperator{\coker}{coker}

\newcommand{\T}{\Omega}                  
\newcommand{\Morse}{\bar\Phi}       

\newcommand{\Bas}{\mathop{\mathrm{Bas}}\nolimits}

\newcommand{\cU}{{\mathcal U}}
\newcommand{\cUc}{{\widecheck{\mathcal U}}}

\newcommand{\Rc}{\widecheck{R}}

\newcommand{\ldiv}{\dashv}
\newcommand{\rdiv}{\vdash}
\newcommand{\uc}{{\check{u}}}
\newcommand{\bmu}{{\boldsymbol{\mu}}}
\newcommand{\bV}{\mathbf{V}}

\newcommand{\uk}{\underline{\kk}} 

\newcommand{\ha}{\mathbin{\hat{\ast}}}
\newcommand{\hcd}{\mathbin{\hat{\cdot}}}

\begin{document}

\maketitle

\begin{abstract}
Motivated by the polynomial representation theory of the general linear group and the theory of symplectic singularities, we study a category of perverse sheaves with coefficients in a field $k$ on any affine unimodular hypertoric variety $\M$.  Our main result is that this is a highest weight category whose Ringel dual is the corresponding category for the Gale dual hypertoric variety $\M^!$.  On the way to proving our main result, we confirm a conjecture of Finkelberg--Kubrak in the case of hypertoric varieties.  We also show that our category is equivalent to representations of a combinatorially-defined algebra, recently introduced in a related paper.
\end{abstract}

\section{Introduction}

Let $k$ be a field and let $\M$ be an affine unimodular hypertoric variety.  The variety $\M$ is endowed with a natural torus action and is stratified by symplectic leaves.  We consider the category $\Perv(\M,k)$ of torus equivariant perverse sheaves with coefficients in $k$ that are constructible for the stratification by symplectic leaves.  Our main result is that the category $\Perv(\M,k)$ is a highest weight category whose Ringel dual is the category $\Perv(\M^!,k)$ for the Gale dual hypertoric variety $\M^!$.

Our motivation for this result comes from the observation in \cite{Mautner} that the geometry of the nilpotent cone $\mathcal{N} \subset \mathfrak{gl}_n(\C)$ encodes the polynomial representation theory of the general linear group $GL_n$ over $k$.  More precisely, in \textit{loc. cit.} it is shown that there is an equivalence between the category of $GL_n(\C)$-equivariant perverse sheaves with coefficients in $k$ on $\mathcal{N}$ and the category of finitely generated modules of the Schur algebra $S_k(n,n)$ over $k$.

On the other hand, recent work in geometric representation theory suggests that for any symplectic singularity (or resolution) $X$ and any representation theory encoded in the geometry of the nilpotent cone (or its Springer resolution) there should be similar representation theory to be found in the geometry of $X$.  An example of this phenomenon is the definition of a category $\mathcal O$ associated to any symplectic resolution introduced in~\cite{BLPWgco}.

Following this line of reasoning, one might hope that for such a symplectic singularity there is a category of perverse sheaves with coefficients in $k$ on $\M$ that is equivalent to the category of representations of a "Schur-like" algebra.  The results of the current paper can be viewed as confirming this hope in the case when $X=\M$ is an affine unimodular hypertoric variety.

In addition to interest within geometric representation theory, these results have connections to combinatorics and modular representation theory.  We begin to pursue the connection to combinatorics in~\cite{BMMatroid}, where it leads to new questions about the combinatorics of matroids.  In modular representation theory, the existence of a combinatorial cousin of the Schur algebra also gives hope that it might prove more accessible and in turn provide insight into the Schur algebra itself.

In the remainder of the introduction, we give a more precise formulation of our main result, a description of the strategy and techniques used in the proof, and a brief discussion of connections to other topics.

\subsection{Main results} Let $K \subset (\C^*)^n$ be a connected complex torus.  The induced action of $K$ on $T^* \C^n$ is Hamiltonian with an algebraic moment map $\mu_K:T^* \C^n \to \operatorname{Lie}(K)^*$.  Let $\M$ denote the associated hyperk\"ahler quotient, in other words the categorical quotient $\mu_K^{-1}(0) \mod K$.  The space $\M$ is known as an affine hypertoric variety.  It carries a residual action of the group $T:=(\C^*)^n/K$ and has dimension $\dim \M = 2\dim T$.  We assume that $\M$ is unimodular, or equivalently that $\M$ has a symplectic resolution.

The hypertoric variety $\M$ carries a natural Poisson structure and we let $\cS$ denote the stratification by symplectic leaves.  Let $\Perv(\M,k) = \Perv_{\cS,T}(\M,k)$ be the abelian category of perverse $k$-sheaves on $\M$ which are $T$-equivariant and constructible with respect to this stratification.  The isomorphism classes of simple objects in $\Perv(\M,k)$ are indexed by the set $\cF$ of symplectic leaves.

\begin{theorem*} The category $\Perv_{\cS,T}(\M,k)$ is highest weight with respect to the poset $\cF$, ordered by the closure relation.
\end{theorem*} 

\begin{remark} Like the Schur algebra, the category $\Perv(\M,k)$ is semisimple when the characteristic of $k$ is zero or of sufficiently large characteristic.  In fact, $\Perv(\M,k)$ is semisimple if the characteristic of $k$ is greater than $n$.  In~\cite{BMMatroid} we determine the precise set of characteristics of $k$ for which the category is semisimple.
\end{remark}

Consider the dual torus $T^\vee := \Hom(T,\C^*) \otimes_\Z \C^* \subset (\C^*)^n$.  Repeating the construction above for $T^\vee$ in place of $K$, one obtains an affine unimodular hypertoric variety $\M^!$, known as the Gale dual of $\M$, and the corresponding category of perverse sheaves $\Perv(\M^!,k)$.  There is a natural order-reversing bijection between the poset $\cF$ of symplectic leaves of $\M$ and the poset $\cF^!$ symplectic leaves of $\M^!$.

\begin{theorem*} The Ringel dual of $\Perv(\M,k)$ is $\Perv(\M^!,k)$ .
\end{theorem*} 

\begin{remark} This is perhaps surprising as the dimension of $\M$ is generally not the same as the dimension of $\M^!$.
\end{remark}

\subsection{Tilting and projective objects} We prove our theorem by constructing explicit tilting and projective generators, and then computing and identifying their endomorphism rings.

We first describe the construction of the tilting objects.  The closure $\M^F$ of each symplectic leaf $S_F \subset \M$ is itself a unimodular hypertoric variety, so it has a semi-small resolution $p: \widetilde{\M^F} \to \M^F \subset \M$.  Let $\T^F$ denote the pushforward of the constant sheaf along $p$ (after the correct shift).  This object is a perverse parity complex in the sense of \cite{JMW14}, and once we have proved that $\Perv(\M,k)$ is highest weight, \cite[3.3]{JMW16} shows that perverse parity complexes are tilting.  Taking the direct sum of $\T^F$ over all $F \in \cF$ gives the desired tilting generator.

Working still on $\M$, we define a projective object $\Pi_F$ to represent a certain exact functor $\Morse_F$ defined using hyperbolic restriction \cite{Br03} for the action of a cocharacter of $T$ whose fixed points are $\M^F$.  In general, hyperbolic restriction is not $t$-exact, but we show that in our situation it is, using an argument similar to one used by Mirkovic and Vilonen \cite{MV} for perverse sheaves on affine Grassmannians.  Conjecturally this will hold more generally for affinizations of other symplectic resolutions.

Our main theorem is obtained by studying the commuting actions of $\End(\bigoplus_F \T^F)$ and $\End(\bigoplus_E \Pi_{E})$ on the vector space $\bigoplus_{E,F} \Hom(\Pi_{E}, \T^F) = \bigoplus_{E,F} \Morse_{E}(\T^F)$.  Using equivariant localization, we show that these actions are faithful and give formulas for the action of certain generators of each algebra.  Ringel duality then follows from the fact that these generators swap places under Gale duality.

\begin{remark}
To prove that $\T^F$ and $\Pi_F$ are well-defined and independent of choices, we depart from the complex algebraic setting used throughout the rest of the paper in order to work in a topological setting.  In doing so, we obtain descriptions of $\T^F$ and $\Morse_F$ respectively as a nearby cycles sheaf and a vanishing cycles functor for the (real analytic) hyperk\"ahler moment map $\bmu\colon \M \to \mathfrak{t}_\R \otimes_\R \R^3$. The required independence then follows from hyperk\"ahler rotation, since the walls where the sheaf or functor would change have codimension three.  Our proofs of these results use specific properties of hypertoric varieties, but we conjecture that these results should hold for more general holomorphic symplectic quotients of symplectic affine spaces by a reductive group, such as Nakajima's quiver varieties.  As an added bonus, we confirm a conjecture of Finkelberg--Kubrak~\cite{FiKu} in the hypertoric case.

\end{remark}

\begin{remark}
Recall that the Schur algebra $S_k(n,n)$ is self-Ringel dual.  A geometric explanation of this fact was given in \cite{AM}, using a derived equivalence coming from the Fourier transform for sheaves on the Lie algebra $\mathfrak{gl}_n$.  We searched for a similar proof in the hypertoric setting, but were unable to find a comparable derived equivalence. 
\end{remark}

\subsection{Matroidal Schur algebras}  Motivated by our work on this paper, we have defined~\cite{BMMatroid} for any matroid $M$ a quasi-hereditary $k$-algebra $R(M)$.  Theorem~\ref{thm:bridge between papers} below explains the relationship between the two papers.  In particular, it follows that $\Perv(\M,k)$ is equivalent to the category of modules of a matroidal Schur algebra.  In~\cite{BMMatroid}, we show that these algebras are closely connected to the reduction modulo $p$ of integral matrices that were studied by Schechtman--Varchenko~\cite{SV} in the hyperplane setting and by Brylawski--Varchenko~\cite{BV} for matroids more generally.  We also describe connections to formulas of Kook--Reiner--Stanton~\cite{KRSlap} and Denham~\cite{Denham}.

\subsection{Hypertoric category $\cO$} The papers \cite{GDKD,BLPWtorico} define a different highest weight category, called hypertoric category $\cO$, associated to a hypertoric variety $\M$, and prove a Koszul duality relating the categories for $\M$ and $\M^!$.  However, those results are quite different from the ones in this paper.

Unlike $\Perv(\M,k)$, the hypertoric category $\cO$ depends on the choice of a resolution of $\M$ and a cocharacter of $T$ with isolated fixed points.
Category $\cO$ is only defined over $\C$ and is not semisimple.  The category $\Perv(\M,k)$ is semisimple in characteristic zero and for all sufficiently large characteristics.  The category $\cO$ has a graded lift; we do not know of such a lift for $\Perv(\M,k)$.  The simple objects in $\cO$ are indexed by the set of all fixed points of $\wt\M$, which is typically much larger than the set $\cF$ indexing simples in $\Perv(\M,k)$.
On the other hand, in hypertoric category $\cO$ the multiplicity of simples in standard modules are at most one.  In $\Perv(\M,k)$, there is no such bound.  The hypertoric variety associated to the braid arrangement $B_4$ gives one example (see~\cite[5.2.3]{BMMatroid} for the analogous matroid example).

However, we do expect that there is a relationship between our results and hypertoric category $\cO$. This is because the space $\Morse_E(\T^F)$ can be naturally identified with the Grothendieck group
of the category $\cO^F_E$ associated to a normal slice to $S_E$ inside $\overline{S_F}$, which is homeomorphic to a hypertoric variety $\M^F_E$; see \cite[Section 6.3]{BLPWtorico}.
Our generators of the algebra $\End(\bigoplus_F \T^F)$ correspond to classes of finite-dimensional simples in the category $\cO^F_E$, and the generators of $\End(\bigoplus_F \Pi_F)$ correspond to the classes of projective covers of simples with maximal GK dimension. 
It should then be possible to categorify the actions of the rings $\End(\bigoplus_F \Omega^F)$ and $\End(\bigoplus_E \Pi_E)$ by defining appropriate induction and restriction functors between the categories $\cO^F_E$.

\subsection{Contents} The paper is organized as follows.  Section 2 contains the basic combinatorial facts about arrangements that we need.  Section 3 introduces the main actors: hypertoric varieties, the stratification, the resolution sheaves $\T^S$, and the exact functors $\Morse_{F}$.  Section 4 collects some results about
equivariant localization of our sheaves and homomorphisms between them.
Section 5 uses these results to compute homomorphisms between resolution sheaves, while Section 6 computes the action of the generating homomorphism between the functors $\Morse_{F}$.  The final proof that these actually do generate is a by-product of the proof in Section 7 that Gale duality exchanges the actions of projectives and tiltings.  In Section 8 we show that the quasi-hereditary algebra which we get is isomorphic to one which is constructed in \cite{BMMatroid} for any matroid (realizable or not).
Section 9 contains topological and stratification-theoretic proofs and an application to the Finkelberg--Kubrak conjecture.

\section{Arrangements and the poset of flats}\label{arrange}

\subsection{Integral arrangements}\label{sec:arrangements}
Let $I$ be an indexing set with $n=|I|$ elements.  For any subset $S\subset I$ we have the natural coordinate inclusion and projection $\Z^S \hookrightarrow \Z^I$ and $\Z^I \twoheadrightarrow \Z^S$.  

Let $V \subset \Z^I$ be a saturated sublattice of rank $d$.
We think of the embedding of $V$ into $\Z^I$ as defining a central arrangement $\cH$ of hyperplanes in $V$, with the $i$th hyperplane given by the intersection of $V$ with the hyperplane $x_i = 0$ in $\Z^I$.
To be more precise, this is a cooriented multiarrangement: each hyperplane is equipped with a linear form which cuts it out, and each hyperplane can appear multiple times, with the same or different linear forms.  

%
%
The arrangement induces a matroid $M(V) = M(\cH)$ whose independent sets are the subsets $S \subset I$ for which the linear forms $x_i, i\in S$ are independent, or in other words the composition of the inclusion $V \hookrightarrow \Z^I$ and the projection $\Z^I \twoheadrightarrow \Z^S$ has finite cokernel.
A \emph{basis} is a maximal independent set; all bases therefore have $\rank V$ elements.  

We assume throughout this paper that $V$ is \emph{unimodular}, which means that for every basis $B$ the map $V \hookrightarrow \Z^I \twoheadrightarrow \Z^B$ is an isomorphism.  If $V$ is the image of a homomorphism $\Z^m \to \Z^n$, $m \le n$ given by a matrix $A$, then 
 unimodularity of $V$ is equivalent to  all maximal minors of $A$ being in the set $\{-1, 0, 1\}$.  

To avoid degenerate situations, it will be useful for our main results to assume that the matroid $M(V)$ has no loops, which means that $V$ is not contained in any coordinate hyperplane  $\{x_i = 0\}$ and also no coloops, which means that $V$
does not contain any nonzero multiple of a coordinate vector $e_i$. In other words, every $i \in I$ is in some basis and no $i$ is in every basis.
Note that the perpendicular space $V^\bot \subset \Z^I$ taken with respect to the standard pairing will again be unimodular, and it will also have no loops or coloops.

\subsection{Flats}

A \emph{flat} of $V$ is any intersection of a subcollection of the hyperplanes $H_i$.  We will identify each flat with the set of hyperplanes that contain it, so that if we put 
$H_S := \bigcap_{i\in S} H_i$ for any $S \subset I$, a subset $F \subset I$ defines a flat if and only if $F = \{i \mid H_F \subset H_i\}$.

We order the flats by inclusion as subspaces of $V$, i.e.\ by \emph{reverse} inclusion as subsets of $I$.  The arrangement has a unique minimal flat $I$ and unique maximal flat $\emptyset$.

Given a flat $F$, we have the lattice $V^F := H_F$
in $\Z^{I\setminus F}$.
The fact that $F$ is a flat implies that $V^F$ has no loops, and it clearly has no coloops since $V$ is coloop-free.  The \emph{rank} of $F$ is defined to be $r(F) = \rank V^F$.

Dually, we can consider the lattice $V_F := V/V^F$ inside $\Z^F$.  The fact that 
$V$ is unimodular implies that $V_F$ is saturated.

Since $V$ has no loops, $V_F$ also has no loops, but $V_F$ can have coloops.  We then define the poset $\cF = \cF(V)$ to be the collection of all coloop-free flats, meaning flats $F$ such that $V_F$ has no coloops.  Coloop-free flats  are sometimes also known as \emph{cyclic} flats; they can be described alternatively as flats which are unions of circuits (minimal dependent sets).

The poset $\cF_F$ of coloop-free flats of $V_F$ is isomorphic to the interval 
$[F, \emptyset]$ in $\cF$, and the poset $\cF^F$ of coloop-free flats of $V^F$ is isomorphic to the interval $[I,F]$.  Furthermore, for flats $E\le F$ in $\cF$, we have 
$(V^F)_E =(V_E)^F = V^F/V^E$ in  
$\Z^{E \setminus F}$.  We let $V^F_E$ denote this sublattice.

We denote by $d^F_E$ the rank of the lattice $V^F_E$, and we put $d^F := d^F_I$, $d_E := d^\emptyset_E$.

%

\section{Hypertoric varieties}

\subsection{Affine hypertoric varieties}

Let $\M = \M(V)$ be the affine hypertoric variety corresponding to a sublattice $V \subset \Z^I$.  
For this section, we do not assume that $V$ has no loops or coloops, but we do assume that it is saturated and unimodular.

Recall that $\M(V)$ is defined as follows.  Let $K \subset (\C^*)^I$ be the kernel
of the natural homomorphism
\[(\C^*)^I = \Hom_{\mathrm{AbGp}}(\Z^I, \C^*) \to \Hom_{\mathrm{AbGp}}(V, \C^*).\]
It is a connected torus with
Lie algebra $\mathfrak{k} = V_\C^\perp \subset \C^I$.  
The coordinate action of $(\C^*)^I$ on $\C^I$ induces a Hamiltonian action on $T^*\C^I$ with complex moment map
\[\mu_I\colon T^*\C^I \to \C^I, \;\; (z_i,w_i)_{i\in I} \mapsto (z_iw_i)_{i\in I}.\]
The induced action of $K$ on $T^* \C^I$ has moment map
\[ \mu_K \colon T^* \C^I \to \mathfrak{k}^*\]
given by composing $\mu_I$ with the projection $\C^I = \operatorname{Lie}{(\C^*)^I}^*\to \mathfrak{k}^*$.

The affine hypertoric variety $\M = \M(V)$ is the categorical quotient \[\mu^{-1}_K(0) \mod K = \operatorname{Spec}(\C[\mu^{-1}_K(0)]^K).\]  It 
is a singular Poisson variety of 
dimension $2d = 2\rk(V)$, and it has a natural Hamiltonian action of the quotient torus $T= (\C^*)^I/K$.

Requiring $V$ to have no loops does not restrict the class of varieties we are considering: if $i\in I$ is a loop, then the variety $\M$ is isomorphic to the hypertoric variety defined by the inclusion of $V$ into 
$\Z^{I\setminus i}$.  We will see in the next section that
asking that $V$ has no coloops also does not impose a serious restriction.

For more details about hypertoric geometry, see \cite{Pr08}.

\subsection{Stratifications of $\M$}
In this section we describe two stratifications of $\M$ which we will need. 

For any flat $F$ of $V$, let $\M^F \subset \M$ be the subvariety given by the equations $z_i = w_i = 0$ for all $i \in F$.  It is isomorphic to the affine hypertoric variety $\M(V^F)$.
For any flats $E$, $F$, we have $\M^E \subset \M^F$ if and only if $E \le F$.

This family of closed subvarieties induces a disjoint decomposition into locally closed manifolds
\[\breve{S}_F := \M^F \setminus \bigcup_{E < F} \M^E\]
indexed by the set of all flats and another, coarser, decomposition $\M = \bigcup_{F \in \cF} S_F$ indexed by cyclic flats $\cF$, 
where we put 
\[S_F := \M^F \setminus \bigcup_{\substack{E \in \cF \\ E < F}} \M^E.\]

The first decomposition is just the classification of points of $\M$ by their $T$-stabilizers: we have $x\in \breve{S}_F$ if and only if 
\[T_x = T_F := (\C^*)^F/K_F \subset T,\]
where $K_F = K\cap (\C^*)^F$.
On the other hand the varieties $S_F$ are the images of the loci of points in $\mu_K^{-1}(0)$ with fixed $K$-stabilizer.  They can also be described\footnote{This uses the hypothesis that $V$ is unimodular.} as the symplectic leaves for the Poisson structure induced from the natural one on $T^*\C^I$.

We will refer to the strata $\breve S_F$ as ``fine" strata and the strata $S_F$ as ``coarse" strata.  Our sheaves will all be constructible with respect to the stratification by coarse strata; the fine strata will mainly appear as a tool in certain technical proofs.

Let $T_{F,\R}$ denote the maximal compact subgroup of the complex torus $T_F$.  A proof of the following proposition is given in \cite[2.5]{BP09}, and similar results appear in \cite[2.4 and 2.5]{PW}.  We will prove a more general result in Proposition \ref{prop:parametric stratification} below.

\begin{proposition} \label{prop:slice to fine strata}
The decomposition $\M = \bigcup_F \breve{S}_F$ is a topological stratification; 
a normal slice to 
a stratum $\breve{S}_F$ is isomorphic as a $T_{F,\R}$-stratified space to the hypertoric variety $\M_F := \M(V_F)$.
\end{proposition}

We give a precise definition of topological stratifications in section \ref{sec:big proofs} below; for what follows it is enough to know that a topological stratification is topologically locally constant along strata, and that the functors $j_*$, $j_!$, $j^*$, $j^!$ and Verdier duality $\D$ preserve constructibility for the stratification, where $j$ is the inclusion of a locally closed union of strata into $\M$.

\begin{lemma}\label{lem:coloop decomposition}
If $C\subset I$ is the set of all coloops of $V$, then there is an isomorphism
\[\M \cong \M_E \times T^*(\C^{C}),\]
where $E = I\setminus C$.  Under this isomorphism
the stratum $\breve{S}_F$ of $\M$ is sent to  $\breve{S}_{E,F} \times T^*\C^{C}$, where 
 $\breve{S}_{E,F}$ is the stratum of $\M_E$ indexed by $F \subset E$.
\end{lemma}

\begin{proof}
The fact that elements of $C$ are coloops means that
$\Z^{C} \subset V$, so the projection $\Z^I \to \Z^{E}$ 
induces an isomorphism $\Z^I/V \stackrel{\sim}{\to} \Z^E/V_E$.  It follows that $K_{E}$ is isomorphic to $K$, acting trivially on the coordinates in $C$, and the moment map $\mu_K$ does not depend on the entries in $T^*(\C^C)$.  The result follows.
\end{proof}

\begin{corollary}
The decomposition $\M = \bigcup_{F\in \cF} {S}_F$ is a topological stratification; 
a normal slice to 
a stratum ${S}_F$ is isomorphic as a $T_{F,\R}$-stratified space to $\M_F$.
\end{corollary}

\begin{proof}
Given a point $p\in \breve{S}_{F'} \subset S_F$, a normal slice $N$ to $\breve{S}_{F'}$ at $p$ is 
isomorphic to $\M_{F'}$.  Since $F$ is the smallest flat in $\cF$ lying above $F'$, the set $F' \setminus F$ consists of all coloops of $V_{F'}$.   Then the Lemma implies that $N$ is isomorphic as a stratified space to 
$\M_F$ times a symplectic affine space. 
\end{proof}

\begin{remark}
If $V$ is not unimodular, then $V_F$ may not be a saturated sublattice, which means that $K_F$ is not connected.  The slice $N$ will be isomorphic to the quotient of $\M_F \times T^*\C^{F'\setminus F}$ by a finite group.
 As a result the coarse strata $S_F$ are only orbifolds.  
\end{remark}

Let $\cS$ denote the stratification by $S_F$, $F \in \cF$, and let $\breve{\cS}$ denote the stratification by
$\breve{S}_F$.  
For any flat $F \in \cF$, the same construction gives stratifications
$\cS^F$, $\breve{\cS}^F$ of $\M^F = \overline{S_F}$ and stratifications $\cS_F$, $\breve{\cS}_F$ of $\M_F$.  Since the coloop-free flats of $V^F$ are just the coloop-free flats of $V$ which are contained in $F$, the stratification $\cS^F$ is just the restriction of $\cS$ to $\M^F$.

\subsection{Perverse sheaves on $\M$}
Fix a field $\kk$.  Consider the triangulated category $D^b_{T,\cS}(\M, \kk)$ of equivariant complexes which are constructible with respect to the stratification $\cS$.  Throughout the paper, we use $f_*$, $f_!$, etc. to denote the derived pushforward and proper pushforward functors.

In this paper, we always use the middle perversity.
Let $\Perv(\M,\kk)$ denote the abelian category of perverse objects in $D^b_{T,\cS}(\M, \kk)$.

\begin{lemma}\label{lem:equivariantly simply connected}
 The equivariant fundamental group $\pi_1(ET \times^T S_F)$ is trivial.  In particular,
 any $T$-equivariant local system on $S_F$ is constant.
\end{lemma}

\begin{proof}  From the $T$-fibration $ET \times S_F \to ET \times^T S_F$, we obtain a long exact sequence
\[ \dots \to \pi_1(T) \to \pi_1(ET \times S_F) \to \pi_1(ET \times^T S_F) \to \pi_0(T).\]
As the last term is trivial, we must show that the homomorphism $\pi_1(T) \to \pi_1(ET \times S_F) = \pi_1(S_F)$ induced by the map from $T$ to any orbit is surjective.

The ``fine stratum" $\breve{S}_F$ 
is a Zariski open subset of $S_F$, hence $\pi_1(\breve{S}_F) \to \pi_1(S_F)$ is surjective
(see \cite[2.10]{Kol95}, for instance). 
Thus it is enough to show that 
$\pi_1(T)\to \pi_1(\breve{S}_F)$ is surjective.

If $T_\R$ denotes the maximal compact torus in $T$, the
stabilizer $T_{\R,x}$ of $x\in S_F$ does not depend on the
point $x$, so $\breve{S}_F \to \breve{S}_F/T_\R$ is a fibration with fiber $T_\R/T_{\R,x}$. The base of this fibration simply connected, since it is homeomorphic to the complement in a real affine space of a union of affine subspaces of codimension at least three --- see the proof of \cite[Corollary 2.6]{BP09} or Section \ref{sec:hk moment map} below.
  Thus by the long exact homotopy sequence,  $\pi_1(T_\R/T_{\R,x}) \to \pi_1(\breve{S}_F)$ is surjective.  Since the stabilizer $T_{\R,x}$ is connected, $\pi_1(T_\R) \to \pi_1(T_\R/T_{\R,x})$ is also surjective, proving the result.
\end{proof}

\begin{proposition}\label{prop:forgetting equivariance}
The forgetful functor from $\Perv(\M,k)$ to the category of non-equivariant perverse $k$-sheaves on $\M$ is a fully faithful embedding.
\end{proposition}
\begin{proof}
Take $Q, \bar{Q} \in \Perv(\M,k)$.  The natural homomorphism from the space of equivariant homomorphisms $Q \to \bar{Q}$ to non-equivariant homomorphisms can be identified with the homomorphism 
\begin{equation*} 
\H_T^0(\RHom(Q,\bar{Q})) \to \H^0(\RHom(Q,\bar{Q}))
\end{equation*}
from equivariant to non-equivariant hypercohomology.

We will make repeated use of the following easy homological fact: if $\phi\colon C^\bullet \to D^\bullet$ is a map of cochain complexes such that $C^i = D^i = 0$ for $i <0$, and $\phi^0$ is an isomorphism and $\phi^1$ is an injection, then $H^0(\phi)\colon H^0(C^\bullet) \to H^0(D^\bullet)$ is an isomorphism and
$H^1(\phi)\colon  H^1(C^\bullet) \to H^1(D^\bullet)$ is an injection.

Choose an ordering $F_1, F_2,\dots, F_r$ of the elements of $\cF$ so that the union $\bigcup_{j=1}^s S_{F_j}$ is closed for $s = 1,\dots, r$.  There is a spectral sequence abutting to $\H_T^{p+q}(\RHom(Q,\bar{Q}))$ with $E_1$-term
\[E_1^{p,q} = \H_T^{p+q}\RHom(j_{F_p}^*Q, j_{F_p}^!\bar{Q}),\]
which maps to another spectral sequence abutting to $\H^{p+q}(\RHom(Q,\bar{Q}))$ with $E_1$-term $\bar{E}_1^{p,q} = \H^{p+q}\RHom(j_{F_p}^*Q, j_{F_p}^!\bar{Q})$ (see \cite[3.4]{BGS}, for instance).
By the fact above, our result will follow if we can show that
\begin{enumerate}
\item[(a)] $E_1^{p,q} = \bar{E}_1^{p,q} = 0$ if $p+q < 0$, and
\item[(b)] $E_1^{p,q} \to \bar{E}_1^{p,q}$ is an isomorphism if $p+q = 0$, and an injection if $p+q = 1$.
\end{enumerate}

Let $A = \RHom(j_{F_p}^*Q, j_{F_p}^!\bar{Q})$.  The $\cS$-constructibility of $Q, \bar{Q}$ implies that
the cohomology sheaves $\cH^j(A)$ are equivariant local systems, and so they are sums of constant sheaves by Lemma \ref{lem:equivariantly simply connected}.  Furthermore, we have  $\cH^j(A)=0$ if $j<0$ since $Q$, $\bar{Q}$ are perverse.  We have spectral sequences
\['\!E_2^{p,q} = \H^p_T(\cH^q A) \Rightarrow \H^{p+q}_T(A), \,\text{and}\]
\['\!\bar{E}_2^{p,q} = \H^p(\cH^q A) \Rightarrow \H^{p+q}(A).\]
These facts, along with the fact that $H^0_T(S_F;k) \to H^0(S_F;k)$ is an isomorphism and
$H^1_T(S_F, k) = 0$ (which follows from Lemma \ref{lem:equivariantly simply connected}) imply that
the map $'\!E \to {}'\!\bar{E}$ of spectral sequences satisfies (a) and (b).  Using our homological fact again 
shows that 
(a) and (b) hold for $E \to \bar{E}$ as well.
\end{proof}

Let $\IC_F \in \Perv(\M,\kk)$ denote the simple perverse sheaf that is the Goresky-MacPherson extension of the equivariant constant sheaf $\uk_{S_F, T}[2d^F]$ on $S_F$.

\begin{corollary}
Every simple object in $\Perv(\M,\kk)$ is isomorphic to $\IC_F$ for some $F \in \cF$. \qed
\end{corollary}

For each $F\in \cF$, let $j_F \colon S_F \to \M$ denote the inclusion.  We consider the standard objects $\Delta_F := {}^p j_{F!} \uk_{S_F,T}[2d^F]$ and costandard objects $\nabla_F := {}^p j_{F*} \uk_{S_F,T}[2d^F]$.  Adjunction gives a natural map
\[ \Delta_F \to \nabla_F \]
with image isomorphic to $\IC_F$, and $\Delta_F \to \IC_F$ is a projective cover in the subcategory $\Perv(\M^F,\kk)$.

\begin{proposition}
\label{prop:enoughprojs}
The category $\Perv(\M,\kk)$ has enough projective objects.
\end{proposition}

\begin{proof}
The proposition is immediate from point (3) of the remarks following Theorem 3.2.1 in \cite{BGS}.
\end{proof}

For any flat $F$ (not necessarily in $\cF$), we have a restriction functor
\[\operatorname{Rest}_F\colon D^b_{T,\breve\cS}(\M,k) \to D^b_{T_F,\breve\cS_F}(\M_F,k)\]
given by restricting to a $T_F$-equivariant structure, pulling back by the inclusion of $\M_F$ as a normal slice to $S_F$ in $\M$, and shifting degree up by $d-d_F$.  (More precisely, since
the identification of normal slices given by Proposition \ref{prop:slice to fine strata} is only equivariant for the compact torus $T_{F,\R}$, here we must use the 
fact that for $\breve\cS$-constructible sheaves, $T$-equivariance and $T_\R$-equivariance are the same.)
This functor is $t$-exact, and it sends $\cS$-constructible objects to $\cS_F$-constructible objects.

Finally, we will need to know the following parity vanishing statement for our strata, which means that the theory of parity sheaves \cite{JMW14} applies to our situation.

\begin{proposition}\label{prop:parity vanishing}
The equivariant cohomology $H^\bullet_T(\breve S_F)$
of any fine stratum vanishes in odd degrees.
\end{proposition}
\begin{proof}
This is proved in \cite[Proposition 5.22]{BP09}; although that paper assumed the coefficient ring of the sheaves had characteristic zero, that assumption was not needed for the proof.
\end{proof}

\subsection{Resolution sheaves}
We identify the lattice of characters $X^*(K)$ of $K$ with a subset of $\mathfrak k^*_{\R} = \R^I/V_\R$ 
in the usual way.  
Call $\alpha\in \mathfrak k^*_{\R}$ \emph{generic} if
every element of $\alpha + V_\R$ has at least $|I| - d$ non-zero entries.  Equivalently, $\alpha$ is generic if every collection of $m$ of the hyperplanes $\{x_i = 0\}$ in $\alpha + V_\R$ which have nonempty intersection intersect in codimension $m$. 
  
For
 $\alpha \in X^*(K)$, we can consider the GIT quotient $\M_\a$ of $\mu_K^{-1}(0)$ by $K$ with linearization $\alpha$.  If $\alpha$ is generic, then $\M_\alpha$ is smooth.  (For non-unimodular arrangements it is only an orbifold.)

The construction of $\M_\a$ as a GIT quotient gives rise to a natural $T$-equivariant morphism $p_\alpha\colon \M_\alpha \to \M$.  It is proper and semismall \cite[2.7]{PW}, and an isomorphism over the largest stratum $S_\emptyset$.

The following result will be proved in Section \ref{proof:Omega is canonical}.
\begin{proposition}\label{prop:Omega is canonical}
For any two generic characters $\a$, $\a'$, there is a 
canonical isomorphism
\[(p_\alpha)_*\uk_{\M_\alpha,T}[2d] \cong (p_{\alpha'})_*\uk_{\M_{\alpha',T}}[2d].\]
\end{proposition}

Let $\Omega$ denote the resulting sheaf.  Since $\M_\a$ is smooth and $p_\a$ is semismall, it is a self-dual perverse sheaf.  By repeating this construction using a
resolution of the
 closed subvariety $\M^F$ with $F \in \cF$, we 
obtain a canonical perverse sheaf $\T^F$ with support equal to $\M^F$.

\begin{lemma}\label{lem:slice to resolution sheaf}
For any flat $E$, not necessarily in $\cF$,  $\operatorname{Rest}_E (\T^F)$ is isomorphic to $\T^F_E$, the resolution sheaf in 
$\Perv(\M_E, k)$ supported on $\M^F_E \subset \M_E$. 
\end{lemma}

\begin{proof}
This follows from \cite[2.5]{PW}, which says that any point in the stratum $\breve{S}_E$ has a 
neighborhood $U$ such that $p_\a^{-1}(U) \to U$ is homeomorphic to $\M_{E,\bar\a} \times T^*\C^E \to \M_E \times T^*\C^E$, where $\bar{\a}$ is the restriction of $\a$ to $K_E$. 
\end{proof}

Combining this with Lemma \ref{lem:coloop decomposition} gives the following result.

\begin{corollary}\label{cor:omega is constructible}
The perverse sheaf $\T^F$ is $\cS$-constructible, so it lies in $\Perv(\M,k)$.
\end{corollary}

\subsection{Morse functors}\label{sec:morse-functors}

Let $\xi$ be a cocharacter of $T$ and let $\M^\xi \subset \M$ denote the $\xi$-fixed points of $\M$.  Consider the functor
\[\Phi_\xi \colon D^b_T(\M,k)\to D^b_T(\M^\xi,k)\]
defined as follows.  
Let $\M^+_\xi$ be the attracting set of $\xi$ and $f\colon \M^\xi \to \M^+_\xi$ and $g\colon \M^+_\xi \to \M$ be the closed inclusions.  Then we define $\Phi_\xi(S) := f^!g^*S$.
This is an example of the hyperbolic restriction functor considered in
 \cite{Br03}.  In particular \cite[Theorem 1]{Br03} implies:

\begin{lemma}
\label{lem:hyperloc}
There is a natural equivalence of functors $\Phi_{-\xi} \cong \D_{\M_{}^\xi}\, \Phi_{\xi}\, \D_\M$.
\end{lemma}

For any flat $F$, the subtorus $T_F$ fixes $\M^F$.  We say 
that a cocharacter $\xi$ of $T_F$ is \emph{generic} if its image is not contained in $T_E$ for any $E \ge F$, or in other words, if the fixed points of $\xi$ are exactly $\M^F$.

\begin{proposition}\label{prop:vanishing cycles iso}
If $\xi$ is a generic cocharacter of $T_F$, then 
$\Phi_\xi$ restricts to a functor
\[\Phi_F\colon D^b_{T,\breve{\cS}}(\M, k) \to D^b_{T,\breve{\cS^F}}(\M^F,k).\]
Up to a unique isomorphism, this restriction is independent of $\xi$.
\end{proposition}
We postpone the proof until Section \ref{sec:hyperloc equals vanishing}.

More generally, for $E \leq F$, define $\Phi^F_E\colon D^b_{T}(\M^F) \to D^b_{T}(\M^E)$ by applying the same construction for a generic cocharacter of $T^F_E = T_E/T_F$.

\begin{lemma}\label{lem:Morse transitivity}
For any coloop-free flats $E \leq F$, there is a natural isomorphism $\Phi_E \cong \Phi^F_E \, \Phi_F$.
\end{lemma}

\begin{proof}
Choose generic cocharacters $\xi$ of $T_F$ and $\eta$ of $T^F_E$  so that $\Phi_F = \Phi_\xi$ and $\Phi^F_E = \Phi_\eta$.  Let $\bar\eta$ be a lift of $\eta$ to a cocharacter of $T_E$, and consider
the cocharacter $\zeta := n\xi + \bar\eta$ of $T_E$.  For $n \gg 0$, we will have $\M^+_\zeta \subset \M^+_\xi$, and we 
have the following diagram:
\[
\xymatrix{
& \M^\xi \ar@<.5ex>[r]^f &  \M^+_\xi \ar@<.5ex>[l]^p \ar[r]^g & \M \\
\M^\zeta \ar@<.5ex>[r]^{\bar{f}} & (\M^\xi)^+_\zeta \ar@<.5ex>[l]^{\bar{p}}\ar@<.5ex>[r]^j \ar[u]^{\bar{g}} & \M^+_\zeta \ar[u]^i \ar@<.5ex>[l]^q & 
}
\]
Here the first two spaces on the top row are 
the fixed and attracting sets for the action of $\xi$ on $\M$, and the first two spaces on the bottom are
the fixed and attracting sets for the action of $\zeta$ (or equivalently $\eta$) on $\M^\xi = \M^F$.  The maps going right and up are the inclusions, and the maps going left are the projections given by taking limits for the action of $\xi$ (middle column) and $\zeta$ (left column).  The center square is Cartesian (read both ways).

Using this, we have
\begin{align*}
\Phi^F_E \circ \Phi_F & = \bar{f}^!\bar{g}^*f^!g^*\\
 & \simeq \bar{p}_!\bar{g}^*p_!g^* \\
 & \simeq (\bar{p}q)_!(gi)^* \;\;\text{(base change)}\\
 & = \Phi_E.\qedhere
\end{align*}
\end{proof}

\begin{lemma}\label{lem:Morse slice compatibility}
For any flats $E \le F$, we have a natural isomorphism of 
functors $D^b_{T,\breve\cS}(\M,k) \to D^b_{T_E,\breve{\cS}^F_E}(\M^F_E,k)$
\[\operatorname{Rest}_E  \Phi_F \simeq \Phi_{E,F} \operatorname{Rest}_E,\]
where $\Phi_{E,F}$ denotes the hyperbolic restriction functor from sheaves on $\M_E$ to $\M^F_E$.
\end{lemma}

\begin{proof}
It is enough to show that under the local homeomorphism of a neighborhood of a point $x\in S_E$ with $\M_E \times \C^{\dim S_E}$, the attracting set $\M^+$ of a generic cocharacter of $T_F$ is sent to $\M_E^+ \times \C^{\dim S_E}$, where $\M_E^+$ is the attracting set in $\M_E$ for the same cocharacter.
This makes sense since $T_F \subset T_E$, so this cocharacter does act on $\M_E$.

As explained in the proof of \cite[Lemma 2.4]{PW}, the projection from a neighborhood of $x$ to $\M_E$ can be constructed as follows.  Let $K_E \subset (\C^*)^E$
be the torus defined in the same way as $K$ using the subspace $V_E\subset \Z^E$.  It is the intersection of $K$
with the coordinate subtorus $(\C^*)^E \subset (\C^*)^I$, 
and its Lie algebra $\mathfrak{k}_E$ is
 the orthogonal space to $V_{E,\C}$ inside $\C^F = \operatorname{Lie}(\C^*)^F$.  Let 
$\mu_E = \mu_{K_E}\colon T^*\C^E \to \mathfrak{k}_E^*$ be the corresponding moment map.

We can choose a subset $S \subset I \setminus E$ so that the projection of $K$ onto $(\C^*)^S$ is surjective, with kernel exactly $K_E$.  Suppose that $x$ is represented by 
a point in 
\[U= \{[z_i,w_i]_{i\in I}\in \mu^{-1}(0) \mid z_i \ne 0 \;\text{for all}\; i\in S\}.\]
Then for any point $[z_i,w_i]$ in $U$ there is a point $[z_i',w_i']$ in the same $K$-orbit with $z_i = 1$ for all $i\in S$, and any two such points are in the same $K_E$-orbit.  Projecting this point onto the coordinates in $E$ gives a point in $\mu_E^{-1}(0)$, so this defines a map from $U\mod K \subset \M$ to $\M_E$.  If the $z_i$ coordinate of the representative of $x$ is zero for some $i\in S$, then $w_i \ne 0$ and we can replace $z_i$ by $w_i$ in this construction.

The local projection map constructed in this way is $T_E$-equivariant.  The required identification of a neighborhood of $x$ in $\M^+$ with $\M^+_E \times \C^{\dim S_E}$ follows immediately.
\end{proof}

\begin{corollary}
For any coloop-free flat $F \in \cF$, the functor $\Phi_F$ preserves constructibility by the coarse stratifications; i.e.\ it restricts to a functor $D^b_{T,\cS}(\M,\kk) \to 
D^b_{T,\cS_F}(\M^F,\kk)$.
\end{corollary}

\begin{proof}
Suppose that $p\in \breve{S}_{E} \subset S_{E'}$, 
so that $E'$ is the smallest coloop-free flat with 
$E \le E'$.  Then we have $\M_E \cong \M_{E'} \times T^*\C^{E\setminus E'}$, and for any 
$\cS$-constructible object $A$, we have 
\[\operatorname{Rest}_E(A) \cong pr_2^* \operatorname{Rest}_{E'}(A)[|E'|-|E|].\] Using Lemma \ref{lem:Morse slice compatibility}, we see that $\operatorname{Rest}_E(\Phi_F(A))$
is isomorphic to a pullback of a sheaf from $\M_{E'}$.  It follows that $\Phi_F(A)$ is  $\cS^F$-constructible.
\end{proof}

\subsection{Morse stalks and $t$-exactness}
For each flat $F$, define a ``Morse stalk" functor $\bar\Phi_F\colon D^b_{T,\cS}(\M,k) \to D^b_{T_F}(p, k)$ by
\[\Morse_F(A):= \Phi_F(A)|_p[-\dim S_F] = \operatorname{Rest}_F\Phi_F(A),\]
where $p$ is any point in $S_F$.  Since $\Phi_F(A)$ is $\cS^F$-constructible, this functor does not depend on the choice of basepoint.

\begin{lemma}\label{lem:Morse is perverse}
For any $F \in \cF$, the functors $\Phi_F$ and $\bar{\Phi}_F$ are t-exact.
\end{lemma}
In particular, $\Morse_F$ induces an exact functor 
$\Perv(\M,k) \to k\md$, which we denote by the same symbol.

\begin{proof}
Step 1: $\Phi_I$ is $t$-exact.  Let $\xi$ be a generic cocharacter of $T$, and let $\M^+$ be the corresponding attracting set.  If $\M_\a \to \M$ is a generic resolution, then $\M^+$ is the image of 
\[\M^+_\a := \{x\in \M_\a \mid \lim_{t\to 0} \xi(t)\cdot x \;\text{exists}\}.\]
Since the action of $T$ is Hamiltonian, $\M^+_\a$ is Lagrangian, so we have 
$\dim (S_\emptyset \cap \M^+) = \half \dim S_\emptyset$, and the same argument applied to $\M^F$ shows that 
$\dim (S_F \cap \M^+) = \half \dim S_F$ for any $F \in \cF$.  

If $p\colon \M^+ \to \M^\xi = \{0\}$ denotes the projection, the contracting lemma gives an isomorphism $\Phi_I(S) \cong p_!(A|_{\M^+})$, so we need to show that the compactly supported hypercohomology
$\H^\bullet_c(A|_{\M^+})$ vanishes outside of degree $0$ for any perverse sheaf $A$.  
But for every stratum $S_F$, the stalks of $A|_{S_F}$ vanish in degrees above $-\dim S_F$,
so $\H_c^\bullet(A|_{S_F})$ vanishes in positive degrees, and so the same is true of $\H_c^\bullet(A|_{\M^+})$.
Lemma \ref{lem:hyperloc} then shows that it vanishes in negative degrees as well.

Step 2: The functor $\bar{\Phi}_F$ is exact for every $F$.  This follows from the first step using Lemma \ref{lem:Morse slice compatibility}.

So we have defined exact functors $\bar\Phi_F\colon \Perv(\M) \to k\md$ for every stratum $S_F$.  This set of functors is complete in the sense that $A \in \Perv(\M)$ is zero if and only if $\bar\Phi_F(A) = 0$ for every $F$. 
Applying this to the perverse cohomology sheaves of an object $A \in D^b_{T,\cS}(\M)$, it follows that $A$ 
is perverse if and only if $\bar\Phi_F(A[i]) = 0$ for all $F$ and all $i \ne 0$.

Step 3: the general case.  Take any $F \in \cF$, and any $A \in \Perv(\M)$. Then by Lemma \ref{lem:Morse transitivity}, for any $E \le F$ and any $i \ne 0$ we have 
\[\bar \Phi^F_E \Phi_F(A[i]) \cong \bar{\Phi}_E(A[i]) = 0,\]
so $\Phi_F(A)$ is perverse.
\end{proof}

By Proposition~\ref{prop:enoughprojs}, the functor $\Morse_F$ is represented by a projective object
(this follows the argument of  \cite[Proposition 2.4]{MiVi}, for instance). We  denote this object by $\Pi_F$ .  Observe that $\Morse_F(\IC_F)=\kk$ and $\Morse_F(\IC_E)=0$ unless $F \le E$. It follows that $\bigoplus_{F \in \F} \Pi_F$ is a projective generator of $\Perv(\M)$. 

\section{Equivariant computations}

\subsection{Localization of equivariant cohomology} 

In order to study the sheaves $\Omega^F$ and the functors $\bar{\Phi}_E$, we will employ equivariant localization.
We begin by looking at the equivariant hypercohomology of the largest resolution sheaf $\Omega = \Omega^{\emptyset}$.  

Fix a resolution $\wt\M = \M_\a$ of $\M$ and put $p = p_\alpha$.  The hypercohomology $\H^\bullet_T(\Omega)$
is naturally identified with $H^{\bullet+2d}_T(\widetilde \M; k)$. 
It is a module over the ring $A = H^\bullet_T(pt; k)$, which is naturally identified with $\Sym(V_k)$, where elements of $V_k := V \otimes_\Z k$ are placed in degree two.

We describe this ring and the restriction homomorphisms to the fixed points using the independence complex $\Delta = \Delta_V$, which is the simplicial complex of all independent sets for the matroid $M(V)$.  The face ring of $\Delta$
with coefficients in $k$ is
\[k[\Delta] := k[\be_i \mid i \in I]/ \langle \be_S \mid S \; \text{is dependent}\rangle.\]
Here we use the shorthand $\be_S := \prod_{i \in S} \be_i$ for any subset $S \subset I$.  We put a grading on this ring where $\deg \be_i = 2$.

The ring $A$ maps naturally to $k[\Delta]$ via the inclusion $V_k \hookrightarrow k^I$, and for each basis $B$ of the matroid the composition
\[A \to k[\Delta] \to k[\Delta]/\langle \be_i \mid i \notin B\rangle\]
is an isomorphism, since $V_k \to k^I \to k^B$ is an isomorphism.  Composing the inverse of this isomorphism with the projection $k[\Delta] \to k[\Delta]/\langle \be_i \mid i \notin B\rangle$ gives a homomorphism 
\[ev_B\colon k[\Delta] \to A.\]
Note that for any two bases $B$, $B'$ we have $ev_B(\be_{B'})\ne 0$ if and only if $B = B'$.

There is a natural bijection $B \mapsto y_B$ between the set of bases of the matroid of $V$ and the set of $T$-fixed points of $\wt\M = \M_\alpha$,
defined as follows.
For a basis $B$, consider the set of points $\{(z_i,w_i)\}_{i \in I} \in \mu^{-1}(0)$ satisfying $z_i = w_i = 0$ when $i\in B$ and exactly one of $z_i, w_i$ is nonzero when $i\notin B$.  It contains $2^{|I\setminus B|}$ orbits of $K$, and exactly one of them is $\alpha$-semistable.  
Let $y_B$ be the image of this orbit in $\wt\M$.

\begin{proposition}[\cite{Ko99}]\label{prop:HT is face ring}
There is a natural isomorphism 
\[H^\bullet_T(\widetilde \M; k) \cong k[\Delta].\]
Under this isomorphism the restriction homomorphism
\[H^\bullet_T(\wt\M;k)\to H^\bullet_T(y_B;k)\cong A\] is identified with $ev_B$.
\end{proposition}

More generally, the $T$-equivariant cohomology of a resolution $\wt{\M^F}$ of a stratum closure $\M^F$ is isomorphic to 
$k[\Delta^F] \otimes_{A^F} A$, where $\Delta^F$ is the independence complex of $V^F$, and $A^F = \Sym(V^F_k)$. 

\begin{remark}\label{rmk:parity}
This result implies that the ordinary cohomology  $H^\bullet(\wt\M; k)$ vanishes in odd degrees, and using Lemma \ref{lem:slice to resolution sheaf} 
it follows that the cohomology sheaves of a resolution sheaf $\T^F$ vanish in odd degrees.  Since $\T^F$ is isomorphic to its Verdier dual, we see that it is a 
parity sheaf in the sense of \cite{JMW14}.  (As noted before, Proposition \ref{prop:parity vanishing} is needed in order for the theory of \cite{JMW14} to apply.)
\end{remark}

\subsection{Localization of Morse groups}\label{sec:localization of Morse groups}
We continue with the notation and assumptions of the previous section.  
In this section we use equivariant localization to study the effect of the hyperbolic restriction $\Phi_F$ on the resolution sheaf $\Omega = \Omega^\emptyset$. 
To avoid shifts in our formulas, we put $\T' = p_*\uk_{\wt\M} = \T[-2d]$.

Fix a choice of generic cocharacter $\xi$ of $T_F$, so that the fixed point set $\M^\xi$ is $\M^F$.  Let $\M^+ \subset \M$ be the corresponding attracting subset, and put $\wt\M^+ = p^{-1}(\M^+)$.  

Since $p\colon \wt\M\to \M$ is proper, we have a natural isomorphism
\[\H^\bullet_T(\Phi_F(\T')) = \H^\bullet_T(f^!g^*\T') \cong \H^\bullet_T(\tilde{f}^!\uk_{\wt\M^+,T}) = H^\bullet_T(\wt\M^+,\wt\M^+\setminus p^{-1}(\M^F); k),\]
where $\tilde{f}\colon p^{-1}(\M^F) \to \wt\M^+$ is the inclusion.  
Restriction induces an isomorphism
\[H^\bullet_T(\wt\M; k) 
 \cong H^\bullet_T(\wt\M^+; k),
\]
since $\wt\M$ and $\wt\M^+$ both deformation retract onto the compact core $p^{-1}(\M^I)$.

Combining these, we get a restriction homomorphism
\[\rho_F\colon\H^\bullet_T(\Phi_F(\T')) \to H^\bullet_T(\wt\M^+; k) \cong H^\bullet_T(\wt\M; k).\]
Our next result describes its image.  

In order to state the result, we must first describe the components of the fixed locus $\wt\M^\xi$.  Consider the cocharacter $\xi$ as a linear function on $V_\R$ and fix a lift $\tilde{\xi} \colon \R^I \to \R$. Consider the subsets $S \subset I$ such that  $\tilde{\xi}$
takes a single value on 
$\wt{H}_S := (V+\alpha) \cap \R^{I\setminus S}$ (in particular, $\wt{H}_S$ is nonempty). 
Let $\wt{F}_1, \dots, \wt{F}_r$ be the smallest subsets 
with this property,
 ordered by increasing $\tilde\xi$-value.  The sets $\wt{F}_i$ only depend on $V$ and $F$; the ordering depends on $\alpha$ and $\xi$ but not on $\tilde{\xi}$.
 
\begin{example}
Take $I = \{1,2,3,4\}$, and let $V \subset \Z^I$ be spanned by $(1,0,1,1)$ and $(0,1,1,1)$.  If we take $F = \{3,4\}$, 
then one choice for $\xi$ is the restriction of $\tilde{\xi} = x_3$ to $V$.  If $\alpha=(0,1,2,3)$, then
we have $\wt{F}_1 = \{4\}$, $\wt{F}_2 = \{3\}$, and $\wt{F}_3 = \{1,2\}$.
\end{example} 

For each $j$, let $\wt{\M}_j$ be the subvariety of
$\wt\M$ defined by the equations $z_i = w_i = 0$ for $i\in \wt{F}_j$.  Then $\wt{\M}_j$ is a  
connected component of the fixed locus $\wt\M^\xi$, and every component appears this way.  
   If $F_j$ is the completion of $\wt F_j$ to a flat of $V$, then $\wt\M_j$ is 
a resolution of the hypertoric variety $\M^{F_j} \subset \M^F$.  
Note that in general $\M^{F_j} \ne \M^F$.

Let $\wt\M^+_j$ be the attracting set for $\wt\M_j$, so $\wt\M^+ = \coprod_j \wt\M^+_j$.  Our choice of ordering ensures that 
$\wt\M^+_{\le s}:=\coprod_{j = 1}^s \wt\M^+_j$ is closed in $\wt\M^+$ for $1 \le s \le r$.

\begin{theorem}\label{thm:Cohomology of intermediate Morse sheaf}
The homomorphism $\rho_F$ is injective, and its image is a free $A$-submodule.  Under the identification $H^\bullet_T(\widetilde \M; k) \cong k[\Delta]$
the image is the direct sum of the monomial ideals
\begin{equation}\label{eq:summand of Phi}
I_j := \be_{\wt{F}_j}k[\Delta] \cong k[\Delta^{F_j}] \otimes_{A^{F_j}} A(-2r_j),
\end{equation}
where $r_j = |\wt{F}_j|$. 
\end{theorem}

We will only need the full generality of this result in the proof of Theorem \ref{thm:extended maps between projectives} below.  The special case $F = I$ is more important, as it gives a combinatorial description of 
our exact functors $\Morse_F$ on resolution sheaves.
In this case, $\{\wt{F}_1, \dots, \wt{F}_r\}$ is just the set of all bases of $V$, and 
$\wt\M^\xi$ consists of isolated points.   Thus we get the following result.

\begin{corollary} \label{cor:localization of Morse group}
The hypercohomology $\H^{i}(\Phi_I(\T))$ vanishes if $i\ne 0$; in particular, the forgetful map $\H^{0}_{T}(\Phi_I(\T)) \to \H^{0}(\Phi_I(\T)) = \Morse_I(\T)$ is an isomorphism.

The map $\rho_I$ followed by restriction to fixed points gives an injective map 
\[\H^{\bullet}_{T}(\Phi_I(\T)) \to H^{\bullet+2d}_T(\wt\M^T; k),\] whose image is the free $A$-submodule generated by $\be_B$ for every basis $B$ of $V$.
\end{corollary}

Applying this to the normal slice to a stratum closure, for any flats $E \le F$ in $\cF$, we get an identification of $\Morse_E(\T^F)$ with the free $k$-module $\cB^F_E$ spanned by $\be_B$ for all bases $B$ of $V^F_E$.

\begin{proof}[Proof of Theorem \ref{thm:Cohomology of intermediate Morse sheaf}]
Each set $\wt{F}_j$ is independent for $V$, but 
for any $j\ne j'$, 
the set $\wt{F}_j \cup \wt{F}_{j'}$ is dependent.  This implies that $I_j$ and $I_{j'}$ have no monomials in common, and so the sum is direct.  

A nonzero monomial in $I_j$ is the product of 
$\be_{\wt{F}_j}$, a nonzero monomial in $k[\Delta^{F_j}]$ and an arbitrary monomial in the $\be_i$, $i\in \wt{F}_j$.  Thus multiplication by $\be_{\wt{F_j}}$ gives an injective map of $A^{F_j}$-modules $k[\Delta^{F_j}] \to I_j(2r_j)$.
The induced map $k[\Delta^{F_j}] \otimes_{A^{F_j}} A \to I_j(2r_j)$ is still injective since $k[\Delta]$ is a free $A$-module.  The isomorphism \eqref{eq:summand of Phi} follows because both sides have the same Hilbert series.

Let $q\colon \wt\M^+ \to \M^\xi$ be the composition of $p|_{\wt\M^+}\colon \wt\M^+\to \M^+$ with the limit map 
$\M^+\to \M^\xi$. We get an isomorphism
$\Phi_F(\T') \cong q_!\uk_{\wt\M^+,T}$ using the contracting lemma (either Lemma \ref{contracting lemma} below, or the more standard version for algebraic torus actions is sufficient). 

Let $c_j\colon \wt\M^+_j \to \wt\M^+$ and $d_j\colon \wt\M_{\le j} \to \wt\M^+$ be the inclusions.  Since the action of $T$ is symplectic and $\codim{\wt\M_j} = 2r_j$, the attracting set $\wt\M^+_j$ is an $\mathbb{A}^{r_j}$-bundle over 
 $\wt\M_j$.  Thus we have 
an isomorphism
\[\H_T^\bullet(q_!(c_j)_!\uk_{\wt\M^+_j,T}) \cong
H^{\bullet-2r_j}_T(\wt\M_j;k).\]
In particular this implies that these groups vanish in odd degrees, so an induction using the triangle
\begin{equation}\label{eq:Morse triangle}
(c_j)_!\uk_{\wt\M^+_j,T} \to (d_j)_!\uk_{\wt\M^+_{\le j},T} \to (d_{j-1})_!\uk_{\wt\M^+_{\le j-1},T} \stackrel{[1]}{\to}
\end{equation}
gives an isomorphism of graded $A$-modules
\begin{equation}\label{eq:cohomology direct sum}
\H^\bullet_T(\Phi_F\T') \cong \bigoplus_{j=1}^r H^{\bullet-2r_j}_T(\wt\M_j;k).
\end{equation}
In particular $\H^\bullet_T(\Phi_F\T')$ is a free $A$-module.

Define 
$\wt\M^> := \wt\M^+ \setminus p^{-1}(\M^\xi)$.
It has no $T$-fixed points, so its equivariant cohomology $H^\bullet_T(\wt\M^{>};k)$ is torsion as an $A$-module.
It follows that the long exact sequence
\[\to\H^\bullet_T(\Phi_F\T') \to H^\bullet_T(\wt\M^+; k) \to H^\bullet_T(\wt\M^{>};k) \to\]
breaks into short exact sequences.   

So to prove the theorem it will be enough to show that kernel of the 
composition 
\[k[\Delta] \cong H^\bullet_T(\wt\M^+; k) \to H^\bullet_T(\wt\M^>; k)\]
is $\bigoplus_j \be_{\wt{F}_j}k[\Delta]$.
To see that $\be_{\wt{F}_j}$ is in the kernel, recall that the isomorphism of 
Proposition \ref{prop:HT is face ring} is induced by the Kirwan homomorphism
\[k[\be_i\mid i\in I] \cong H^\bullet_{(\C^*)^I}(T^*\C^I; k) \to H^\bullet_{(\C^*)^I}(\mu_K^{-1}(0)^{ss}; k)\stackrel{\sim}{\to} H^\bullet_T(\wt\M; k),\]
where $\mu_K^{-1}(0)^{ss}$ is the set of $\alpha$-semistable points in $\mu_K^{-1}(0)$.  It is easy to check that $(\C^*)^{\wt{F}_j}$ acts freely on the preimage of $\wt\M^> \subset  
\wt\M \setminus \wt\M_j$ in $\mu_K^{-1}(0)$, and so $\be_{\wt{F}_j}$ restricts to zero there. 

For the other inclusion, 
just note 
that the isomorphisms \eqref{eq:summand of Phi} and \eqref{eq:cohomology direct sum} imply that $\bigoplus_j I_j$ and 
$\H^\bullet_T(\Phi_F\T')$ have the same Hilbert series.
\end{proof}

\subsection{Faithfulness of hypercohomology and Morse cohomology}

In this section and for the remainder of the paper we
will assume that $V$ has no coloops, or in other words that $I\in \cF$.  This is purely for convenience; all of the statements hold in the general case with suitable modifications.

\begin{proposition}\label{prop:Morse is faithful}
The global equivariant hypercohomology functor $\H^\bullet_T$ and the functor $\Morse_I$ are faithful on resolution sheaves. In other words, for any $E,F \in \F$, we have injections
\begin{align*}
 \Hom(\T^E,\T^F) & \to \Hom_A(\H^\bullet_T(\T^E),\H^\bullet_T(\T^F)) \\ & = \Hom_A(k[\Delta^E]\otimes_{A^E} A,k[\Delta^F]\otimes_{A^F} A)
\end{align*}
and
\[ \Hom(\T^E,\T^F) \to \Hom_\kk (\Morse_I(\T^E),\Morse_I(\T^F))= \Hom_\kk(\B^E_I,\B^F_I).\]
\end{proposition}

As a result we will be able to describe the ring $\End(\bigoplus_F \Omega^F)$ by studying its action on $\bigoplus_{E\le F} \cB^F_E$. 

The proof will be based on the following result.
Let $X$ be a $T$-variety endowed with a $T$-invariant stratification $\cS$.  For each $S \in \cS$, let $j_S \colon S\to X$ be the inclusion. 

\begin{lemma}\label{lem:faithful hypercohomology}
 Suppose that for each $S \in \cS$ the equivariant cohomology $H^\bullet_T(S; k))$ vanishes in odd degrees, and that  
 objects $B,C \in D^b_{T,\cS}(X; k)$ satisfy the following for every $S\in \cS$:
\begin{enumerate}
\item[(a)] the cohomology sheaves of $j^*_S(B)$ are constant and vanish in odd degrees,
\item[(b)] the restriction $\H_T^\bullet(B) \to \H_T^\bullet(j_S^*B)$ is surjective,
\item[(c)] the natural map $\H_T^\bullet(j_S^!C) \to \H^\bullet_T(j_S^*C)$ is injective.
\end{enumerate}
Then taking hypercohomology induces an injection
\[\Hom(B, C) \to \Hom_{H^\bullet_T(pt)\mathrm{-mod}}(\H_T^\bullet(B),\H_T^\bullet(C)).\]
\end{lemma}

\begin{proof}
Take a map $\psi\colon B \to C$ inducing the zero map 
$\H^\bullet_T(B) \to \H^\bullet_T(C)$.  We prove by induction on $|\cS|$ that $\psi = 0$.
If $|\cS| = 1$, so $X$ is a single stratum, then this follows from 
\cite[Lemma 5.5(c)]{BP09} (this is where we need the hypothesis that $H^\bullet_T(S; k))$ vanishes in odd degrees).

If $|cS| > 1$, let $S \in \cS$ be closed in $X$, and set $U = X \setminus S$.  
Then $B|_U$ and $C|_U$ also satisfy (a), (b) and (c), so we have $\psi|_U = 0$.  This implies that
$\psi$ factors as 
\[B \to j_{S*}j^*_S B \to j_{S!}j_S^! C \to C\]
where the first and last maps are the adjunctions and
the middle map is the pushforward of a map
\[\psi_S \in \Hom(j_S^*B, j_S^!C) = \H^0_T(j_S^!\RHom(B,C)).\]
Composing $\psi_S$ with the natural map $j_S^!C \to j_S^*C$ gives the restriction map $j_S^*(\psi)\colon j_S^*B \to j_S^*C$.   But our assumption (b) and the fact that 
$\H_T^\bullet(\psi) = 0$
imply that $j_S^*(\psi)$ induces the zero map
$\H^\bullet_T(j_S^*B) \to \H^\bullet_T(j_S^*C)$, and so
(c) implies that $\psi_S$ also gives the zero map on hypercohomology.  Another application of the one-stratum case of the Lemma shows that $\psi_S = 0$, and so $\psi = 0$, as desired. 
\end{proof}

\begin{proof}[Proof of Proposition \ref{prop:Morse is faithful}]
If a map $\phi\colon \T^E \to \T^F$ induces the zero map $\Morse_I(\T^E) \to \Morse_I(\T^F)$,
then the equivariant map 
$\H^\bullet_T(\Phi_I(\T^E)) \to \H^\bullet_T(\Phi_I(\T^F))$ is also zero.
By Corollary \ref{cor:localization of Morse group}, the
map $\rho_I\colon \H^\bullet_T(\Phi_I(\T^F)) \to \H^\bullet_T(\T^F)$ is a map of free graded $H^\bullet_T(pt)$-modules which becomes an isomorphism when tensored with the quotient field.  It follows that $\phi$ induces the zero map
$\H^\bullet_T(\T^E) \to \H^\bullet_T(\T^F)$.  Thus we only have to show that 
 $\H_T$ is faithful.  We will prove this by showing that 
 the hypotheses of Lemma \ref{lem:faithful hypercohomology} hold for 
 the resolution sheaves $\T^F$ and the fine stratification $\{\breve S_F\}$.  It is enough to consider the sheaf $\T = \T^\emptyset$.

Property (a) holds for $\Omega$ by
Corollary \ref{cor:omega is constructible} and Remark \ref{rmk:parity}.

Next we show that  $\T$ satisfies (b), 
or in other words that the restriction
\[H^\bullet_T(\M_\a) \to H^\bullet_T(p_\a^{-1}(\breve S_F))\]
is surjective, where we fix a choice of resolution $p_\a\colon \M_\a \to \M$.  In fact, we claim that the composition of this map with
the Kirwan
homomorphism
\[\kappa\colon k[\be_i\mid i\in I] = H^\bullet_{(\C^*)^I}(T^*\C^I) \to H^\bullet_T(\M_\a)\]
is surjective. 

By \cite[2.5]{PW}, any fiber $p_\a^{-1}(x)$ is isomorphic as a $T_F$-variety to the central fiber of 
the hypertoric resolution $\M_{F,\bar \a} \to \M_F$, which in turn is homotopy equivalent to $\M_{F,\bar \a}$.  Under this isomorphism, the Kirwan homorphism 
\[\kappa_F\colon k[\be_i\mid i\in F] \to H^\bullet_{T_F}(\M_{F,\bar{\a}})\]
is sent to the restriction of $\kappa$ to $k[\be_i\mid i\in F]$ followed by the pullback to $p_\a^{-1}(x)$.

We have a spectral sequence
\[E_2^{pq} = H^p_{T/T_F}(\breve S_F) \otimes_k H^q_{T_F}(p_\a^{-1}(x)) \Rightarrow H^{p+q}_T(p_\a^{-1}(\breve S_F)).\]
The argument of \cite[Proposition 5.22]{BP09} shows that under the Kirwan homomorphism $k[\be_i\mid i\notin F]$ surjects onto the first term and we have seen that
$k[\be_i\mid i \in F]$ surjects onto the second, so $k[\be_i\mid i\in I]$ surjects onto the $E_2$ page.  Thus the spectral sequence degenerates for parity reasons, and the claim follows.

To see that (c) holds, we first consider the case of the 
smallest stratum $S_I$.  In that case, the contracting lemma identifies the homomorphism $\H_T^\bullet(j_I^!\T)
\to \H_T^\bullet(j_I^*\T)$ with 
the natural homomorphism $H^\bullet_{T,c}(\M_\a) \to \H^\bullet_T(\M_\a)$.  To see that it is injective, it 
is enough to show that the localization map $H^\bullet_{T,c}(\M_\a) \to H^\bullet_{T}(\M_\a^T)$ is injective.  This holds by equivariant formality: since $H^\bullet_c(\M_\a)$ vanishes in odd degrees, $H^\bullet_{T,c}(\M_\a)$ is a free $A$-module.

To see that (c) holds for arbitrary $F$, we use \cite[Lemma 5.5(a)]{BP09}.  This result implies that
$\H_T^\bullet(j_S^!\T) \to \H_T^\bullet(j_S^*\T)$ is a
map of free $R := H^\bullet_T(S_F)$-modules, and  tensoring over $R$ with $A_F = H^\bullet_T(Tx)$ for any $x \in S_F$ gives the homomorphism in hypercohomology induced by $(j^!_S\T)|_{Tx} \to (j^*_S\T)|_{Tx}$.  So it is enough to see that this second map is injective.  But taking a normal slice, this is identified with the same map for the smallest stratum in $\M_F$, so we are reduced to the previous case.
\end{proof}

\subsection{Poincar\'e-Verdier dual pairing}\label{sec:pairing}
For any flats $E\le F$, we have a natural nonsingular pairing on $\Morse_E(\Omega^F)$ induced by Poincar\'e-Verdier duality.  For simplicity, let us assume that $E = I$; the general case proceeds by restricting to a transverse slice.  Fix a generic cocharacter $\xi$ of $T$, so that $\Morse_I$ is the hypercohomology of the hyperbolic restriction $\Phi_\xi$.
Using the identifications $\Phi_\xi \cong \Phi_{-\xi}$, 
$\D\Omega \cong \Omega$ and Lemma~\ref{lem:hyperloc}, we have 
\begin{align*}
\Morse_I(\T^F) & = \mathbb{H}^0(\Phi_\xi(\T^F)) \cong \mathbb{H}^0(\Phi_\xi(\D\T^F))  \\ & \cong \mathbb{H}^0(\D\Phi_{-\xi}(\T^F))
\cong \mathbb{H}^0(\D\Phi_{\xi}(\T^F)) = \Morse_I(\T^F)^*.
\end{align*}
This gives the required pairing $\Morse_I(\T^F) \times \Morse_I(\T^F) \to k$.

We wish to describe this pairing in terms of the localized basis provided by Corollary \ref{cor:localization of Morse group}.  To do this,
it will be useful to have a slightly different description of the Morse functor.  Let $\M^+$ be the attracting set of the cocharacter $\xi$, so $\Phi_I = f^!g^*$ where $f\colon \M^\xi \to \M^+_\xi$ and $g\colon \M^+_\xi \to\M$ are the inclusions.  
Then \cite{Br03} implies that this functor is naturally isomorphic to $f_-^*g_-^!$, where $f_-\colon \M^\xi \to \M^+_{-\xi}$ and 
$g_-\colon \M^+_{-\xi} \to \M$ are the inclusions.  Then we have an isomorphism
\begin{align*}
\Morse_I(\Omega) \cong \H^0(f_-^*g_-^!\Omega) & \cong \H^0(g_-^!\Omega)  \cong H^{2d}(\wt\M, \wt\M \setminus p^{-1}(\M^+_{-\xi}); k) \\
& \cong H^{2d}_T(\wt\M, \wt\M \setminus p^{-1}(\M^+_{-\xi}); k).
\end{align*}

\begin{proposition}
Using the basis $\be_B$ for $\Morse_E(\T^F)$ provided by Corollary \ref{cor:localization of Morse group}, the pairing is given by
\[\langle \be_B,\be_{B'}\rangle = (-1)^d\delta_{B,B'},\]
\end{proposition}

\begin{proof}
The pairing can be computed as the cup product
\[H^{2d}_T(\wt\M, \wt\M \setminus p^{-1}(\M^+_\xi)) \times H^{2d}_T(\wt\M, \wt\M\setminus p^{-1}(\M^+_{-\xi})) \to H^{4d}_T(\wt\M, \wt\M \setminus p^{-1}(0))\]
(taking coefficients in $k$ and using the isomorphism between $\Morse_\xi$ and $\Morse_{-\xi}$ given by Proposition \ref{prop:vanishing cycles iso}), followed by the integration for compactly supported cohomology. The integral can be computed using 
the Atiyah-Bott-Berline-Vergne localization theorem:
\begin{equation*}
\int_{[\wt\M]} \alpha = \sum_{y=y_B\in \wt\M^T}  \frac{\alpha_y}{e_T(N_y(\wt\M))}
\end{equation*}
where $\alpha \in H^{4d}_T(\wt\M, \wt\M \setminus p^{-1}(0); k)$ and $e_T(N_y(\wt\M))$ is the equivariant Euler class of the tangent space to $\wt\M$ at $y$.  Since the tangent space at a fixed point $y_B$ is symplectic, the weights appear in pairs $\{\a,-\a\}$, and the restriction of $\be_B$ to this point will contain one weight from each such pair.  Thus
 we have \[e_T(N_{y_B}(\wt\M)) = \prod_{i\in B} [- ev_B(\be_i)^2] = (-1)^dev_B(\be_B)^2.\]
The proposition follows.
\end{proof}

\section{Homomorphisms between resolution sheaves}

The main result of this section is an explicit description of the algebra $R = \End(\oplus_{F\in\F} \T^F)$; we show that it is isomorphic to the algebra generated by a certain set of concrete linear operators on the $k$-vector space
$\cB := \bigoplus_{E,F\in \cF} \cB^F_E$, where
$\cB^F_E = \Morse_E\T^F$.  This is possible because Proposition \ref{prop:Morse is faithful} says that  $R$ acts effectively on this space.  (Indeed it even acts effectively on the smaller space $\bigoplus_F \cB^F_I$, but it will be important for us to use the whole space $\cB$.)

\subsection{Homomorphisms to and from the point sheaf}

We first study a part of our algebra that is easier get a hold of, namely the subspaces $\Hom(\T^F,\T^I)$ and $\Hom(\T^I,\T^F)$.  Since these sheaves have support contained in $\M^F \subset \M$, we can assume without loss of generality that $\M^F = \M$, or in other words $F = \emptyset$.  We put $\T = \T^\emptyset$ for the resolution sheaf with maximal support.

Homomorphisms to and from $\T^I$ can be computed by adjunction:
\[\Hom(\T,\T^I) = \H^0(j_I^*\T)^*,\; \Hom(\T^I,\T) = \H^0(j_I^!\T).\]
Furthermore, fixing an isomorphism of $\T$ with its Verdier dual, we have an isomorphism between these two vector spaces.  

By Proposition \ref{prop:Morse is faithful},  these Hom spaces inject into the spaces of linear maps between $\cB^\emptyset_I = \Morse_I(\T)$ and $\cB^I_I = \Morse_I(\T^I) = k$.
We let $\cU^\emptyset_I$ denote the subspace of 
$\cB^\emptyset_I$ obtained as the image of $\Hom(\T^I, \T)$.  The image of $\Hom(\T,\T^I)$ in $(\cB^\emptyset_I)^*$ is then given by pairing with elements of $\cU^\emptyset_I$, using the pairing 
constructed in Section \ref{sec:pairing}.

It follows that the orthogonal space $(\cU^\emptyset_I)^\bot$ under the pairing is the kernel of the natural homomorphism
\[\Morse_I\T \to \H^0(j_I^*\T) \cong  \H^0(\T) = H^{2d}(\widetilde\M;k)\]
where the middle isomorphism comes from the contracting lemma.
This map factors as 
\[\Morse_I\T \hookrightarrow H^{2d}_T(\widetilde\M;k) \twoheadrightarrow H^{2d}(\widetilde\M;k);\]
the first map is an injection by Corollary \ref{cor:localization of Morse group} and the second map is a surjection by equivariant formality. The kernel of the second map is the degree $2d$ component of the ideal generated by $H^2_T(pt) = A_2$.

\subsection{Edge classes}

Using this, we can produce the following obvious elements in $(\cU^\emptyset_I)^\bot$, which we call ``edge classes".   An \emph{edge} of $V$ is an independent set with $d-1$ elements; every edge can be obtained by removing a single element from a basis.  For an edge $X$, the subset $H_X \subset V$ corresponds to a rank $1$ flat $E$; in other words, we have $H_X = H_E$.  Let $v$ be a generator of $H_E$ and define
\[\alpha_X := v\cdot \be_X \in k[\Delta],\]
where we consider $v$ as an element of $k[\Delta]_2$ via the map $A \to k[\Delta]$.   

If we write $v = \sum_{i\in I} a_i\be_i$, then 
$a_i \ne 0$ if and only if $B_i = X \cup \{i\}$ is a basis.  Therefore $\alpha_X$ is squarefree, so it lies in $\cB^\emptyset_I$.  Since $v$ is in $A_{>0}$, the class $\a_X$ lies in $(\cU^\emptyset_I)^\bot$.

\begin{proposition}\label{prop:edge classes generate}
The classes $\a_X$ span $(\cU^\emptyset_I)^\bot$.
\end{proposition}
\begin{proof}
Fix a total ordering of the index set $I$.   The induced lexicographic order $\prec$ on the set of bases gives a shelling of the matroid complex $\Delta$ \cite{Bj92}. Let $\bar{\cB} \subset \cB^\emptyset_I$ be the subspace spanned by the monomials $\be_B$ 
such that for every $i \in B$, there exists $j < i$ such that 
$B' = (B \setminus \{i\}) \cup \{j\}$ is a basis (and so $B' \prec B$).  In matroid terminology, this says that every element of $B$ is ``internally passive." 

These bases $B$ are exactly the maximal simplices of $\Delta$ all of whose boundary faces appear earlier in the shelling.
This implies that $\bar{\cB}$ 
maps isomorphically to $k[\Delta]/A_{>0}k[\Delta] = H^{2d}(\wt\M; k)$.
Thus it will be enough to show that the edge 
classes $\alpha_X$ together with $\bar{\cB}$
span $\cB^\emptyset_I$.

To see this, note that if $\be_B \notin \bar{\cB}$, then there is some $i \in B$ such that no smaller basis is obtained by replacing $i$ with another element, so $\alpha_{B \setminus i}$ is a linear combination of $\{\be_{B'} \mid B \preceq B'\}$, and
the coefficient of $\be_B$ is non-zero.
\end{proof}

\subsection{Extending homomorphisms from slices}\label{sec:extending homs}

Using these subspaces, we can reconstruct the whole endomorphism algebra of $\bigoplus_F \T^F$ in the following way.

Consider flats $E \leq F \in \F$. 
Lemma \ref{lem:slice to resolution sheaf} says that 
the restriction of $\T^F$ to a transverse slice 
to $\M$ at the stratum $S_E$, is (up to a shift) isomorphic to the resolution sheaf $\T^F_E$.
 We will construct extension maps 
\[\Xi  \colon \Hom(\T^E_E,\T_E^F) \to \Hom(\T^E,\T^F),\] 
\[\Xi' \colon \Hom(\T_E^F,\T^E_E) \to \Hom(\T^F,\T^E)\]  
which are right inverses to the maps obtained by restriction.

Without loss of generality, we may assume that $F = \emptyset$, since the general case follows by pushing forward along the inclusion $\M^F \to \M$.

The construction of our maps arises from factoring a resolution of the singular variety $\M$ through a carefully chosen partial resolution, which we now describe.

The set of generic GIT parameters in $X^*(K) = \Z^I/V \subset \R^I/V_\R$ is the intersection of the lattice 
$\Z^I/V$ with the complement of a hyperplane arrangement known as the discriminantal arrangement \cite{BaBr97}.
Consider the space of characters which are trivial on the subtorus $K_E = K \cap (\C^*)^E$, which can be identified with the sublattice $\Z^{I\setminus E}/V^E \subset \Z^I/V$.  Unless $E = \emptyset$ (in which case the construction is trivial), there will be no generic characters in this sublattice.  But we can choose 
a character $\b \in \Z^{I\setminus E}/V^E$ which is
generic when considered as a character of $K^E := K/K_E \subset (\C^*)^{I\setminus E}$.
This means that
every element of $\b + V^E_\R$ has at least 
$|I \setminus E|-\rank V^E$ non-zero entries indexed by elements of $I\setminus E$.

We then choose a generic character $\a \in \Z^I/V$ so that the chamber of the discriminantal arrangement in $\R^I/V_\R$ which contains $\a$ contains $\b$ in its closure.

To see an example of the choice of $\a$ and $\b$, let
$V \subset \Z^4 = \Z^{\{1,2,3,4\}}$ be spanned by $(1,0,1,1)$ and $(0,1,1,1)$.  If we take the flat $E = \{3,4\}$, then one possible choice of $\b$ is $(1,0,0,0)$.
A compatible choice of generic $\a$ is $(1,0,0,-2)$, but
$(1,0,0,2)$ is not compatible, although it is generic (in fact the closure of its chamber does not contain any suitable $\b$).

  
By \cite[3.2]{Th96}, there is a map $q\colon \M_\a \to \M_\b$ such that $p^{}_\b q = p_\a$.
It is proper and semi-small.


\begin{lemma}\label{lem:iterated reduction}
 These spaces and maps have the following properties:
\begin{enumerate}
\item[(a)] $(p_\b)^{-1}(\M^E)$ can be identified with the resolution  $\M_\b^E$ of $\M^E$
defined by the GIT parameter $\b \in \Z^{I \setminus E}/V^E$,
\item[(b)]  $q^{-1}(\M^E_\b) \to \M^E_\b$ is a fiber bundle whose fibers are $T_E$-equivariantly isomorphic to the fiber over $o$ of the resolution $\M_{E,\bar{\a}} \to \M_E$,  where $\bar{\a} = \a|_{K_E}$ and $o\in \M_E$ is the cone point, and
\item[(c)] there is a neighborhood $U$ of the stratum $S_E \subset \M$ such that restricting $p_\b$ gives a homeomorphism $(p_\b)^{-1}(U) \to U$. 
\end{enumerate}
\end{lemma}
\begin{proof}
The key idea is to
 view $\M_\a$ and
 $\M_\b$ as iterated symplectic reductions. The 
reduction of $T^*\C^I$ by $K_E = K\cap (\C^*)^E$ with
GIT parameters $0$ and $\bar{\a}$ gives the 
resolution 
\[p_E\colon T^*\C^{I\setminus E} \times \M_{E,\bar{\a}} \to T^*\C^{I\setminus E} \times \M_E \]
of the affine hypertoric variety $T^*\C^{I\setminus E} \times \M_E$. 
The map $p_E$ is equivariant for the residual action of  $(\C^*)^I/K_E$; we will reduce both spaces further by the action of the subtorus $K^E = K/K^E$. 

The moment map for $T^*\C^I$ induces moment maps for $K^E$ making the following diagram commute:
\[\xymatrix{
T^*\C^{I\setminus E} \times \M_{E,\bar{\a}} \ar[r]^{p_E}\ar[dr]_{\bar{\mu}_E} & T^*\C^{I\setminus E} \times \M_E \ar[d]^{\bar{\mu}} \\
& (\mathfrak{k}^E)^*\\
}
\]
Since $\b$ is trivial on $K_E$, it gives a $K^E$-equivariant structure on the trivial line bundle over 
$T^*\C^{I\setminus E} \times \M_E$.  On the other hand,
the action of $K$ on the trivial line bundle over $T^*\C^I$ determined by $\alpha$ decends to an action of  $K^E$ on the induced line bundle on $T^*\C^{I\setminus E} \times \M_{E,\bar\a}$.  We have corresponding subsets of semistable points, which satisfy
\begin{equation}\label{eqn:iterated reduction}
\bar{\mu}_E^{-1}(0)^{ss} = p_E^{-1}(\bar{\mu}^{-1}(0)^{ss}).
\end{equation}

The map $q\colon \M_\a \to \M_\b$ can then be identified with the map \[\bar{\mu}_E^{-1}(0)^{ss}/(K/K_E) \to \bar{\mu}^{-1}(0)^{ss}/(K/K_E)\]
induced by $p_E$.  

Using this identification, the variety $(p_\b)^{-1}(\M^E)$
is obtained as the image of $T^*\C^{I\setminus E} \times 0 \subset T^*\C^I$ under the two reductions which produce $\M_\b$.  

Reducing this subvariety by $K_E$ with parameter $\b|_{K_E} = 0$ gives $T^*\C^{I\setminus E} \times \{o\}$.  Restricting $\bar{\mu}$ to this 
subspace gives the moment map for the $K^E$-action on
$T^*\C^{I\setminus E}$, so we have
an isomorphism $(p_\b)^{-1}(\M^E)\cong \M^E_\b$, giving (a).  

The statement (b) now follows using \eqref{eqn:iterated reduction} and our commutative diagram.  The map $p_\b^{-1}(S_E) \to S_E$ is thus an isomorphism, so in a neighborhood of $S_E$ the fibers are zero-dimensional.  Since the fibers are connected, $p_\b$ is a bijection over a neighborhood of $S_E$, giving (c).
\end{proof}


Now we can use our partial resolution to define the maps $\Xi$, $\Xi'$.  Pushing the sheaves $\wt\T := q_* \uk_{\M_\a,T}[\dim \M_\a]$ and $\wt\T^E := \uk_{\M^E_\b,T} [\dim \M^E_\b]$ forward from $\M_\b$ to $\M$ gives $\T^\emptyset = \T^F$ and $\T^E$, respectively.

Since $\wt\T$ and $\wt\T^E$ are locally constant on the smooth and simply-connected subvariety $\M^E_\b \subset \M_\b$, 
$R\mathcal{H}om(\wt\T, \wt\T^E)$
and $R\mathcal{H}om(\wt\T^E, \wt\T)$ are constant local systems on $\M^E_\b$ extended by $0$.  As a result, restricting to a normal slice to 
 $\M^E_\b$ at a point of $p_\b^{-1}(S_E)$ and identifying the slice with $\M_E$, we get isomorphisms:
\[ \Hom_{\M_\b} (\wt\T^E, \wt\T) \stackrel{r}{\longrightarrow} \Hom_{\M_E}(\T^E_E,\T^\emptyset_E),\]
\[ \Hom_{\M_\b} (\wt\T,\wt\T^E) \stackrel{r'}{\longrightarrow} \Hom_{\M_E}(\T^\emptyset_E,\T^E_E).\]
Then we define $\Xi= p_{\b*} \circ r^{-1}$ and $\Xi' = p_{\b*} \circ (r')^{-1}$.

\begin{remark}
To be precise, here we are implicitly using Proposition \ref{prop:forgetting equivariance} to pass between equivariant and non-equivariant perverse sheaves.
\end{remark}


\begin{remark}
In this construction, the fact that the resolution sheaves are independent of the choice of resolution $\M_\a$ is essential, since not all $\a$ will lie in a chamber adjacent to an appropriate $\b$.
\end{remark}

\subsection{Relation with convolution algebra}
In this section, which is not needed for the rest of the paper, we explain how the extension maps $\Xi$ and $\Xi'$ can be obtained in the well-known presentation due to Ginzburg \cite{CG} of 
endomorphisms of pushforward sheaves as Borel-Moore homology of a convolution variety.  We will consider $\Xi$ only, but the same argument applies to $\Xi'$ with only minor modifications.

We continue with the notation of Lemma \ref{lem:iterated reduction}.  We consider the commutative diagram
\[\xymatrix{Z' \ar[r]^\iota \ar[dr] & Z \ar[d]^{\pi_1}\ar[rr]^{\pi_2} & & \M^E_\b \ar[dl]_j\ar[d]^{p^E_\b}\\
& \M_\a \ar[r]_q & \M_\b \ar[r]_{p_\b} & \M \\
}\]
where $j$ is the inclusion resulting from the identification of $(p_\b)^{-1}(\M^E)$ with $\M^E_\b$,
$p^E_\b$ is the projection $\M^E_\b\to\M^E$ followed by the inclusion into $\M$, and $Z$ and $Z'$ are the convolution varieties
\[Z = \M_\a \times_\M \M^E_\b,\;\; 
Z' = \M_\a \times_{\M_\b} \M^E_\b.\]
The maps $\pi_1$ and $\pi_2$ are the natural projections, and  $\iota$ is the induced inclusion.

Both $Z$ and $Z'$ are pure of dimension $(d+d^E)/2$.  
In fact, $Z$ is Lagrangian in $\M_\a \times \M^E_\b$, and 
$Z'$ is the union of the irreducible components of $Z$ which surject onto $\M^E_\b$.

Then using the properness of $p^E_\b$ and the fact that the rectangle is Cartesian, we have natural isomorphisms
\begin{align*}
\Hom(\T^E, \T^F) & \cong \Hom(p^E_{\b!}\uk_{\M^E_\b}, p_{\a*}\uk_{\M_\a}[d-d^E])\\
& \cong \Hom(\uk_{\M^E_\b}, (p^E_\b)^!p_{\a*}\uk_{\M_\a}[d-d^E])\\
& \cong \Hom(\uk_{\M^E_\b}, \pi_{2*}\pi_1^! \uk_{\M_\a}[d-d^E])\\
& \cong \H^0(\pi_1^! \uk_{\M_\a}[d-d^E]) \cong \H^{-d-d^E}(\D\uk_Z)\\
& \cong H^{BM}_{d+d^E}(Z),
\end{align*}
where the penultimate isomorphism comes from $\pi_1^!\uk_{\M_\a} \cong (\D\uk_Z)[-2d]$.

Applying the same argument to the resolutions $q$ and $j$ gives an isomorphism
\[\Hom(\wt\T^E, \wt\T) \cong H^{BM}_{d+d^E}(Z'),\]
and under these isomorphisms pushing forward by $p_\b$ 
is identified with the pushforward $H^{BM}_{d+d^E}(Z') \to H^{BM}_{d+d^E}(Z)$.  In the same way we get an isomorphism
\[\Hom(\T^E_E, \T^F_E) \cong H^{BM}_{d_E}(Y),\]
where $Y$ is the fiber over $o$ of 
$\M_{E,\bar{\a}} \to \M_E$.

Since $j$ is an inclusion, we have $Z' \cong q^{-1}(\M^E_\b)$, so by Lemma \ref{lem:iterated reduction}, it is a fiber bundle over $\M^E_\b \subset \M_\b$ with fiber $Y$.  
Since $\M^E_\b$ is simply connected, we get a bijection
$Z_j \mapsto Y_j$ between irreducible components of $Z'$ and irreducible components of $Y$, given by intersecting with $q^{-1}(x) \cong Y$ for some $x \in \M^E_\b$.

Putting this together, we have proved the following.
\begin{proposition}
Applying $\Xi$ to the map $\T^E_E \to \T^F_E$ corresponding to $[Y_j] \in H^{BM}_{d_E}(Y)$ gives the map
$\T^E\to\T^F$ corresponding to $[Z_j] \in H^{BM}_{d+d^E}(Z)$.
\end{proposition}

This description of $\Xi$ is in some ways simpler than the definition we have used; the main advantage of our approach is that it leads to a formula in terms of equivariant localization, given in Theorem \ref{thm:extended U action} below.

\subsection{Cellular basis for $R$}

\begin{theorem}\label{thm:cellular homs}
For any two flats $E,F \in \F$, the map
\begin{equation}\label{eq:cellular homs}
\bigoplus_{D\leq E \wedge F} \Hom(\T^D_D,\T^F_D) \otimes \Hom(\T^E_D,\T^D_D) \longrightarrow \Hom(\T^E,\T^F) 
\end{equation}
given by $f\otimes g \mapsto \Xi(f) \circ \Xi'(g)$ is an isomorphism.
\end{theorem}

\begin{proof}
This follows from a triangularity argument, based on the following easy facts:
\begin{enumerate}
\item[(a)] a homomorphism in the image of the term  $\Hom(\T^D_D,\T^F_D) \otimes \Hom(\T^E_D,\T^D_D)$ is zero outside of $\M^D = \overline{S_D}$,
\item[(b)] the map
\[ \Hom(\T^D_D,\T^F_D) \otimes \Hom(\T^E_D,\T^D_D) \to \Hom(\T^E_D,\T^F_D)\]
obtained by restricting the $D$th component of \eqref{eq:cellular homs} to a transverse slice to $S_D$ (which we identify with $\M_D$) is just composition of morphisms, and
\item[(c)] the map in (b) is an injection, whose image is the subspace
$\Hom_o(\T^E_D,\T^F_D) \subset \Hom(\T^E_D,\T^F_D)$ of morphisms supported on the smallest stratum $\{o\} = S_D\cap \M_D$ in $\M_D$.  This holds because, if we let $j\colon \{o\} \to \M_D$ be the inclusion, we have natural isomorphisms
\[\Hom_o(\T^E_D,\T^F_D) \cong \H^0(j^!\RHom(\T^E_D, \T^F_D)) \cong \H^0(\RHom(j^*\T^E_D, j^!\T^F_D)).\]
\end{enumerate}

To show that \eqref{eq:cellular homs} is injective, take 
an element in the kernel.  If the terms corresponding to all flats $D' > D$ are zero, then facts (a) and (c) imply that the term corresponding to $D$ must also vanish, and so all of the terms vanish.

To see that it is surjective, suppose that a homomorphism $\phi \in \Hom(\T^E,\T^F)$ is not in the image and has the smallest number of strata in its support among homomorphisms with this property.  Let $D$ be a maximal stratum in the support.  Then property (c) implies that we can choose an element in 
$\Hom(\T^D_D,\T^F_D) \otimes \Hom(\T^E_D,\T^D_D)$ whose image agrees with $\phi$ on a slice to $S_D$.  Subtracting this from $\phi$ gives a morphism with smaller support which is not in the image, a contradiction.
\end{proof}

\begin{corollary}\label{cor:generating set}
The algebra $R = \End(\bigoplus_F \T^F)$ is generated by all  maps $\Xi(f)$, $\Xi'(g)$ where $E\le F$, $f \in \Hom_{\M_E}(\T^E_E, \T^F_E)$, 
$g \in \Hom_{\M_E}(\T^F_E,\T^E_E)$.
\end{corollary}

\begin{corollary}\label{cor:extending from slice}
For any flats $E, F_1, F_2\in \cF$ with $E \le F_i$, $i=1,2$, the map 
\[\Hom_\M(\T^{F_1},\T^{F_2}) 
\to \Hom_\M(\T^{F_1}_E,\T^{F_2}_E)\]
obtained by restricting to a normal slice to $S_E$ is a surjection.
\end{corollary}

\subsection{Combinatorial description of $R$}

We now give a combinatorial description of the algebra $R = \End(\bigoplus_F \T^F)$.  By Proposition \ref{prop:Morse is faithful}, the action of $R$ on $\cB$ is faithful, so Corollary \ref{cor:generating set} gives an isomorphism of $R$ with the subalgebra of $\End_k(\cB)$ generated by the actions of maps of the form $\Xi(f)$, $\Xi'(g)$.  We therefore need to compute how these elements act.

Define an associative product $\ast$ on $\cB$ as follows.  If $\be_B \in \cB^F_E$ and $\be_{B'} \in \cB^{F'}_{E'}$, then 
\[\be_B \ast \be_{B'} =
\begin{cases} 
\be_{B\cup B'} \in \cB^F_{E'} & \mbox{if } E = F'\\
0 & \mbox{otherwise.}
\end{cases}\]
This makes sense because if $B$ is a basis of $V^F_E$ and $B'$ is a basis of $V^E_{E'}$, then $B \cup B'$ is a basis of $V^F_{E'}$.
We can also define adjoint operations $\ldiv$ and $\rdiv$ to left and right multiplication by 
\[\be_B \ldiv \be_{B'} =
\begin{cases} 
(-1)^{|B|} \be_{B'\setminus B} \in \cB^E_{E'} & \mbox{if $F = F'$, $E' \le E$ and $B \subset B'$} \\
0 & \mbox{otherwise.}
\end{cases}\]
\[\be_B \rdiv \be_{B'} =
\begin{cases} 
(-1)^{|B'|}\be_{B\setminus B'} \in \cB^{E'}_{E} & \mbox{if $F = F'$, $E \le E'$ and $B' \subset B$} \\
0 & \mbox{otherwise.}
\end{cases}\]

\begin{theorem}\label{thm:extended U action}
Take $u \in \cU^F_E$ and let $f_u \in \Hom(\T^E_E, \T^F_E)$, $g_u \in \Hom(\T^F_E, \T^E_E)$ be the corresponding morphisms.  Then the actions of $\Xi(f_u)$ and $\Xi'(g_u)$ on $\cB$ are given by 
\begin{align*}
\Xi(f_u) \cdot x = u * x\\
\Xi'(g_u) \cdot x = u \ldiv x.
\end{align*}
\end{theorem}

\begin{proof}
We only need to show the first formula, since the second will follow using Poincar\'e-Verdier duality.  In addition, we can assume that $F = \emptyset$.

We use the varieties $\M_\a$ and $\M^E_\b \subset \M_\b$ and the sheaves $\wt\T$, $\wt\T^E$ 
defined in Section \ref{sec:extending homs}.
Let $h = r^{-1}(f_u)\colon \wt\T^E \to \wt\T$ be the map which restricts to $f_u$ on a normal slice to $\M^E_\b$ at $y\in S_E$.  Restricting to fixed points and taking
hypercohomology, we get an induced homomorphism 
\[H_T^\bullet((\M^E_\b)^T; k) \to H_T^{\bullet+2d_E}(q^{-1}((\M^E_\b)^T); k) \to H_T^{\bullet+2d_E}((\M_\a)^T; k)\]
of $H_T^\bullet(\M_\b;k)$-modules, which we denote $h_*$.
Let $B$ be a basis of $V^E$, so $y_B$ is a fixed point in $\M^E_\b$.  We need to show that 
$h_*(\be_B) = u\ast \be_B$, where we identify $\be_B \in H^\bullet_T(\M^E_\b; k) = k[\Delta^E]$ with its 
restriction to the fixed point locus. 

The fiber bundle $q^{-1}(\M^E_\b) \to \M^E_\b$ is not in general trivial, but it does have invariant sections.  For each basis $B'$ of $V_E$, let $\M_\a^{B'} \subset \M_\a$ be the subvariety defined by setting $z_i=w_i=0$ for all $i\in B'$.  Then $\M_\a^{B'}$ projects isomorphically onto $\M^E_\b$, and the union of the $\M^{B'}_\a$ over all $B'$ is the $T_E$-fixed locus in $q^{-1}(\M^E_\b)$.  Furthermore, we have
$q^{-1}(y_B) \cap \M^{B'}_\a = \{y_{B\cup B'}\}$, and although the identification of $q^{-1}(y_B)$ with the central fiber of $\M_{E. \bar{\a}} \to \M_E$ provided by Lemma \ref{lem:iterated reduction}(b) is not canonical, under any such identification $y_{B'}$ corresponds to $y_{B \cup B'}$.

Consider the class $1_B \in H^0_{T}((\M^E_\a)^T;k)$ obtained by pushing forward $1 \in H^0_T(y_B;k)$.
Ifwe put  $u = \sum_{B' \in \Bas(V_E)} c_{B'}\be_{B'}$, then the previous discussion implies that
\begin{equation}\label{eqn:pushing forward point class}
h_*(1_B) - \sum_{B'} c_{B'}\be_{B'} \in H_T^{\bullet+2d_E}((\M_\a)^T; k)
\end{equation}
becomes zero after passing to $T_E$-equivariant cohomology.

Now, to compute $h_*(\be_B)$, we use the fact that
$h_*$ is a map of $H_T^\bullet(\M_\b;k)$-modules.  
Suppose that for each $i \in I\setminus E$ there exists a class $\be'_i \in H^2_T(\M_\b;k)$ which pulls back to the classes $\be_i$ on both $\M^E_\b$ and $\M_\a$.  Then we get $h_*(\be_B) = \be_Bh_*(1_B)$, 
so the theorem would follow if we knew that \eqref{eqn:pushing forward point class} is zero.  But this follows easily from the fact that $h_*(\be_B)$ is a linear combination of $\be_{B''}$, $B''\in \Bas(V)$.

To construct the classes $\be'_i$, 
we use the construction of 
$\M_\b$ as the quotient $\bar{\mu}^{-1}(0)^{ss}/(K/K_E)$
from the proof of Lemma \ref{lem:iterated reduction}.
The genericity of $\b$ implies that $K/K_E$ acts freely on $\bar{\mu}^{-1}(0)^{ss}$, so there is a
 Kirwan homomorphism  
\[H^\bullet_{(\C^*)^I/K_E}(pt;k)\to H^\bullet_T(\M_\b;k).\]
Pulling back the coordinate characters of $(\C^*)^{I\setminus B}$ by the quotient homomorphism 
$(\C^*)^I/K_E \to (\C^*)^{I\setminus B}$
gives characters which, considered as elements of 
$H^2_{(\C^*)^I/K_E}(pt;k)$, map to the required classes 
$\be'_i$ in $H^2_T(\M_\b;k)$.
\end{proof}

\begin{corollary}\label{cor:U is star-closed}
For any flats $D \le E \le F$ in $\cF$, we have
\[\cU^F_E \ast \cU^E_D \subset \cU^F_D.\]
\end{corollary}
\begin{proof}
Without loss of generality we can assume that $D=I$.  Take any
$u \in \cU^F_E$ and $u'\in \cU^E_I$, and let 
\[\Xi(f_u)\colon \T^E \to \T^F, \;\; \Xi(f_{u'})\colon \T^I \to \T^E\]
be the corresponding homomorphisms of resolution sheaves.  Applying $\Morse_I$ to the composition $\Xi(f_u)\circ \Xi(f_{u'})$ sends $1 \in \cB^I_I = k$ to $u \ast u' \in \cB^F_I$.  In other words, $u \ast u'$ arises by applying $\Morse_I$ to a homomorphism $\T^I \to \T^F$, and so it lies in $\cU^F_I$.
\end{proof}

\subsection{Circuit classes}

The description of $\cU^F_E$ given by Proposition \ref{prop:edge classes generate} is indirect, since it
gives a spanning set for its perpendicular space.  In this section we describe explicit classes associated to circuits which can be used to generate $\cU^F_E$.  As a side effect, we get a smaller generating set for the algebra $\End(\bigoplus_F \T^F)$ than the one given by Corollary \ref{cor:generating set}.

Recall that a circuit of $V$ is a minimal dependent set.
For a circuit $C$, 
let $F_C$ be the flat for which $H_{F_C} = H_C$; it is 
coloop-free, since $V_{F_C}$ has a full-rank circuit,
namely $C$ itself.  The image of $V\subset k^I$ under the projection 
$k^I \to k^C$ has codimension $1$, so it is given by an equation
\[\sum_{i\in C} b_ix_i = 0.\]
Since we have assumed that $V$ is unimodular, we can take all 
$b_i$ to be in $\{\pm 1\}$. Define a class 
\[u_C = \sum_{i \in C} b_i \be_{C \setminus i}.\]

\begin{proposition}\label{circuit class formula}
Up to multiplication by a nonzero scalar, $u_C$ is the unique class in $\cU^\emptyset_{F_C}$ 
which is a linear combination of $\be_{C\setminus i}$, $i\in C$.
\end{proposition}
\begin{proof}
Without loss of generality we can assume that $F_C = I$, so the sets $C \setminus \{i\}$ are bases for $V$.  By Proposition \ref{prop:edge classes generate}, in order to show that $u_C \in \cU^\emptyset_I$ it is enough to show that 
$\langle\alpha_X, u_C\rangle = 0$ for every edge $X$ of $V$. 
 We have 
 \[\alpha_X = \sum_{i\in I \setminus X} a_i\be_{X\cup\{i\}}\]
where $v := \sum_{i \in I \setminus X} a_i \be_i$ is a nonzero 
vector in $H_X$.  

If $X \not\subset C$, then monomials in $\alpha_X$ and $u_C$ are disjoint, so the pairing is automatically zero.  Otherwise, we have $C = X \cup \{j, j'\}$.  There are exactly two bases $B$ with $X \subset B \subset C$, and we get
\[\langle \alpha_X, u_C\rangle = (-1)^d(a_jb_{j'} + a_{j'}b_j).\]
On the other hand, since $v \in V$, we have
\[0 = \sum_{i\in C\cap(I\setminus X)} b_ia_i = a_jb_j + a_{j'}b_{j'}.\]
This together with the fact that $b_j, b_{j'} \in \{\pm 1\}$ 
gives $\langle\alpha_X, u_C\rangle = 0$.
Finally, uniqueness follows from the fact that $C \setminus \{j,j'\}$ is an edge for all $j \ne j'$ in $C$.
\end{proof}

More generally, the same definition applies to $V^E_D$ for any flats
$D\le E$ in $\cF$; any circuit $C$ of this arrangement gives rise
to $u_C \in \cU^E_{F_C}$.

\begin{proposition}\label{prop:circuit classes generate}
The space $\cU^E_D$ is spanned by classes
of the form
\begin{equation}\label{product of circuit classes}
u_{C_1} \ast u_{C_2} \ast \cdots \ast u_{C_r},
\end{equation}
where we let $F_0 = E$, $F_\ell = F_{C_\ell}$ for $i > 0$, 
$C_\ell$ is a circuit in $V^{F_{\ell-1}}_{D}$ for $1 \le \ell \le r$,
and $F_r = D$.  
\end{proposition}
\begin{remark}
In many examples, the classes $u_C$ with $F_C = D$ already span $\cU^E_D$.  For an example where they do not, take $V' = \Z(1,1) \subset \Z^2$, and let 
$V = V'\times V' \subset \Z^I$, $I = \{1,2,3,4\}$.  Then
there are no circuits $C$ with $F_C = I$, so there are no classes $u_C$ in $\cU^\emptyset_I$. 
\end{remark}
\begin{proof}
We know that these classes are in $\cU$ by Corollary \ref{cor:U is star-closed}.    
Fix a total ordering of $I$ and give the bases of $V^E_D$ the lexicographic ordering.  
Consider the subspace $\ol\cB = \ol\cB{}^E_D \subset \cB^E_D$ 
 from the proof of Proposition \ref{prop:edge classes generate}.
That proof showed that $\ol\cB$ is complementary to $(\cU^E_D)^\bot$, so the orthogonal projection of $\cU^E_D$ onto
$\ol\cB$ is an isomorphism.  Thus it is enough to find
elements of the form \eqref{product of circuit classes} which project to a basis of $\ol\cB$. 

Take any basis $B_1$ with $\be_{B_1} \in \ol{\cB}{}^E_D$.  We will construct an element of the form 
\eqref{product of circuit classes} whose projection to $\cC$ is
$\be_{B_1}$ plus a linear combination of $\be_{B'}$ for $B' \prec B_1$.  
Let $i_1$ be the minimal index in $B_1$.  Since $i_1$ is internally passive for $B_1$, there
must be a $j < i_1$ in $F_0 = E$ 
so that $B' := (B_1 \setminus \{i_1\}) \cup \{j\}$ is a basis.
The set $B_1 \cup B' = B_1 \cup \{j\}$ will contain a unique circuit $C_1$, which 
must contain $j$.  

Let $F_1 = F_{C_1}$.  Note that the largest basis which appears with a nonzero coefficient in $u_{C_1}$ is 
$B_1 \cap C_1$, so if $F_1 = D$, we are done.  
Otherwise, consider the basis $B_2 := B_1 \setminus (B_1 \cap C_1)$ 
of $V^{F_1}_D$.  We have $\be_{B_2} \in \ol{\cB}{}^{F_1}_D$, so
we can repeat the process, obtaining some $j < i_2 = \min B_2$ in $F_1$
so that $(B_2 \setminus \{i_2\}) \cup \{j\}$ is a basis, and letting 
$C_2$ be the unique circuit in $B_2 \cup \{j\}$, and $F_2 = F_{C_2}$.

Continuing this way, we eventually get circuits $C_1, \dots, C_r$ 
as in the statement of the proposition, and it is easy to see that
the lexicographically largest basis appearing in
$u_{C_1} \ast u_{C_2} \ast \cdots \ast u_{C_k}$ with a nonzero coefficient is 
$\bigcup_{\ell =1}^r (C_\ell \setminus \{i_\ell\}) = B_1$, as desired.  

The resulting elements of $\cU^E_D$ therefore project to a 
triangular basis of $\ol\cB$, and we are done.
\end{proof}

\begin{corollary}
The algebra $R$ is isomorphic to the smallest subalgebra of $\End_k(\cB)$ generated by 
\[x \mapsto u_C \ast x, \;\; x \mapsto u_C \ldiv x\]
for $E\le F$ in $\cF$ and all circuits $C$ of $V^E_F$.
\end{corollary}

\subsection{An example}

Let us give an explicit example of the ring $R$.  Consider
the lattice \[V = \Z(1,0,1,1)+\Z(0,1,1,1) \subset \Z^I, \; I = \{1,2,3,4\}.\]
This gives an arrangement of four lines in the plane with two of them equal.

There are three flats in $\cF$, namely $\emptyset, I$, and $F = \{3,4\}$.  The sub and quotient arrangements 
are given by
\[V^F = \Z(1,-1) \subset \Z^{\{1,2\}},\;\; V_F = \Z(1,1) \subset \Z^{\{3,4\}}.\]

 The spaces $\cU^E_F \subset \cB^E_F$ are as follows:
\begin{align*}
\cU^F_I & = \Z(\be_1 + \be_2) \subset \Z\langle \be_1,\be_2\rangle,\\
\cU^\emptyset_F & = \Z(\be_3-\be_4) \subset \Z\langle \be_3,\be_4\rangle,\\
\cU^\emptyset_F & = \Z\langle u_1,u_2 \rangle \subset
\Z\langle \be_{12},\be_{13},\be_{14},\be_{23},\be_{24}\rangle,
\end{align*}
where we put $u_1 = \be_{12} - \be_{13} - \be_{23}$ and  $u_2 = \be_{12} - \be_{14} - \be_{24}$.

The resulting algebra is the quotient of the quiver algebra of the double of the quiver
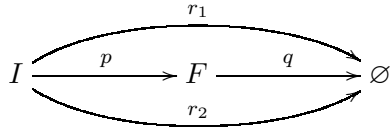
\[\xymatrix{
I \ar[rr]^p \ar@(ur,ul)[rrrr]^{r_1} \ar@(dr,dl)[rrrr]^{r_2} & & F \ar[rr]^q & &  \emptyset
}\]
(where the doubled arrows $p^*$, $q^*$, $r_1^*$, $r_2^*$ are the adjoint operations)
by the relations
\[qp = r_2 - r_1,\;\;\; p^*q^* = r_2^* - r_1^*\]
\[p^*p = -2 \cdot  1_I,\;\; q^*q = -2\cdot 1_F\]
\[q^*r_1 = p,\;\; q^*r_2 = -p\]
\[r_1^*q = p^*,\;\; r_2^*q = -p^*\]
\[\begin{pmatrix}
r_1^* \\ 
r_2^*
\end{pmatrix} \begin{pmatrix}
r_1 & r_2
\end{pmatrix}  = \begin{pmatrix}
3 & 1 \\ 
1 & 3
\end{pmatrix} 1_I \]

These relations are not independent; for example the first can be used to eliminate $r_2$, and then $r_1^*r_1$ entry of the last line implies the other three entries.

Over a field $k$, this algebra is semisimple unless $\operatorname{char} k = 2$.

\section{Homomorphisms between projective objects}

For each flat $F$, we have a projective object $\Pi_F \in \Perv(\M)$ which represents the exact functor $\Morse_F$.  In other words, there is a natural isomorphism of functors
$\Hom(\Pi_F, -) \simeq \Morse_F(-)$.  Note that in general $\Pi_F$ is not indecomposable.  
  In this section, we describe the algebra
\[\textstyle \Rc := \End(\bigoplus_{F\in \cF} \Morse_F) \cong \End(\bigoplus_{F\in \cF} \Pi_F)^{\mathrm{opp}}\]
of natural transformations between these functors.  It acts naturally on 
the space \[\cB = \bigoplus_{E, F\in \cF} \Morse_E(\T^F) = \bigoplus_{E, F\in \cF} \Hom(\Pi_E, \T^F).\]
\begin{proposition}
\label{prop:Rc action is faithful}
The action of $\Rc$ on $\cB$ is faithful.
\end{proposition}
\begin{proof}
Suppose that $\phi\colon \Morse_F \to \Morse_E$ acts trivially on $\cB$; this means that for every $D\in \cF$, the induced map $\Morse_F(\T^D) \to \Morse_E(\T^D)$ vanishes.  But since a composition series for $\T^D$ contains only copies of IC sheaves with support contained in $\M^D$, and exactly one copy of $IC_D$, the Serre subcategory generated by all of the $\T^D$ is all of $\Perv(\M,k)$.  It follows that $\phi = 0$.
\end{proof}

We will show that the image of $\Rc$ in $\End_k(\cB)$ is  generated by a small set of explicit operators, in a manner very similar to our description of $\End(\bigoplus_F \T^F)$.

\subsection{Homomorphisms to and from top stratum}

Projective perverse sheaves are in general difficult to describe topologically.  However, there is one case which is relatively simple: by adjunction, the projective $\Pi_\emptyset$ corresponding to the open stratum $S_\emptyset$ is isomorphic to ${}^p(j_{\emptyset})_!\underline{k}_{S_\emptyset,T}[2d]$.  Using this, we can give a simple description of maps from $\Pi_F$ to $\Pi_\emptyset$.

\begin{theorem}\label{thm:Homs to top projective}
Let $\epsilon\colon \Pi_\emptyset\to \T^\emptyset$ be the map obtained by adjunction from the identity map on the open stratum $S_\emptyset$.
Then for any $F\in \cF$ the map 
\begin{equation}\label{eq:projective injection}
\Hom(\Pi_F, \Pi_\emptyset) \cong \Morse_F(\Pi_\emptyset) \stackrel{\Morse_F(\epsilon)}{\longrightarrow}\Morse_F(\T^\emptyset)
\end{equation} 
is an injection, and an element of $\Morse_F(\T^\emptyset)$ is in the image if and only if it is in the kernel of $\Morse_F(\psi)\colon \Morse_F(\T^\emptyset) \to \Morse_F(\T^E)$ for all maps $\psi\colon \T^\emptyset \to \T^E$ with $E \ne \emptyset$.
\end{theorem}

Another way to describe the map \eqref{eq:projective injection} is as the composition of
\[\Hom(\Pi_F, \Pi_\emptyset) = \Hom(\Morse_\emptyset, \Morse_F) \to \Hom(\Morse_\emptyset(\T^\emptyset), \Morse_F(\T^\emptyset))\]
with evaluation at $1 \in \Morse_\emptyset(\T^\emptyset) = k$.

\begin{proof}
Let us prove this first for $F=I$.  Let $\M^+$ be the attracting set for a generic cocharacter in $V$.  
For any $A \in \Perv(\M,k)$, since $\Morse_I(A) \cong \H^0_c(A|_{\M^+})$, we have 
an exact sequence
\[\H_c^{-1}(A|_{\M^+ \setminus \M^+\cap S_\emptyset}) \to 
\H_c^{0}(A|_{\M^+\cap S_\emptyset}) \to \Morse_I(A) \to 
\H_c^{0}(A|_{\M^+ \setminus \M^+\cap S_\emptyset}) .
\] 

If $A = \Pi_\emptyset$, the left and right terms vanish, since for $F \ne \emptyset$ the stalks of $\Pi_\emptyset = {}^p(j_{\emptyset})_!\underline{k}_{S_\emptyset,T}[2d]$ at 
points of $\M^+ \cap S_F$ vanish in degrees greater than $-2\dim_\C(\M^+\cap S_F)-2$; see \cite[1.4.23]{BBD}.  Thus the middle map is an isomorphism.

On the other hand, when $A = \T^\emptyset$, the middle map is 
\[H^{2d}_c(p_\a^{-1}(\M^+\cap S_\emptyset);k) \to H^{2d}_c(p_\a^{-1}(\M^+);k),\] which is an injection, since $p_\a^{-1}(\M^+)$ is purely $2d$-dimensional and $S_\emptyset$ is open.

The map $\epsilon$ restricts to an isomorphism on the stratum $S_\emptyset$, so the induced map $\H_c^{0}(\Pi_\emptyset|_{\M^+\cap S_\emptyset})\to \H_c^{0}(\T^\emptyset|_{\M^+\cap S_\emptyset})$ is an isomorphism.  Thus we have an exact sequence
\[0 \to \Morse_I(\Pi_\emptyset) \to \Morse_I(\T^\emptyset) \to \H_c^{0}(\T^\emptyset|_{\M^+ \setminus \M^+\cap S_\emptyset}). 
\]

Now suppose that a class $a \in \Morse_I(\T^\emptyset)$ is annihilated by all maps $\T^\emptyset \to \T^E$, $E\ne \emptyset$.  The composition \[\cU^\emptyset_E \to \Hom_{\M_E}(\T^\emptyset_E,\T^E_E) \stackrel{\Xi'}{\longrightarrow}\Hom(\T^\emptyset,\T^E)\] gives a map $\psi_E\colon \T^\emptyset \to (\cU^\emptyset_E)^* \otimes_k \T^E$.  It induces an isomorphism on stalk cohomology in degree $-\dim_\C S_E$ at all points of $S_E$, so applying the functor $A \mapsto \H^0_c(A|_{S_E})$ to $\psi_E$ 
gives an isomorphism.  
Since $\Morse_I(\psi_E)(a)=0$ for every $E \in \cF \setminus \{\emptyset\}$, it follows that $a$ restricts to $0$ in $\H_c^{0}(\T^\emptyset|_{\M^+ \setminus \M^+\cap S_\emptyset})$,
and so it comes from $\Morse_I(\Pi_\emptyset)$.  

For the other direction, simply note that for any $E \in \cF\setminus\{\emptyset\}$, any map $\Pi_\emptyset \to \T^E$ is zero by adjunction.

The case of a general $F$ follows now from Corollary \ref{cor:extending from slice}.
\end{proof}

We let $\cUc^\emptyset_F \subset \cB^\emptyset_F = \Morse_F\T^\emptyset$ denote the image of the map
\eqref{eq:projective injection}.  More generally, if $F \le D$, we can take a projective 
object $\Pi^D_F \in \Perv(\M^D;k)$ which represents the exact functor $\Morse^D_F\colon \Perv(\M^D;k) \to k\md$, and the same construction gives an injective map
$\Hom(\Pi^D_F, \Pi^D_D) \to \Morse^D_F\T^D = \cB^D_F$, whose image we denote by $\cUc^D_F$.

We can define a map $\cUc^D_F \to \Hom(\Pi^D_D, \Pi^D_F)$ by using Verdier duality: an element $\uc\in \cUc^D_F$ is associated to a morphism of functors $\Morse^D_D \to \Morse^D_F$, which induces a morphism
\[\Morse^D_F \cong \D\,\Morse^D_F\,\D \to \D\,\Morse^D_D\,\D \cong \Morse^D_D,\]
where the left $\D$ is just duality of $k$-vector spaces.
Applying this morphism to $\T^D$ gives a map
$\cB^D_F \to \cB^D_D$ which is just pairing with $\uc$.

\begin{proposition}\label{prop:perp space to cUc}
A class $x\in \cB^D_F$ lies in $\cUc^D_F$ if and only if
\[\langle u_C \ast \be_B , x \rangle = 0\]
for all circuits $C$ of $V^D_F$ and all bases $B$ of $V^{F_C}_F$.
\end{proposition}

\begin{proof}
Combining Theorems \ref{thm:extended U action} and \ref{thm:Homs to top projective}, we see that $x$ is in $\cUc^D_F$ if and only if $u \ldiv x = 0$ for all 
flats $F \le E < D$ and all $u \in \cU^D_E$.  By adjunction,
$u \ldiv x = 0$ if and only if $\langle u \ast \be_B, x\rangle = \langle \be_B, u \ldiv x\rangle =0$ for all bases $B$ of $V^E_F$.  
The result now follows by Proposition \ref{prop:circuit classes generate}. 
\end{proof}

\subsection{Homomorphisms between general projectives}

In this section, we construct maps which generate the algebra $\Rc = \End(\bigoplus_F \Pi_F)^{\mathrm{opp}}$, and show how they act on $\cB$.

We do this by defining  injective linear maps
\begin{align*}
\Upsilon\colon \cUc^F_E & \to \Hom(\Pi_E,\Pi_F) \\
\Upsilon'\colon \cUc^F_E & \to \Hom(\Pi_F,\Pi_E)
\end{align*} 
for any flats $E \le F$ in $\cF$.  These maps can be viewed as ``promoting" maps between $\Pi^F_E$ and $\Pi^F_F$ to maps between $\Pi_E$ and $\Pi_F$. They are defined as follows.

By Theorem \ref{thm:Homs to top projective}, an element of $\uc \in \cUc^F_E$ gives an element of $\Hom_{\M^F}(\Pi^F_E, \Pi^F_F)$, which in turn gives a natural transformation $\Morse^F_F \to \Morse^F_E$.  Composing this with $\Phi_F$ and using Lemma \ref{lem:Morse transitivity} gives a 
transformation $\Morse_F\to \Morse_E$, and $\Upsilon(\uc)$ is defined to be the corresponding element of $\Hom(\Pi_E, \Pi_F)$.

To define $\Upsilon'(\uc)$ we proceed similarly --- compose the morphism $\Morse^F_E\to \Morse^F_F$ associated to $\uc$ with $\Phi_F$, and use Lemma \ref{lem:Morse transitivity}.

\begin{theorem}\label{thm:extended maps between projectives}
For any $E \le F \le D$ in $\cF$ and any $\check{u} \in \cUc^F_E$ we have 
\[\Upsilon(\check{u})\cdot x = x \ast \check{u}, \;
\Upsilon'(\check{u})\cdot y = y \rdiv \check{u}\]
for any $x \in \cB^D_F$, $y \in \cB^D_E$.
\end{theorem}

The proof relies on the following result, which is a sheaf-theoretic version of Theorem \ref{thm:Cohomology of intermediate Morse sheaf}. 

\begin{proposition}\label{prop:multiplication geometry}
For any coloop-free flats $F \le D$, there is an isomorphism
\[\Phi_F(\T^D) \cong (\cB^D_E \otimes_k \T^F) \oplus A,\]
where $\operatorname{supp} A \subsetneq \M^F$, such that
for any $E \le F$, applying $\Morse_E$ to both sides of this isomorphism and using Lemma \ref{lem:Morse transitivity} gives a map 
\[\cB^D_F \otimes_k \cB^F_E \to \cB^D_E\]
which is just the operation $\ast$.  
\end{proposition}

\begin{proof}
Without loss of generality we can assume that $D = \emptyset$.  We use the notation from Section \ref{sec:localization of Morse groups} and the proof of Theorem \ref{thm:Cohomology of intermediate Morse sheaf}.  
In particular, $\wt\M$ is a fixed resolution of $\M$, 
$\xi$ is a generic cocharacter of $T$ so that $\M^\xi = \M^F$, and we have the $\xi$-fixed point components
$\wt\M_j = \wt\M^{\wt F_j}$, $j=1,\dots,r$, their attracting sets $\wt\M^+_j$, the closed unions
$\wt\M^+_{\le s} := \bigcup_{j=1}^s \wt\M^+_j$,
inclusions 
$c_j\colon \wt\M^+_j \to \wt\M^+$, $d_j\colon \wt\M^+_{\le j} \to \wt\M^+$, and the map $q\colon \wt\M^+ \to \M^F$.  

Define sheaves 
\[\T_j := q_!(c_j)_!\uk_{\wt\M^+_j,T}[2d], \;\; \T_{\le j} := q_!(d_j)_!\uk_{\wt\M^+_{\le j},T}[2d]\]
on $\M^F$.  
Shifting the distinguished triangle 
\eqref{eq:Morse triangle} and applying $q_!$ gives a triangle
\[ \T_j \to \T_{\le j} \to \T_{\le j-1} \stackrel{[1]}{\to}.\]
As we saw in the proof of Theorem \ref{thm:Cohomology of intermediate Morse sheaf}, $\wt\M^+_j$ is an affine bundle over $\wt\M_j$ of rank $\frac{1}{2}\codim_{\wt\M}{\wt\M_j} = 
\frac{1}{2}\codim_{\M} \M^{F_j}$, so we have
$\T_j \cong \T^{F_j}$.  In particular it 
is a perverse parity complex, 
so by induction each $\T_{\le j}$ is a perverse parity complex as well.  It follows that $\Hom^1(\T_{\le j-1}, \T_j)=0$, so our triangle gives a split exact sequence 
\[0 \to \T_j \stackrel{\sigma}{\to} \T_{\le j} {\to} \T_{\le j-1} \to 0\]
of perverse sheaves.

We need to fix a particular splitting of this exact sequence.  Applying 
$q_!(d_j)_!$ to the adjunction map $\uk_{\wt\M^+_{\le j}} \to \uk_{\wt\M_j}$ gives a map $\tau\colon \T_{\le j} \to \T_j[2r_j]$ such that 
$\tau\sigma\colon \T_j \to \T_j[2r_j]$ is  multiplication by the equivariant Thom class $\tau_j$ of the normal bundle to $\wt\M_j$ in $\wt\M^+_j$, which up to a sign is 
the restriction of $\be_{\wt{F}_j}$ to $\wt\M_j$. 

 Let $\Theta$ be the cone of $\tau\sigma$; its hypercohomology is isomorphic to 
\[k[\Delta^{F_j}] \otimes_{A^F} A/A\tau_j .\]
Theorem \ref{thm:Cohomology of intermediate Morse sheaf} implies that the composition $\T_{\le j} \stackrel{\tau}{\to} \T_{j}[2r_j] \to \Theta$ induces the zero map on hypercohomology, and then Lemma \ref{lem:faithful hypercohomology} implies that the map itself is zero.  Since $\T_{\le j}$ is a sum of resolution sheaves, the proof of Proposition \ref{prop:Morse is faithful} shows that it satisfies the hypotheses (a) and (b) of Lemma \ref{lem:faithful hypercohomology}.  To see that $\Theta$ satisfies (c), just note that $\H^\bullet_T(j_S^!\Theta) \to \H^\bullet_T(j^!_S\T) \otimes_A A/\tau_j$ is an isomorphism, and similarly for $j_S^*\Theta$.

It follows that there is $\eta\colon \T_{\le j} \to \T_j$ so that $\tau\sig\eta = \tau$.  Thus we have $\tau\sig\eta\sig = \tau\sig$, and since $\tau\sig$ is injective on hypercohomology, Proposition \ref{prop:Morse is faithful} implies that $\eta\sig = \id_{\T_j}$.  If we define $\epsilon_j$ to be the composition of the natural projection $\Phi_F(\T) = \T_{\le r}$ onto 
$\T_{\le j}$ with $\eta$, then $\oplus_j \epsilon_j$ defines an isomorphism $\Phi_F(\T) \cong \bigoplus_j \T_j$ with the property that applying hypercohomology to $\epsilon_j$ gives the natural isomorphism $\be_{\wt{F}_j}k[\Delta] \stackrel{\sim}{\to} k[\Delta^{F_j}] \otimes_{A^{F_j}} A$ and kills all other terms.  

Each basis $B$ of $V_F$ is equal to $\wt{F}_j$ for some $j$, and the associated sheaves $\T_j$ give 
precisely the summands of $\Phi_F\T$ which are supported on $\M^F$.  We therefore get an isomorphism $\Phi_F(\T) \cong (\cB^\emptyset_F \otimes_k \T^F) \oplus A$ as desired.  
Now to see that it has the required property, we first note that it is enough to prove to the case $E=I$, since we can restrict to a transverse slice to $S_E$. 

Consider the
composition
\[\cB_F^\emptyset \otimes_k \Morse_I(\T^F) \hookrightarrow \Morse^F_I(\Phi_F\T) \to \H_T^\bullet(\Phi_F\T) \stackrel{\rho_F}{\longrightarrow} \H^\bullet_T(\T) \cong k[\Delta],\]
where the first inclusion comes from the inclusion 
$\cB^\emptyset \otimes_k \T^F \hookrightarrow \Phi_F\T$,
the second comes from the natural transformation $\Phi_I \to j_I^*$, and the third is the map $\rho_F$ introduced in \S\ref{sec:localization of Morse groups}.  

If
 $\be_B$, $\be_{B'}$ are bases of $V^\emptyset_F$ and $V^F_I$, respectively, then 
the image of $\be_B \otimes \be_{B'}$ under this map
is $\be_{B\cup B'}$.  On the other hand, we can also factor our map using the  
composition $\Morse^F_I(\Phi_F\T) \to \H_T^\bullet(\Phi_I\T) \stackrel{\rho_I}{\longrightarrow} \H^\bullet_T(\T)$, where the first map comes from Lemma \ref{lem:Morse transitivity}.  Since $\rho_I$ is injective, this implies that the image of $\be_B \otimes \be_{B'}$ in $\Morse_I(\T)$ is $\be_{B\cup B'} = \be_B \ast \be_{B'}$, as desired.   
\end{proof}

\begin{proof}[Proof of Theorem \ref{thm:extended maps between projectives}]

Taking $\uc\in \cUc^F_E$, applying the natural transformation $\phi_\uc\colon \Morse^F_F \to \Morse^F_E$ represented by $\uc$ to $\T^F$ gives the map $\cB^F_F = k \to \cB^F_E$ which sends $1$ to $\uc$.

Our definition says that applying $\Upsilon(\uc)$ to $\T^D$ is the result of 
applying $\phi_\uc$ to $\Phi_F(\T^D)$.  Proposition \ref{prop:multiplication geometry} gives 
\[\Phi_F(\T^D) \cong (\cB^D_F \otimes_k \T^F) \oplus A\]
with $A$ supported on strata strictly smaller than $\M^F$.  Since $\Morse^F_F(A) = 0$, applying $\phi_\uc$ to this decomposition 
gives the map $\cB^D_F \to \cB^D_F\otimes_k \cB^F_E$ given by $x \mapsto x \otimes \uc$.  The result now follows from the second part of Proposition \ref{prop:multiplication geometry}.
\end{proof}

\subsection{Bounding Homomorphisms between projectives}
We would like to have a result analogous to Corollary \ref{cor:generating set} to show that the homomorphisms we have constructed between the projective $\Pi_F$ generate the endomorphism algebra of $\bigoplus_F \Pi_F$.  Unfortunately, we do not know of a direct proof of such a result.  Instead, we will use the combinatorics of Gale duality along with a bound on the dimension of Hom spaces which is proved in this section.  This result, Theorem \ref{thm:bound on projective homs} below, will also be used to show that our algebras are highest weight.


Let $X$ be an algebraic variety with an action of a group $G$ stratified by equivariantly simply connected $G$-invariant strata, and suppose that the category $\Perv(X)$ of $G$-equivariant perverse sheaves constructible with respect to the stratification has enough projectives.  For each stratum $S$, let $L_S$ denote the simple perverse sheaf with support $\overline{S}$, and let $P_S$ be its projective cover.  Note that since all local systems on strata are trivial, we have $\D L_S \cong L_S$ for every $S$.

If we let $j_S\colon S \to X$ be the inclusion, then the object $\Delta_S := {}^pj_{S!}\uk_S$ is a projective cover of $L_S$ in $\Perv(\overline{S})$.  In fact, it is projective in 
$\Perv(Y)$ for any closed union of strata $Y \subset X$ in which $S$ is maximal, even if $Y$ is not irreducible.

For any $A \in \Perv(X)$ and any stratum $S$, let $[A: L_S] = \dim \Hom(P_S, A)$ be the multiplicity of $L_S$ in a composition series for $A$.

\begin{lemma}\label{lem:projective symmetry}
For any strata $R, S \subset X$, we have
\[\dim \Hom(P_R, P_S) = \dim \Hom(P_S, P_R)\]
\end{lemma}
\begin{proof}
Note that $\D P_S$ is an injective hull of $\D L_S \cong L_S$, so 
$[A:L_S] = \dim \Hom(A, \D P_S)$ for any object $A$.  In addition, $P_S$ and $\D P_S$ have isomorphic composition series.  So we have
\begin{align*}
\dim \Hom(P_R, P_S) & = \dim \Hom(P_R, \D P_S) = [P_R: L_S] \\ & = [\D P_R: L_S] = \dim \Hom(\D P_R, \D P_S) = \dim\Hom(P_S, P_R).
\end{align*}
\end{proof}

\begin{theorem} \label{thm:bound on projective homs}
 For any strata $R, S$, we have
\begin{equation}\label{eq:BGG ineq}
[P_S: L_R] = \dim \Hom(P_R, P_S) \le \sum_{Q} [\Delta_Q : L_R] [\Delta_Q : L_S]
\end{equation}
where the sum is over all strata $Q$ (note that the summand is zero unless $R \cup S \subset \overline{Q}$).

Equality holds for every $R, S$ if and only if $\Perv(X)$ is highest weight for the partial order $\le$ given by inclusion of closures of strata. 
\end{theorem}

\begin{proof}
We use induction on the number of strata in $X$.  If there is only one stratum, the theorem is trivial.  Otherwise suppose the theorem holds for all varieties with fewer strata. Let $U$ be an open stratum, and set $Y = X \setminus U$.
  
First suppose that $R=U$, so that $P_R = \Delta_R = \Delta_U$.  
The only nonzero summand on the right side of \eqref{eq:BGG ineq} is when $Q = U$, and $[\Delta_U:L_U] = 1$,  so the sum gives 
\begin{align*}
[\Delta_U: L_S] & = \dim \Hom(P_S, P_U) = \dim \Hom(P_U, P_S)
\end{align*}
by Lemma \ref{lem:projective symmetry}.
Thus the inequality \eqref{eq:BGG ineq} holds.

The case $S = U$ now follows by Lemma \ref{lem:projective symmetry}, so we can assume now that $R, S \ne U$.   
 Consider the evaluation map
\[\phi\colon \Hom(\Delta_U, P_S) \otimes \Delta_U \to P_S,\]
and let $C = \coker(\phi)$.  Since $\phi$ is an isomorphism on the open set $U$, $C$ is supported on $Y$.  
In fact,  $C$ is a projective cover of $L_S$ in the subcategory $\Perv(Y) \subset \Perv(X)$,  
 since $\Hom(\Delta_U, A) = 0$ for any $A \in \Perv(Y)$.  
 
 By the inductive hypothesis, we have
\begin{equation}\label{eq:Projective inequality}
[C: L_R] \le \sum_{Q \ne U} [\Delta_Q:L_R][\Delta_Q:L_S].
\end{equation} 
  On the other hand we have 
  $[P_S:L_R] = [C:L_R] + [\im(\phi):L_R]$ and Lemma \ref{lem:projective symmetry} gives \[[\im(\phi):L_R] \le (\dim \Hom(\Delta_U, P_S))[\Delta_U:L_R] = [\Delta_U:L_S][\Delta_U:L_R].\]
  The inequality \eqref{eq:BGG ineq} follows.

If equality holds for every $R$, $S$, then each map $\phi$ is an injection, and \eqref{eq:Projective inequality} is an equality, so an induction shows that every projective $P_S$ has a filtration by $\Delta_Q$, $Q\subset X$, so $\Perv(X)$ is highest weight.    

On the other hand, if $\Perv(X)$ is highest weight, then we have the reciprocity formula
\[(P_S:\Delta_Q) = \dim \Hom(P_S,\nabla_Q) = [\nabla_Q: L_S] = [\Delta_Q: L_S],\]
so we have equality in \eqref{eq:BGG ineq}.
\end{proof}

\subsection{Applying the bound to endomorphisms of Morse functors}

We now apply the result of the previous section to study homs between the projective objects $\Pi_F \in \Perv(\M, k)$. 

\begin{theorem}\label{thm:bound on Rc}
For any strata $E,F \in \cF$, we have
\[\dim \Hom(\Morse_E,\Morse_F) = \dim \Hom(\Pi_F, \Pi_E) \leq \sum_{D \ge E, F} \dim \cUc^D_F \dim \cUc^D_E \]
and equality holds for every $E,F$ if and only if $\Perv(\M)$ is highest weight.
\end{theorem}

\begin{proof}
We can decompose the projective $\Pi_F$ in terms of the indecomposable projectives: $\Pi_F \cong \bigoplus_{A \in \cF}  (P_A)^{\bigoplus m^F_A}$ for some $m^F_A \ge 0$.  Then Theorem \ref{thm:bound on projective homs} gives
\begin{align*}
\dim \Hom(\Pi_F,\Pi_E) &= \sum_{A,B} \dim \Hom(P_A,P_B) m_A^F m_B^E  \\
& \leq \sum_{A,B} \sum_{D} [\Delta_D : L_A] [\Delta_D : L_B] m_A^F m_B^E \\
& = \sum_{D} \sum_{A,B} (m^F_A \dim \Hom(P_A,\Delta_D))(m^E_B \dim\Hom(P_B,\Delta_D)) \\
& = \sum_D \dim \Hom(\Pi_F, \Delta_D) \dim \Hom(\Pi_E, \Delta_D)) \\ 
& = \sum_D \dim \Morse_F (\Delta_D) \dim \Morse_E (\Delta_D) \\
& = \sum_D \dim \cUc^D_F \dim \cUc^D_E. 
\end{align*}
The last equality holds because considered as an object in $\Perv(\M^D, k)$, $\Delta_D$ is isomorphic to $\Pi^D_D$, the projective object representing the Morse functor at the top stratum of $\M^D$.

Equality holds if and only if 
\begin{equation}\label{eq:projective hom}
\dim \Hom(P_A, P_B) = \sum_D [\Delta_D:L_A][\Delta_D:L_B]
\end{equation}
 for all $A, B$ with $m^F_Am^E_B \ne 0$.  Since $m^F_F = m^E_E = 1$, we have equality for all $E, F$ if and only if \eqref{eq:projective hom} holds for all $A, B$, which by Theorem \ref{thm:bound on projective homs} holds if and only if $\Perv(\M,k)$ is highest weight with respect to the partial order $\le$.
\end{proof}

\section{Gale duality}

To the sublattice $V\subset \Z^I$ we can associate its 
Gale dual $V^! \subset \Z^I$, which is simply the 
 perpendicular space to $V$ under the standard pairing.  It defines an arrangement of hyperplanes indexed by the same set $I$, in an $|I|-\rank V$-dimensional space.  Our running assumption that $V$ is unimodular and has no loops or coloops implies that the same is true of $V^!$.
 The matroid of $V^!$ is the dual of the matroid of $V$.  This means that $B \mapsto I \setminus B$ defines a bijection from bases of $V$ to bases of $V^!$.

Let $\cF^!$ be the poset of cyclic flats of $V^!$.  There 
is an order-reversing bijection from $\cF$ to $\cF^!$ given by $F \mapsto F^! := I \setminus F$.  Moreover, for any flats $E \le F$ of $\cF$, the
Gale dual of $V^E_F$ is $(V^!)^{F^!}_{E^!}$.

Let $\M^!$ be the hypertoric variety defined by $V^!$.  As we did for $\M$, for each $F\in \cF^!$ we can define resolution sheaves $\T^F \in \Perv(\M^!, k)$ and Morse functors $\Morse_F\colon \Perv(\M^!, k) \to k\md$,
and we get a $k$-vector space
\[\cB^! = \bigoplus_{E,F\in \cF^!} (\cB^!)^F_E = \bigoplus_{E,F} \Morse_E(\T^F)\]
with commuting actions of two rings $R^! := \End(\bigoplus_{F\in \cF^!} \T^F)$ and $\Rc^! := \End(\bigoplus_{E \in \cF^!} \Morse_E)$.  We will use our combinatorial descriptions of these rings to show that they have a very simple relation with the rings $R, \Rc$ defined by $V$.

The space $(\cB^!)^F_E$ has a basis consisting of vectors $\be_B$, where 
$B$ is a basis of $(V^!)^F_E$.  We define an isomorphism $\Gamma\colon \cB \to \cB^!$ by
\[\Gamma(\be_B) = \be_{E\setminus (F\cup B)}\in (\cB^!)^{E^!}_{F^!}\]
for $\be_B \in \cB^F_E$.  Up to a sign, this isomorphism respects the pairings on both sides.  It also sends the multiplication operation $\ast$ on $\cB$ to the opposite of the multiplication on $\cB^!$: if $\be_B \in \cB^F_E$, $\be_{B'}\in \cB^E_D$, we have
\begin{align*}
\Gamma(\be_B \ast \be_{B'}) & = \Gamma(\be_{B\cup B'}) = \be_{D \setminus (F \cup B \cup B')} \\
& = \be_{D\setminus(E\cup B')}\ast \be_{E\setminus(F\cup B)} = \Gamma(\be_{B'})\ast \Gamma(\be_B).
\end{align*}
Similarly, it interchanges the left and right adjoint  operations, up to a sign: if $\be_B\in \cB^F_E$, $\be_{B'} \in \cB^F_D$, $\be_{B''} \in \cB^D_E$, then
\begin{align*}
\Gamma(\be_{B'} \ldiv \be_B) & = \pm \Gamma(\be_{B}) \rdiv \Gamma(\be_{B''}), \;\text{and} \\ \Gamma(\be_{B}\rdiv \be_{B''}) & = \pm \Gamma(\be_{B''}) \ldiv \Gamma(\be_B).
\end{align*}

Next we describe the effect of $\Gamma$ on the edge classes $\alpha_X$.  Let $X$ be an edge of $V$; it is the result of removing a single element from some basis.  Then $I \setminus X$ is given by adding an element to a basis of $V^!$.  It follows that $X^! := I\setminus X$ contains a unique circuit $C$, and in fact $X^!$ is the disjoint union of $C$ and a basis $B$ of $(V^!)^{F_C}$.

\begin{lemma}\label{lem:edges and circuits}
Up to multiplication by a nonzero scalar, we have \[\Gamma(\a_X) = u_C\ast \be_B.\]
\end{lemma}
\begin{proof}
Take a nonzero vector 
$v = (a_i)_{i\in I\setminus X}$ in the lattice $V \cap \Z^{I\setminus X}$, then we have $\a_X = \sum_{i\in I\setminus X} a_i\be_{X\cup \{i\}}$.  On the other hand, easy linear algebra shows that the image of $V^!$ under the projection $\Z^I \to \Z^{I\setminus X}$ is cut out by the unique linear equation $\sum_{i\in I\setminus X} a_ix_i =0$, and the projection to $\Z^C$ is given by the same equation (note that $a_i = 0$ if $i\notin C$), so we have
$u_C = \sum_{i} a_i\be_{C \cup \{i\}}$ up to multiplication by a scalar.  The result follows easily.  
\end{proof}

\begin{corollary}
For any $E \le F$ in $\cF$, we have
\[\Gamma(\cU^F_E) = \cUc^{E^!}_{F^!}, \; \Gamma(\cUc^F_E) =\cU^{E^!}_{F^!}. \]
\end{corollary}
\begin{proof}
Combine Lemma \ref{lem:edges and circuits} and Propositions \ref{prop:edge classes generate} and \ref{prop:perp space to cUc}.
\end{proof}

Since the actions of $R$ and $\Rc$ on $\cB$ are faithful (Propositions \ref{prop:Morse is faithful} and \ref{prop:Rc action is faithful}), we will consider them as subalgebras of $\End_k(\cB)$, and similarly we consider $R^!$ and $\Rc^!$ as subalgebras of $\End_k(\cB^!)$.  The map $\Gamma$ induces an isomorphism
$\End_k(\cB) \to \End_k(\cB^!),$ which we also denote by $\Gamma$.

\begin{theorem}
We have $\Gamma(R) = \Rc^!$ and  $\Gamma(\Rc) = R^!$.  All of these algebras are quasi-hereditary.
\end{theorem}
\begin{proof}
Since $R$ is generated by the left action by elements of $\cU^F_E$ and their adjoints, the previous corollary implies immediately that $\Gamma(R)$ is contained in $\Rc^!$.  For any $E, F \in \cF$, this implies that
\begin{align*}
\dim \Hom(\Morse_{E^!},\Morse_{F^!}) & \ge \dim \Hom(\T^E, \T^F) \\
& = \sum_{D \le E, F} \dim \cU^E_D \dim \cU^F_D \;\;\;\;\text{by Theorem \ref{thm:cellular homs} }\\
& = \sum_{D^! \ge E^!, F^!} \dim \cUc^{D^!}_{E^!} \dim\cUc^{D^!}_{F^!}.
\end{align*}
Combining this with Theorem \ref{thm:bound on Rc}, we see that this must be an equality, and therefore the algebras $R$ and $\Rc^!$ are quasi-hereditary.

The statements for $\Rc$ and $R^!$ follow immediately, since the Gale dual of $V^!$ is $V$.
\end{proof}

This implies that the category $\Perv(\M,k)$ is highest weight for the closure order $\le$ on the poset $\cF$ of strata.  We can therefore apply Proposition 3.3 from \cite{JMW16}: the complexes $\T^F$ are perverse, and parity for the trivial pariversity, and since all strata have even dimension, they are also parity for the dimension pariversity.  (Note that although in \cite[3.3]{JMW16} it is assumed that all strata are simply connected, all that is required is that they be equivariantly simply connected.)  We can therefore deduce the following, which is the main result of the paper.

\begin{theorem}
The objects $\T^F \in \Perv(\M,k)$ are tilting; the sum
$\T := \bigoplus_{F\in \cF} \T^F$ is a tilting generator.  We have an isomorphism of rings
\[\End_{\Perv(\M,k)}(\T) \cong \End_{\Perv(\M^!,k)}(\Pi),\]
where $\Pi := \bigoplus_{F^! \in \cF^!} \Pi_{F^!}$, under which $\T^F$ corresponds to $\Pi_{F^!}$, $F^! = I \setminus F$. 
\end{theorem}

\section{Isomorphism with matroidal Schur algebras}

In the paper \cite{BMMatroid}, we define a $k$-algebra $R(M)$ associated to any matroid $M$, realizable or not, and prove that $R(M)$ is quasi-hereditary and has Ringel dual isomorphic 
to $R(M^*)$, where $M^*$ is the dual matroid.
In this section, we show that that construction agrees with the one in this paper.  (More precisely, the algebra in \cite{BMMatroid} also depends on the choice of
 a weight function 
$a \colon I \to k^\times$.  In what follows, we will take $a(i) \equiv 1$; in this case $R(M)$ can be obtained by extending scalars from a $\Z$-form $R(M)_\Z$.) 
 
Fix a lattice $V\subset \Z^I$ satisfying our assumptions, and let $R(V)$ be the combinatorial algebra described in Corollary \ref{cor:generating set} and Theorem \ref{thm:extended U action}, so that we have 
$\End_{\Perv(\M,k)}(\bigoplus_F\T^F) \cong R(V)$.

\begin{theorem}\label{thm:bridge between papers}
There is an isomorphism
\[R(M)\cong R(V)\]
where $M$ is the matroid defined by $V$.
\end{theorem}

The rest of this section will be devoted to a proof of this result.

The definition of $R(M)_\Z$ is very similar to the definition of $R(V)$: it is a subalgebra of $\End_\Z(\cB(M))$, where 
$\cB(M) = \bigoplus_{E,F} \cB(M)^F_E$ is a free $\Z$-algebra bigraded by the poset $\cF(M)$ of cyclic flats of $M$.  It is generated by the operators of multiplication by elements of a certain homogeneous ideal $\cU(M) \subset \cB(M)$ and their adjoints under a dual pairing on $\cB(M)$.

We fix an ordering on the index set $I$.  Let $\Lambda$
be the exterior $\Z$-algebra generated by variables
$e_i$, $i\in I$.  For any subset $S = \{i_1, \dots, i_r\}$ of $I$ with $i_1< \dots < i_r$, let
$e_S = e_{i_1} \wedge \dots \wedge e_{i_r}$.
Put a pairing on $\Lambda$ such that the basis of monomials $e_S$ is orthonormal.

For flats $E \le F$ in $\cF$, we define $\cB(M)^F_E$ to be the subspace of $\Lambda$ spanned by monomials
$e_B$, where
$B$ is a basis of
the matroid $M(E)/F$ associated to $V^F_E$, with the induced pairing.
 For $D \le E\le F$, the multiplication 
\[\cB(M)^F_E \otimes_\Z \cB(M)^E_D \to \cB(M)^F_D\]
is given by the exterior product.

The subspace $\cU(M)^F_E \subset \cB(M)^F_E$ is defined to be the kernel of the homological boundary map $\partial$ on $\Lambda$
\[\partial(e_{i_1} \wedge \dots \wedge e_{i_r}) = \sum_{j=1}^r (-1)^{j-1}e_{i_1} \wedge \dots \wedge \widehat{e_{i_j}} \wedge \dots \wedge e_{i_r}\]
restricted to $\cB(M)^F_E$.


Let $V \subset \Z^I$ satisfy all the hypotheses of this paper, and let $M$ be the corresponding matroid.  We define an isomorphism $\phi\colon\cB(V) \to \cB(M)$ as follows.  For each flat $F \in \cF$, choose a generator $\omega^F$ of $\bigwedge^{d^F} (V^F)^*$.  Then there are unique generators 
$\omega^F_E \in \bigwedge^{d^F_E} (V^F_E)^*$ so that $\omega^F = \omega^F_E \wedge \omega^E$.
It follows that 
\begin{equation}\label{eq:multiplicative system of forms}
\omega^F_E \wedge \omega^E_D = \omega^F_D \; \text{for all} \; D\le E\le F.
\end{equation}

Now we define the map $\phi$ on $\cB(V)^F_E$ by
\[\phi(\be_B) = \left\langle\omega^F_E,  ev_B(\be_{i_1}) \wedge \dots \wedge ev_B(\be_{i_r}) \right\rangle e_{i_1}\wedge \dots \wedge e_{i_r}\]
where $B = \{i_1, \dots i_r\}$ is a basis of $V^F_E$.  
Note that this is independent of the ordering of the elements of $B$.  

For any $j\in B$, the vector $ev_{B}(\be_{j})$ is the image in 
$(V^F_E)_k$ of a vector $ev^\Z_B(\be_{j})$, where
$ev^\Z_B$ is the inverse of the composition
$V^F_E \hookrightarrow \Z^I \twoheadrightarrow \Z^B$ (note that this is an isomorphism because $V^F_E$ is unimodular).  The vectors $ev^\Z_{B}(\be_{j})$ form a basis of
the unimodular lattice $V^F_E$, so we have $\phi(\be_B) = \pm e_B$.
It follows that $\phi$ respects the pairings on $\cB(V)$ and $\cB(M)$, up to a sign.
It also follows easily from \eqref{eq:multiplicative system of forms}
that $\phi$ is an isomorphism of rings.

To complete the proof of Theorem \ref{thm:bridge between papers}, we need to show that $\phi$ sends $\cU(V)^F_E$ to $\cU(M)^F_E$.  
To see this, we will compute what $\phi$ does to an edge class and use Proposition \ref{prop:edge classes generate}.  Without loss of generality, we can assume that $V^F_E = V$.

Let $X$ be an edge of $V$, and let
$v = \sum_{j} a_j \be_j$ be a generator of the rank one lattice $H_X \subset V$.  Recall that $a_j \ne  0$ if and only if $B_j = X \cup \{j\}$ is a basis; let $J \subset I$ be the set of indexes for which this holds.

 For any $i\in X$ and any $j, j'\in J$, we have
\[ev^\Z_{B_j} \be_i \equiv ev^\Z_{B_{j'}} \be_i \pmod{v},\]
since both sides are sent to the same element by the projection $\Z^I \to \Z^X$, while the kernel of $V \hookrightarrow \Z^I \twoheadrightarrow \Z^X$ is spanned by $v$.  

For any $j\in J$, we have 
$v = ev_{B_j}(a_j\be_j)$.
Thus we see that 
\[ev_{B_j} (a_j\be_j) \wedge \bigwedge_{i\in X} ev_{B_j} \be_i =  v \wedge \bigwedge_{i\in X} ev_{B_j} \be_i\] is independent of $j$.  Let $c$ be the pairing of this element with $\omega$.
Since $ev^\Z_{B_j}$ sends $v$ and $\be_i$, $i\in X$ to a basis of the lattice $V$, we have $c = \pm 1$.

 This gives
\[\phi(\a_X) = \sum_{j\in J} \phi(a_j \be_{X\cup j}) = c\sum_{j\in J} e_j\wedge \bigwedge_{i\in X} e_i.\]
This means that for any $x\in \cB^\emptyset_I = \cB(V)^\emptyset_I$ we have 
\[\langle \phi(\a_X), x\rangle = c\langle \sum_{j\in J}e_j \wedge e_X, x\rangle = c\langle\sum_{j\in I} e_j \wedge e_X, x\rangle = c\langle e_X, \partial x\rangle,\]
where the last equality follows easily from the definitions of $\partial$ and the pairing.

Now for any $z\in \cB(V)^\emptyset_I$, we have
\begin{align*}
z \in \cU(V)^\emptyset_I & \iff \langle \alpha_X, z\rangle =0\;\; \text{for all $X$}\\
& \iff \langle \phi(\alpha_X), \phi(z)\rangle = 0  \;\;\text{for all $X$}\\
& \iff \langle e_X, \partial\phi(z)\rangle =0 \;\;\text{for all $X$ (since $c = \pm 1$)}\\
& \iff \phi(z) \in \cU(M)^\emptyset_I. 
\end{align*}

Theorem \ref{thm:bridge between papers} follows.

\section{Technical proofs}\label{sec:big proofs}

\subsection{Topological $G$-stratifications}

Before describing stratifications of hypertoric varieties, we must explain 
carefully what we mean by a stratification.  
The following definitions were previously used in \cite{BP09} to study hypertoric varieties.
Suppose that $G$ is a compact connected Abelian Lie group, and let $X$ be a $G$-space.  A {\em $G$-decomposition} of $X$ is a 
partition $X = \bigsqcup_{\a \in P} X_\a$ into locally closed $G$-invariant topological manifolds indexed by a poset $(P, \le)$ so that
\begin{enumerate}
\item For all $\a$, we have $\overline{X_\a} = \bigcup_{\b \le \a} X_\b$,
\item For all $\a$, the stabilizer $G_x$ of a point $x\in X_\a$ does not depend on $x$.
\end{enumerate}

We can then give a definition of when a $G$-decomposition is a (topological) {\em $G$-stratification}, which is inductive in $|A|$.  If $A = \{a\}$ has one element, then $X = X_\a$ is itself a manifold and we impose no further conditions.  Otherwise, we ask that for every $\a \in P$ and every $x \in X_\a$, there exist
\begin{enumerate}
\item a $G$-invariant neighborhood $U \subset X$ of the orbit $G \cdot x$,
\item a $G_x$-space $L$ with a $G_x$-stratification $\{L_\b \mid \b > \a\}$, and 
\item a $G$-equivariant homeomorphism   
\[\phi\colon U \stackrel{\sim}{\longrightarrow} G \times_{G_x} \operatorname{cone}(L) \times D \]
where $D$ is an open disk in a Euclidean space of dimension $\dim X_\a - \dim G/G_x$, and $\operatorname{cone}(L)$
is the open topological cone $(L\times [0,1))/(L\times \{0\})$
on $L$ with the induced $G_x$-action,
\end{enumerate}
such that for all $\b > \a$ we have
\[\phi(U \cap X_\b) = G \times_{G_x} (L_\b \times (0,1)) \times D\]
and $\phi(U \cap X_\a) = G \times_{G_x} \{v\} \times D$, where $v \in L$ is the apex of the cone.
The space $\operatorname{cone}(L)$ is called a normal slice to $X$ at the stratum $X_\a$.

In the special case when $G = 1$, we simply say that the decomposition is a topological stratification.  

\begin{lemma}\label{lem:G-strat is topological strat}
Any $G$-stratification is also a topological stratification.
\end{lemma}
\begin{proof}
Given a map $\phi$ as in the definition above, let $\Delta \subset \mathfrak g$ be the intersection of a subspace complementary to $\mathfrak g_x$ with a small ball around $0$.  Then restricting $\phi$ to 
 to  $U' = \phi^{-1}(\exp \Delta \times \operatorname{cone}(L) \times D)$ gives a homeomorphism of $U'$ with $\exp \Delta  \times \operatorname{cone}(L) \times D$.
\end{proof}

\begin{definition}
If $\cS = \{X_\a\}_{\a \in P}$ is a $G$-decomposition of a space $X$, the constructible 
derived category $D^b_\cS(X,k)$ is the full subcategory of objects in $D^b(X,k)$ whose cohomology sheaves are are locally constant on each $X_\a$ and have finite-dimensional stalks.  The equivariant
constructible category $D^b_{G,\cS}(X,k)$ is the full subcategory of $D^b_G(X,k)$ of objects whose underlying non-equivariant complex is $\cS$-constructible.
\end{definition}

\begin{proposition} If $\cS = \{X_\a\}_{\a \in P}$ is a $G$-stratification of $X$, then 
Verdier duality $\D$ sends 
$D^b_\cS(X, k)$ to itself, and for any inclusion 
 $j \colon Y \to X$ of a locally closed union of the $X_\a$, the functors $j^*, j_*, j^!, j_!$ restrict to functors between $D^b_\cS(Y, k)$
and $D^b_\cS(X, k)$ (and therefore also give functors between $D^b_{G, \cS}(Y, k)$ and $D^b_{G, \cS}(X, k)$).
\end{proposition}
\begin{proof}
By Lemma \ref{lem:G-strat is topological strat}, $\cS$ is a topological stratification.  The result in this case is standard --- see \cite{GM83} Propositions 1.9 and 1.12, for instance.
\end{proof}

The utility of the stronger notion of $G$-stratification lies in the following result.

\begin{lemma}\label{inheriting G-stratifications} If $\{X_\a\}_{\a\in P}$ is a $G$-stratification of $X$ and $G' \subset G$ is a closed
Lie subgroup, then  $\{X_\a/G'\}_{\a \in P} $ is a $G/G'$-stratification of $X/G'$.  
\end{lemma}
\begin{proof}
Given a homeomorphism $\phi$ as in the definition above, taking the quotient of both sides by $G'$ gives a homeomorphism
\[U/G' \stackrel{\sim}{\longrightarrow} (G \times_{G_x} \operatorname{cone}(L) \times D)/G' \cong G/G' \times_{G'/G'_x} \operatorname{cone}(L)/G'_x \times D.\]
The result follows.
\end{proof}

\subsection{Monodromy of constructible sheaves}

The following lemma is probably known, but we could not find it in the literature.

\begin{lemma} \label{lem:sheaf family}
Let $\cS = \{X_\a\}_{\a \in P}$ be a $G$-invariant topological stratification of $X$ and let $M$ be a connected manifold with the trivial $G$-action.  
For each $p\in M$, let 
$s_p\colon X \to X \times M$ be given by $s_p(x) = (x, p)$.  Suppose that $A \in D^b_G(X \times M,k)$ is constructible with respect to the product stratification $X \times M = \sqcup_\a X_\a \times M$.  Then for any $p, q\in M$ there is an isomorphism $s_p^*A \cong s_q^*A$ in $D^b_{G,\cS}(X, k)$.  If $M$ is simply connected, the isomorphism is canonical.
\end{lemma}
\begin{proof}
For any map $f \colon Y \to Z$, we let $\hat{f}\colon X \times Y \to X \times Z$ be the induced map. 
If we choose a path $\gamma\colon I \to M$ from $p$ to $q$, then  $\hat\gamma^*A$ is constructible with respect to the ``stratification with corners" $X \times I = \sqcup_\a X_\a \times I$.  Let $\pi\colon I \to \{pt\}$ be the constant map, and for $t\in I$ let $i_t\colon \{pt\} \to I$ be the inclusion with image $\{t\}$.  Applying $\hat\pi_*$ to the adjunction map $\hat\gamma^*A \to \hat\imath_{t*}\hat\imath_t^*\hat\gamma^*A$ gives a map
$\hat\pi_*\hat\gamma^*A \to \hat\imath_t^*\hat\gamma^*A$; the first statement will follow if we can show it is an isomorphism for all $t$.  If this holds for two objects in a distinguished triangle, it also holds for the third, so we are reduced to the case $A = j_!L$, where $j\colon X_\a \times M \hookrightarrow X \times M$ is 
the inclusion and $L$ is a local system on $X_\a \times M$.  Using base change and the contractibility of $I$ we have $\hat{\gamma}^* j_! L \cong \hat\pi^* j_\a L_\a$, where
$j_\a \colon X_\a \to X$ is the inclusion and $L_\a$ is a local system on $X_\a$.  The result in this case follows from \cite[2.7.7(iv)]{KS}. 

For the second statement, we choose a homotopy $I\times I \to M$ of paths from $p$ to $q$, and proceed similarly.
\end{proof}

\subsection{The contracting lemma}

We need the following lemma about restricting sheaves to the fixed point locus of a contracting action.  This is a familiar result in the algebraic setting, but we need a topological version, so we include a proof.

Suppose that a space $X$ has an action of $\R_{>0}$ preserving a topological stratification $\cS$, such that the set $X_0$ of fixed points is a union of strata.  In addition, suppose that
the action contracts $X$ onto $X_0$, in the sense that the action map $m\colon \R_{>0} \times X \to X$ extends to a continuous map $\bar m\colon \R_{\ge 0} \times X \to X$.  Then we have a projection $p\colon X \to X_0$ given by $p(x) = \bar{m}(0,x)$.
Let $i\colon X_0 \to X$ be the inclusion.

\begin{lemma}\label{contracting lemma}
If $\cF \in D^b_\cS(X,k)$, then there are natural isomorphisms
\[p_*\cF \cong i^*\cF, \; p_!\cF \cong i^!\cF.\]
\end{lemma}
\begin{proof}
The second isomorphism follows from the first by Verdier duality, so we only need to show the first isomorphism holds.

We can assume that $\cF$ is the lower-$*$ extension of a local system $L$ on a stratum $S\in \cS$, since these objects generate $D^b_\cS(X)$.

Applying $p_*$ to the adjunction map $\cF \to i_*i^*\cF$ gives a natural map $p_*\cF \to i^*\cF$.  For any point $x \in X_0$, the induced map on stalk cohomology at $x$ is
\begin{equation}\label{eq:contracting morphism}
\varinjlim H^\bullet(p^{-1}(V), \cF) \to \varinjlim H^\bullet(U, \cF),
\end{equation}
where the limits are taken over all open neighborhoods $U$ of $x$ in $X$ and $V$ of $x$ in $X_0$.

The second limit can be replaced by the limit over neighborhoods $U'$ of $x$ in $X$ such that $tU' \subset U'$ for all $0 < t \le 1$, since every neighborhood $U$ contains such a neighborhood $U'$.  To see this, use the continuity of $\bar{m}$ to get an open neighborhood
$[0,\epsilon) \times U''$ of $(0, x)$ contained in $\bar{m}^{-1}(U)$,
and let $U' = \bar{m}((0,\epsilon)\times U'')$.  Furthermore, by intersecting with $p^{-1}(U'\cap X_0)$, we can assume that $p(U') = U'\cap X_0$.

For such an open set $U'$ the restriction map 
\[H^\bullet(p^{-1}(p(U')), \cF) \to H^\bullet(tU', \cF)\] is an isomorphism for all $t > 0$, since these groups can be identified with singular cohomology  
$H^\bullet(p^{-1}(p(U'))\cap S, L)$ and
$H^\bullet(tU'\cap S, L)$, which are isomorphic by a homotopy argument.  It follows that \eqref{eq:contracting morphism} is an isomorphism, as desired.
\end{proof}

\subsection{The Hyperk\"ahler moment map}\label{sec:hk moment map}

In order to prove that the resolution sheaves are independent of the choice of resolution and the 
hyperbolic restriction functors are independent of the 
choice of cocharacter, we show that they can be expressed 
in terms of the hyperk\"ahler moment map.  This allows 
a discrete parameter (resolution, cocharacter) to be 
replaced by a continuous one which lives in a simply connected space.

Let $\M$ be the affine hypertoric variety corresponding to the lattice $V \subset \Z^I$ satisfying the hypotheses of Section \ref{sec:arrangements}.
We denote hyperk\"ahler moment maps by a bold font.  The
hyperk\"ahler moment map for the $(S^1)^I$-action on $T^*\C^I$ is a map
\[\bmu_I\colon T^*\C^I \to \R^I \otimes_\R \R^3 = \R^I \oplus \C^I.\]
It is made up of a real part $\mu_{I,\R}\colon T^*\C^I \to \R^I$ 
given by  
\[(z_i,w_i)_{i\in I} \mapsto \half\left(|z_i|^2 - |w_i|^2\right)_{i\in I},\] 
and a 
 complex part $\mu_I\colon T^*\C^I \to \C^I$ given by 
\[(z_i, w_i)_{i\in I}  \mapsto (z_iw_i)_{i\in I}.\] 
The fibers of $\bmu_I$ are exactly the $(S^1)^I$-orbits in $T^*\C^I$ and the orbit of a point $(z_i, w_i)$ is isomorphic to $(S^1)^r$, where
$r$ is the number of $i\in I$ for which $z_i \ne 0$ or $w_i \ne 0$.

Let $K_\R = K \cap (S^1)^I$ be the maximal compact torus in $K$ and let $\mathfrak{k}_\R \subset \mathfrak{k}$ be its Lie algebra.
The hyperk\"ahler moment map $\bmu_K$ for the action of $K_\R$ is the composition
 of  $\bmu_I$ with the natural projection $\R^I \otimes \R^3 \to \mathfrak{k}^*_\R \otimes \R^3$. 

Put $\bV = V \otimes_\Z \R^3$.
The inclusion $\bmu_K^{-1}(0) \subset \mu_K^{-1}(0)$ induces a $T_\R$-equivariant homeomorphism (see \cite[Proposition 3.3]{Ko07})

\begin{equation*}\label{eq:Kirwan-Ness}
\M \cong \bmu^{-1}_K(0)/K_\R = \bmu_I^{-1}(\bV)/K_\R. 
\end{equation*}

Since $\bmu_I$ is $K_\R$-equivariant for the trivial action on the target, we obtain a residual moment map $\bmu_T\colon \M \to \bV$.
Our two stratifications of $\M$ can then be described in terms of this map: if we put $\bV^F := V^F_\R\otimes \R^3 \subset \bV$, $\breve{\bV}^F := \bV^F \setminus \bigcup_{E  < F} \bV^E$ and
\[\mathring\bV^F := \bV^F \setminus \bigcup_{\substack{E \in \cF \\ E < F}} \bV^E,\]
then for any $F \in \cF$ we have $\M^F = \bmu_T^{-1}(\bV^F)$, 
 $\breve{S}_F = \bmu_T^{-1}(\breve\bV^F)$, and 
$S_F := \bmu_T^{-1}(\mathring\bV^F)$.  

\subsection{Refined stratifications of $\M$}

In this section explain how to refine the fine stratification of a hypertoric variety $\M$ by adding linear subspaces to the decomposition of $\bV$ by flats, and show that it varies nicely as the linear spaces move.  In this section we drop our earlier assumption that the lattice $V$ defining $\M$ is coloop-free.

Let $Z$ be a connected smooth manifold, and define $\bmu_Z = \bmu_T \times \id_Z \colon \M\times Z \to \bV \times Z$.  Consider $\bV \times Z$ as a trivial real vector bundle over $Z$, and 
let $\cV$ be a finite collection of smooth subbundles 
such that 
\begin{itemize}
\item $\bV^F \times Z$ is in $\cV$ for every flat $F$ (not just the coloop-free flats)
\item For any $\cE_1, \dots, \cE_r \in \cV$, the intersection $\cE_1 \cap \dots \cap \cE_r$ is in $\cV$; in particular it has constant rank. 
\end{itemize}
We then let $\cS$ be the coarsest decomposition of $\bV\times Z$ into connected subsets such that every $\cE \in \cV$ is a union of elements of $\cS$.  

\begin{proposition}\label{prop:parametric stratification} The decomposition $\bmu^{-1}_Z(S)$, $S \in \cS$ is a $T_\R$-stratification of $\M \times Z$.  
\end{proposition}

The proof of this will occupy the remainder of this section.

First note that Lemma \ref{inheriting G-stratifications} implies that it is enough to prove this result when $\M = T^*\C^n$, $T_\R = (S^1)^I$, and $\bV = \R^I \otimes \R^3$, so we will assume from now on that we are in this case.  We denote the corresponding hyperk\"ahler moment map $\bmu_T = \bmu_I$ simply by $\bmu$.

The proof involves showing that $\cS$ is a stratification of $\bV\times Z$, and then lifting this structure to $\M$.
We do this by lifting paths along the moment map $\bmu_Z\colon \M\times Z \to \bV \times Z$.  
Given a smooth path $\gamma\colon [0,1] \to S$
and a point $p_0 \in \bmu_Z^{-1}(\gamma(0))$, there is a unique 
$\tilde\gamma \colon [0,1] \to \bmu_Z^{-1}(S)$ such that 
\begin{enumerate}
\item $\bmu_Z \circ \tilde\gamma = \gamma$ and $\tilde\gamma(0) = p_0$, and 
\item $\tilde\gamma'(t)$ is perpendicular to $T_{\tilde\gamma(t)}(T_\R\cdot \tilde\gamma(t))
\subset T_{\tilde\gamma(t)}(\M \times Z)$ for all $t \in [0,1]$, under the standard metric on $T^*\C^n \cong \R^{4n}$. (We can put any metric on $Z$; the resulting lift is independent of this choice.)
\end{enumerate} 
Since the  metric is $T_\R$-invariant, the map $p_0 \mapsto \tilde\gamma(t)$ is $T_\R$-equivariant for every $t$.

We also need to lift families of paths. 

\begin{lemma}\label{path lifting}  Let $Y$ be any space, and suppose $\gamma\colon Y \times [0,1] \to \bV \times Z$ and $p_0\colon Y \times \{0\} \to \M \times Z$ are continuous functions so that
\begin{itemize}
\item For every $y \in Y$ the path $t \mapsto \gamma_y(t) = \gamma(y,t)$ is smooth and has image entirely contained in one stratum $S_y \in \cS$,
\item $\frac{d}{dt}\gamma(y,t)$ is a continuous function of $y$ and $t$, and
\item $\bmu_Z \circ p_0 = \gamma|_{Y \times \{0\}}$.
\end{itemize}
Then lifting each path $\gamma_y$ with initial point $p_0(y)$ produces a continuous map
$\tilde\gamma\colon Y \times [0,1] \to \M \times Z$.
\end{lemma}

\begin{proof} It is enough to prove this when $\cS$ is the minimal stratification  where $\cV$ contains only the coordinate bundles $(\R^F\otimes \R^3) \times Z$ for all $F \subset I$.  In this case, $Z$ plays no role, so we can assume that $Z$ is a point. We can also assume that $|I| = 1$, i.e.\ $\M = T^*\C$ and $\bV = \R^3$, since the lift of a product of maps to $T^*\C$ is the product of lifts. 

Let $Y_0 = p_0^{-1}(0)$; then $\gamma$ sends all of $Y_0 \times [0,1]$ to $0$ and all of $(Y\setminus Y_0) \times [0,1]$ to $\R^3\setminus \{0\}$.  At points of the open set $(Y \setminus Y_0) \times [0,1]$ continuity of $\tilde\gamma$ is a standard result about parallel transport.  To complete the proof, note that $\tilde{\gamma}(Y_0 \times[0,1]) = \{0\}$, and  pulling back a basis of neighborhoods of $\{0\}$ in $\R^3$ gives a basis of neighborhoods of $\{0\} \in T^*\C$, so the continuity of $\tilde\gamma$ at points of $Y_0 \times [0,1]$ follows from the continuity of $\gamma$ there.
\end{proof}

We return to the proof of Proposition \ref{prop:parametric stratification}. 
Recall our assumption that
$\M = T^*\C^n$, $T_\R = (S^1)^I$, and $\bV = \R^I \otimes \R^3$.

 We begin by showing that projection makes $\M \times Z$ into a stratified fiber bundle over $Z$.  
Fix a point $z_0 \in Z$. 
We will find a neighborhood $U$ of $z_0$ in $Z$ and a homeomorphism
$\kappa\colon \bV \times U \to \bV \times U$ so that 
\begin{enumerate}
\item for every $z \in U$, $\kappa$ restricts to a homeomorphism $\bV \times \{z\} \to \bV\times \{z\}$,
which is the identity when $z = z_0$,
\item for each stratum $S \in \cS$, we have $\kappa^{-1}(S) = S_0 \times U$, where $S \cap (\bV \times \{z_0\}) = S_0 \times \{z_0\}$, and
\item for every $v \in \bV$, the map $\kappa_v(z) = \kappa(v, z)$ is a smooth map $U \to \bV \times U$,
and its derivative $d\kappa_v(z) \colon T_zU \to T_{(v, z)}(\bV \times U)$ is continuous as a function of $v$ and $z$.  
\end{enumerate}

(Note that this is similar to, but stronger than, the conclusion of Thom's first isotopy lemma \cite{Math}.  In that result, a local trivialization is obtained by integrating a controlled vector field, which need not be continuous, so $d\kappa_v(z)$ need not be continuous.)

Assuming that $\kappa$ exists for the moment, we will lift it to a $T_\R$-equivariant homeomorphism
$\tilde\kappa\colon \M\times U \to \M \times U$ so that $\mu_Z \circ \tilde\kappa = \kappa \circ \mu_Z|_{\M \times U}$.  This implies that 
\[\tilde\kappa^{-1}(\bmu_Z^{-1}(S)) = \bmu_Z^{-1}(\kappa^{-1}(S)) = \bmu_Z^{-1}(S_0 \times U) = \bmu^{-1}(S_0) \times U\]
for all $S \in \cS$, so $\tilde\kappa$ gives a local trivialization of the decomposition $\{\bmu_Z^{-1}(S)\}_{S\in \cS}$ of $\M \times Z$.  

To define this lift, let us assume without loss of generality that $Z$ is an open ball in $\R^m$ centered at $0$.  Then we can 
apply Lemma \ref{path lifting} with $Y = \M \times Z$ and the maps
\[\gamma\colon \M \times Z \times [0,1] \to \bV \times Z, \;\; \gamma(x,z,t) = \kappa(\bmu(x), tz)\]
with initial condition $\tilde\delta(x, z, 0) = (x, 0)$ to get a map 
$\tilde\gamma\colon \M \times \bV \times [0,1] \to \M \times \bV$.  We define $\tilde\kappa(x, z) = \tilde\gamma(x,v,1)$.  Then we have \[\bmu_Z(\tilde\kappa(x, z)) = \bmu_Z(\tilde\gamma(x,z,1)) = \gamma(x,z,1) = \kappa(\bmu(x),z) = \kappa(\bmu_Z(x,z)),\]
so $\tilde\kappa$ is a lift of $\kappa$.  Since $\tilde\kappa$ is $T_\R$-equivariant, it is a bijection $\M\times Z \to \M\times Z$.  Finally, the inverse function is continuous since it can be obtained by applying Lemma \ref{path lifting} to 
\[\delta(x,z,t) = \kappa(\lambda_{1-t}(\kappa^{-1}(\bmu_Z(x,z)))),\;\; \tilde\delta(x,z,0) = (x,z),\]
where $\lambda_t(x,z) = (x, tz)$.

We now turn to constructing the homeomorphism $\kappa$. 
The the closures of intersections of strata with $\bV \times \{z_0\}$ are polyhedral cones, so we can refine this decomposition to a simplicial fan $\Phi$ in $\bV \times \{z_0\}$ so that the closure of each stratum is a union of cones of the fan.  

For each one-dimensional cone $\tau \in \Phi$, let $S_\tau \in \cS$ be the unique stratum containing $\tau \setminus \{0\}$, and choose a smooth function $\nu_\tau\colon U_\tau \to \bV$, defined on a neighborhood of $z_0$, so that $\nu_\tau(z_0) \in \tau$ and $(\nu_\tau(z), z)$ is in $S_\tau$ for all $z \in U_\tau$.  Such a function exists because each stratum, and in particular $S_\tau$, is a fiber bundle over $Z$ with path-connected fibers. 

For a cone $\sigma \in \Phi$ spanned by the rays 
$\tau_1,\dots, \tau_r$, let the cone $\sigma(z) \subset \bV\times\{z\}$ be the sum of 
$\R_{\ge 0}\nu_{\tau_i}(z)$ for $1\le i\le r$.  Then there is a neighborhood $U \subset \bigcup_\tau U_\tau$ of $z_0$ so that
\[\Phi(z) = \{\sigma(z) \mid \sigma \in \Phi\}\]
is a simplicial fan for all  $z \in U$.  We then define \[\kappa(x, z) = \left(\sum_{\tau} a_\tau \nu_{\tau}(z), z\right),\]
where $x = \sum_\tau a_\tau \nu_\tau(z_0)$ and $a_\tau = 0$ unless $\tau$ is contained in the smallest cone of $\Phi$ containing $x$.  The required properties of $\kappa$ are now easy to check.

We have thus reduced the problem to the case where $Z = \{z_0\}$ is a single point; we will write $\M$, $\bV$ instead of $\M\times Z$, $\bV\times Z$.  
Take a point $s$ in a stratum $S \subset \bV$. 
Since our strata are open subsets of linear spaces, we can find a set $N$ which is the intersection of an affine space meeting $S$ transversely at $s$ and a ball around $s$, and 
a ball $W$ around $0$ in $T_sS$, with the property
that the addition map $N \times W \to N + W$ is a stratum-preserving homeomorphism, where $N$ and $N+W$ are given the induced stratifications and $W$ is given the trivial stratification.  

We can then use Lemma \ref{path lifting} to lift this to a $T_\R$-equivariant stratified 
homeomorphism $\bmu^{-1}(N + W) \cong \bmu^{-1}(N) \times W$: take $Y = \bmu^{-1}(N) \times W$, and define $\gamma\colon Y\times I \to N+W$ by $\gamma(x, w, t) = \bmu(x) + tw$.  Lifting this map using the initial condition $p_0\colon Y\times \{0\} \to \M$ given by 
$p_0(x, w, 0) = x$ gives a map $\tilde{\gamma}\colon Y \to \M$ whose restriction to $Y \times \{1\} \cong Y$ is the required homeomorphism.
 
Fix a point $\tilde s \in \bmu^{-1}(s)$.   
Denote the standard holomorphic coordinate functions on $T^*\C^I$ by $z_i, w_i$.  Let 
$J$ be the set of $i\in I$ for which either $z_i(\tilde s) \ne 0$ or $w_i(\tilde s) \ne 0$.  
The role of $z_i$ and $w_i$ is symmetric, so we can assume that $z_i(\tilde s) \ne 0$ for all $i \in J$. 
Shrinking $N$ if necessary, we can assume that 
$z_i$ doesn't vanish on $\bmu^{-1}(N)$ for all $i\in J$. 

Define a map $q \colon \bmu^{-1}(N) \to (S^1)^J$ by $q(x) = (z_i(x)/|z_i(x)|)_{i\in J}$.  It maps each fiber of $\bmu$ surjectively 
onto $(S^1)^J$, and
$q^{-1}(q(\tilde s))$ is invariant under $(T_\R)_{\tilde s} = (S^1)^{I\setminus J}$, so acting by $T_\R = (S^1)^I$
gives a stratified homeomorphism
\[\bmu^{-1}(N) \cong (S^1)^I \times_{(S^1)^{I\setminus J}} q^{-1}(q(\tilde s)).\]
Thus $q^{-1}(q(\tilde s))$ will be our normal slice to $\bmu^{-1}(S)$ at $\tilde s$.
We can show that $q^{-1}(q(\tilde s))$ is the cone over a $(T_\R)_{\tilde s}$-stratified space $L$ by lifting radial paths from points of $N$ to $s$; the details are left to the reader.

\subsection{Proof of Proposition \ref{prop:Omega is canonical}}\label{proof:Omega is canonical}
In order to prove Proposition \ref{prop:Omega is canonical}, we use an alternative description of the resolution $\M_\a$ as a hyperk\"ahler quotient.  
If we consider the character $\alpha$ of $K$ as an element of $\mathfrak k^*_\R$, then \cite[3.2]{Ko07} gives a natural homeomorphism $\M_\alpha \cong \bmu^{-1}_K(\alpha,0,0)/K_\R$, where we identify
$\mathfrak k^*_\R \otimes \R^3$ with $\mathfrak k^*_\R \oplus \mathfrak k^*_\R \oplus\mathfrak k^*_\R \cong \mathfrak k^*_\R \oplus \mathfrak k^*$.
In addition, the map 
$p_\a\colon \M_\a \to \M$ can be obtained by factoring the composition
\[\bmu_K^{-1}(\a,0,0)\hookrightarrow \mu_K^{-1}(0) \to  \mu_K^{-1}(0)\mod K = \M\]
through the projection $\bmu_K^{-1}(\a,0,0) \to \bmu_K^{-1}(\a,0,0)/K_\R$.  

We can view this as a fiber of a family over $\mathfrak k^*_\R\otimes \R^3$, in the following way. Let 
 $M = T^*\C^I/K_\R$, and let $\bar\bmu_K\colon M \to \mathfrak k^*_\R\otimes \R^3$ be the map induced by $\bmu_K$.
For any subset $Y \subset \mathfrak k^*_\R\otimes \R^3$, 
we put $M_Y = \bar{\bmu}_K^{-1}(Y)$.
Then if $\iota\colon \mathfrak k^*_\R \to \mathfrak k^*_\R \otimes \R^3$ is given by $\iota(\a) = (\a,0,0)$, the previous paragraph says that $\M_\a \cong M_{\iota(\alpha)}$.
 
 Let $M_\R = M_{\iota(\mathfrak k^*_\R)} = \mu_K^{-1}(0)/K_\R$.  As with our description of the map $p_\a$ above, we have a natural map 
\[M_\R = \mu^{-1}_K(0)/K_\R \to \mu^{-1}_K(0)\mod K = \M.\] 
Let $p\colon M_\R \to \mathfrak k^*_\R \times \M$ denote the product of the first coordinate of $\bar\bmu_K$ with  this map.  We can consider $p$ as a family of maps $p_{a} \colon M_{\iota(a)} \to \M$ for all $a\in \mathfrak k^*_\R$, and this notation agrees with the description of $p_\a$ above.  At the other extreme, $p_0\colon M_0 \to \M$ is a homeomorphism.

We call an element $a\in \mathfrak k^*_\R \otimes \R^3$ generic if any lift to $\R^I \otimes \R^3$ has at least 
$|I| - d$ nonzero entries, where we view $\R^I \otimes \R^3$ as $I$-tuples of elements of $\R^3$.  
Thus $\a\in \mathfrak{k}^*_\R$ is generic if and only
if $\iota(\a)$ is generic.  Then $M_a$ is smooth for all 
generic $a$.

Take $\a\in \mathfrak{k}^*_\R$ generic, and 
consider the ray $A = \R_{\geq 0} \cdot \iota(\a) \subset \mathfrak k^*_\R \otimes \R^3$ and the open ray $A^\circ = A \setminus \{0\}$.  We have a diagram of fiber squares:

\begin{equation*}
\xymatrix{
 M_{A^\circ} \ar[d]^{p_{A^\circ}} \ar[r]^{\tilde \jmath} & M_A \ar[d]^{p_A} & M_0 \ar[d]^{p_0}_{\sim} \ar[l]_{\tilde\imath} \\
 A^\circ \times \M \ar[r]^j & A \times \M & \{0\} \times \M \ar[l]_i
}
\end{equation*}
where the horizontal maps are inclusions and the vertical maps are restrictions of $p$.

 The map $p_{A^\circ}$ is locally trivial over $A^\circ$, 
 so there is a natural isomorphism 
 \[(p_{A^\circ})_* \uk_{M_{A^\circ}} \cong \uk_{A^\circ} \boxtimes (p_a)_* \uk_{M_{\iota(a)}}.\]  
It follows that \[i^* j_* (p_{A^\circ})_* \uk_{M_{A^\circ}} \cong i^*(\uk_{A} \boxtimes (p_a)_* \uk_{M_{\iota(a)}}) = (p_a)_* \uk_{M_{\iota(a)}}.\]
  Then by proper base change we get a canonical isomorphism
  \[(p_a)_* \uk_{M_{\iota(a)}} \cong p_{0*}{\tilde\imath}^* {\tilde\jmath}_* \uk_{M_{A^\circ}}. \]
Since $p_0$ is a homeomorphism, this allows us to regard the pushforward sheaf for $p_\a\colon \M_\a \to \M$ as a nearby cycles sheaf.

Now we allow $a$ to vary in the whole space of hyperk\"ahler parameters.   Let $Z \subset \mathfrak k^*_\R \otimes_\R \R^3$ be the set of generic points.  It is a complement of a union of  codimension 3 subspaces and hence is simply connected.

Define
\[\wh M = \{(m, a) \in M\times Z \mid \bar{\bmu}_K(m) \in \R_{\ge 0}a\}.\]
The fiber of $a \in Z$ under the projection
$\wh M \to Z$ is $M_{\R_{\ge 0}a}$.
Let $\wh{M}_{0} = M_0 \times Z \subset \wh M$,
and let $\wh{M}_{>0} = \wh{M} \setminus \wh{M}_0$. 
If $\hat\imath\colon \wh{M}_0 \to \wh{M}$ and $\hat\jmath\colon \wh{M}_{>0} \to \wh{M}$ are the inclusions, then 
we can consider the ``global nearby cycles sheaf" $\hat\imath^*\hat\jmath_*\uk_{\wh{M}_{>0}}$.  We claim that 
this sheaf is constructible with respect to the stratification of $M_0 \cong \M \times Z$ by the product strata
$\breve{S}_F \times Z$.
To do this, we will apply Proposition \ref{prop:parametric stratification}  
to construct a stratification of $M \times Z$ so that $\wh{M}$ and $\wh{M}_0$ are unions of strata, and the induced stratification of $\wh{M}_0$ is the desired product stratification.

As a first step, we construct a stratification of $T^*\C^I \times Z$, where we consider $T^*\C^I$ as a hypertoric variety with the moment map $\bmu_I \colon T^*\C^I \to \R^I \otimes \R^3$, using the set $\cV$ of subbundles of $(\R^I \otimes \R^3) \times Z$ consisting of all intersections of
\begin{itemize}
\item the tripled coordinate spaces $(\R^J \otimes \R^3) \times Z$, $J\subset I$,
\item  the trivial bundle $\bV \times Z$,
\item and the inverse image $W$ 
of the tautological line bundle
\[\{(tv,v) \in (\mathfrak k^*_\R \otimes \R^3) \times Z \mid t\in \R\}.\]
under the projection $(\R^I \otimes \R^3) \times Z \to (\mathfrak k^*_\R \otimes \R^3) \times Z$. 
\end{itemize}
Notice that $\bV \times Z$ is a subbundle of $W$ of codimension one, and it separates $W$ into two components; let $W_+$ denote the component corresponding to points where $t > 0$.

Taking the quotient of the resulting $(S^1)^I$-stratified space $T^*\C^I \times Z$ by $K_\R$, we obtain, by Lemma \ref{inheriting G-stratifications}, a $T_\R$-stratification of $M \times Z$.  If $\bmu_{I,Z} = \bmu_I \times \id_Z\colon T^*\C^I \times Z \to (\R^I \otimes \R^3)\times Z$, then
we have $\bmu_{I,Z}^{-1}(\bV \times Z)/K_\R = \wh{M}_0$, and $\bmu_{I,Z}^{-1}(W_+)/K_\R = \wh{M}$.  Hence both $\wh{M}$ and $\wh{M}_0$ are identified with unions of strata in $M \times Z$.  Moreover, since intersecting $\bV \times Z$ with any of the subbundles defined above gives a product bundle, the resulting stratification of $\M \times Z$ is a product stratification.

We can now apply Lemma \ref{lem:sheaf family} to deduce that the restrictions of
 $\hat\imath^*\hat\jmath_*\uk_{\wh{M}_{>0}}$ to the fibers over any two points of $Z$ are canonically isomorphic.  If $\alpha \in \mathfrak{k}^*_\R$ is generic and comes from a cocharacter of $K$, 
 the restriction of this complex to the fiber over $\iota(\a)$ is isomorphic to 
  the pushforward sheaf $(p_\a)_*\uk_{\M_a}$.  
  Proposition \ref{prop:Omega is canonical} follows.

\subsection{Proof of Proposition \ref{prop:vanishing cycles iso}}\label{sec:hyperloc equals vanishing}

In this section we show that our hyperbolic restriction functor is isomorphic to a ``real vanishing cycles" functor applied to the hyperk\"ahler moment map, which allows us to use hyperk\"ahler rotation to see that it is independent of the choice of a generic cocharacter.

 Given any point $\zeta \in \mathfrak t_\R \otimes \R^3 = \bV^*$, let $f_\zeta\colon\M \to \R$ be the composition of $\zeta$ with the hyperk\"ahler moment map $\bmu_T$.  
Let $\M^{\geq 0}_\zeta = f^{-1}_\zeta(\R_{\geq 0})$, $\M^0_\zeta = f_\zeta^{-1}(0)$, and $i\colon \M^0_\zeta \to \M^{\geq 0}_\zeta$ and $j\colon \M^{\geq 0}_\zeta \to \M$ be the closed inclusions.  Our  ``vanishing cycles" functor is the functor $\phi_\zeta \colon D^b_{T_\R}(\M) \to D^b_{T_\R}(\M^0_\zeta)$ given by 
$\phi_\zeta(S) = i^!j^* S$.

In what follows, we will view cocharacters of $T$ also as elements of $\mathfrak t_\R \otimes \R^3$ via the inclusion $V^* \hookrightarrow V^* \otimes_\Z \R^3$
which sends $\xi$ to $\xi\otimes (1,0,0)$. Thus for a cocharacter $\xi \in V$ we can take $\zeta = \xi$ and consider both $\Phi_\xi$ and $\phi_\xi$.

The fixed set $\M^\xi$ is contained in $\M^0_\xi$.  Let $h$ denote the inclusion.

\begin{proposition}\label{prop:hyperloc equals vanishing}
For any cocharacter $\xi$ of $T$ there is a natural isomorphism  
\[\phi_\xi \circ Res \simeq h_* \circ  Res \circ  \Phi_\xi,\] where $Res = Res_{T,T_\R}$ denotes the forgetful map from $T$-equivariant sheaves to $T_\R$-equivariant sheaves.
\end{proposition}

Before giving the proof of this result, we explain why it implies Proposition \ref{prop:vanishing cycles iso}.
For any flat $F$, let $\bV_F^* = (V_F)^*\otimes \R^3 = (\bV/\bV^F)^* \subset \bV^*$.  We say that $\zeta \in \bV_F^*$ is generic if the only tripled flats $\bV^E$ which are annihilated by $\zeta$ are those
with $E \le F$.  This agrees with our definition of genericity for cocharacters of $T$.

\begin{proof}[Proof of Proposition \ref{prop:vanishing cycles iso}.]

Let 
$Z \subset \bV_F^*$ be the set of generic covectors.  (Warning: this is not the set $Z$ from the previous section, which was a subset of $\bV$.)
Define $\M_Z^0, \M_Z^{\ge 0} \subset \M \times Z$ to be the locus of 
$(p, \zeta)$ for which $\langle\bmu(p), \zeta \rangle = 0$ or $\langle\bmu(p), \zeta \rangle \ge 0$, respectively.  

Consider the ``parametric vanishing cycles" functor $i^!_Zj^*_Z\vartheta^*$, where $i_Z\colon \M^0_Z \to \M^{\ge 0}_Z$ and $j_Z\colon \M^{\ge 0}_Z \to \M \times Z$ are the inclusions and $\vartheta\colon \M \times Z \to \M$ is the projection.  An argument using Proposition \ref{prop:parametric stratification} similar to the one in Section \ref{proof:Omega is canonical}
shows that this functor takes sheaves constructible with respect to the stratification by $\breve{S}_F$ to sheaves constructible with respect to the stratification by
$\breve{S}_F \times Z$.  Lemma \ref{lem:sheaf family} now implies that the restriction of this functor to 
$\M \times \{\zeta\} \cong \M$ does not depend on $\zeta \in Z$.  But this restriction is precisely the functor $\phi_\zeta$.  The result now follows from Proposition \ref{prop:hyperloc equals vanishing} and the fact that $h_*$ and $Res$ have left inverses.
\end{proof}

The main idea of the proof of Proposition \ref{prop:hyperloc equals vanishing} is to modify the action of $\xi(\R_{>0})$ to obtain an action which contracts $\M$ onto $\M^0_\xi$.  This is similar to arguments in symplectic geometry which use the gradient flow of the norm squared of the moment map.

\begin{lemma}\label{lem:orbit structure}
For any $x\in \M$, the map $\R_{>0} \to \R$ given by $t\mapsto f_\xi(\xi(t)x)$
is constant with value zero if $x$ is fixed by $\xi$, and otherwise it is strictly increasing, with image $\R, \R_{>0}$, or $\R_{<0}$.  If the image is $\R_{>0}$ or $\R_{<0}$, then the orbit $\xi(\R_{>0})x$ contains a fixed point in its closure.
\end{lemma}
 As a result,
we can reparametrize the action of $\xi(\R_{>0})$ to obtain an action which contracts $\M$ onto $\M^0_\xi$: for $t > 0$ and $x\in \M$, let $t \ast x$ be the unique point in $\xi(\R_{>0})x$ for which 
$f_\xi(t\ast x) = tf_\xi(x)$.  
This action extends to an action of the multiplicative semigroup $\R_{\ge 0}$: for every $x \in \M$, 
the closure $\overline{\xi(\R_{>0})x}$ contains a unique element of $\M^0_\xi$, and we let
$0 \ast x$ be this point. 
\begin{lemma}\label{lem:continuous action}
This action of $\R_{\ge 0}$ is continuous, and contracts $\M^+_\xi$ onto $\M^\xi$.
\end{lemma}

Before giving the proofs of lemmas \ref{lem:orbit structure} and \ref{lem:continuous action}, let us show how they imply Proposition \ref{prop:hyperloc equals vanishing} and thus Proposition \ref{prop:vanishing cycles iso}.
Consider the diagram
 \[
\xymatrix{
\M^+_\xi \ar[r]^k\ar@<.5ex>[d]^{q'} & \M^{\geq 0}_\xi \ar[r]^j\ar@<.5ex>[d]^{q} & \M\\
\M^\xi \ar[r]_h \ar@<.5ex>[u]^f & \M^0_\xi \ar@<.5ex>[u]^i &
}
\]
where $f$, $h$, $i$, $j$, $k$ are the inclusions and $q$, $q'$ send $x$ to $0 \ast x$. The contracting lemma (Lemma \ref{contracting lemma}) gives an isomorphism $\phi_\xi = i^!j^* \simeq q_!j^*$.  Let $A$ be a $T$-equivariant complex on $\M$, and take any $x \in \M^0_\xi \setminus \M^\xi$.  By proper base change, we have $(\phi_\xi A)_x \cong \H^\bullet_c(A|_{q^{-1}(x)})$.  But $A$ is constant on $q^{-1}(x) \cong [0,\infty)$, so this cohomology vanishes.  In other words, the sheaf $\phi_\xi A$ is supported on the fixed locus $\M^\xi$.  

Since $\M^+_\xi = q^{-1}(\M^\xi)$, another application of proper base change and the contracting lemma 
now implies that 
\[h^*(\phi_\xi A) \cong (q')_!k^*j^*A \cong f^!g^*A = \Phi_\xi(A).\]   
Proposition \ref{prop:hyperloc equals vanishing} follows.

We turn now to the proofs of Lemmas \ref{lem:orbit structure} and \ref{lem:continuous action}. 
Both the original action 
\[(t, p) \mapsto t \cdot p = \xi(t)p\]
of $\R_{>0}$ on $\M$ and the modified semigroup action $\ast$ of $\R_{\ge 0}$ 
on $\M$ commute with the action of the maximal compact torus $T_\R \subset T$, and so they 
induce actions ''downstairs" on $\bV$, which we denote by $\hcd$ and $\ha$, respectively.

As a first step to showing that $\ast$ is continuous, we will show that $\hat{\ast}$ is continuous.
To do this, we need to analyze how the moment map image of a point in $\M$ moves under the action of $\xi(\R_{>0})$.

Before treating a general hypertoric variety $\M$, we look at $\M=T^*\C^I$ and study the infinitesimal action of $\R^I = \operatorname{Lie}(\R_{>0})^I \subset (\C^*)^I$.  Take a point $p\in T^*\C^I$, and let $v = \bmu_I(p)$ be its image under the full hyperk\"ahler moment map.  Take any $X \in \operatorname{Lie}(\R_{>0})^I$. 
Since the action of $(\C^*)^I$ does not affect the complex part $\mu_I$ of $\bmu_I$, the vector 

\begin{equation*} 
X_p := \left.\frac d{dt} \bmu_I(\exp(tX)\cdot p)\,\right|_{t=0} \in T_v(\R^I \otimes \R^3) = \R^I \otimes \R^3
\end{equation*}
lies in the subspace $\R^I \otimes \R(1,0,0)$.  (We freely use the canonical isomorphism of a vector space with its tangent space at any point.)
 
 The $i$th coordinate $v_i = (\half (|z_i|^2 - |w_i|^2), z_iw_i) \in \R \times \C \cong \R^3$ of $v$ has Euclidean norm $|v_i| = \half(|z_i|^2 + |w_i|^2)$, so
 the $i$th coordinate of 
 $X_p$ has real part
 \[\left.\frac{d}{dt} \half \left( |e^{tX_i}z_i|^2 - |e^{-tX_i}w_i|^2\right)\right|_{t=0} = 2X_i |v_i|.\]
 In other words, if $|v| \in \R^I$ denotes the vector whose $i$th coordinate is $|v_i|$, and for $a, b \in \R^I$ we put $a \circ b = (a_ib_i)_{i\in I} \in \R^I$, then 
 we have
\begin{equation}\label{eq:n-torus action}
X_p = 2(|v| \circ X) \otimes (1,0,0).
\end{equation}
(Note that in this formula $X$ lies in $\operatorname{Lie}((\R_{>0})^I)$, while $|v| \circ X$ is naturally in the dual space $\operatorname{Lie}((\R_{>0})^I)^*$; both spaces are identified with $\R^I$ in the natural way.) In what follows, we extend this notation to subsets $A, B \subset \R^I$: we put $A \circ B = \{a\circ b\mid a\in A, b\in B\}$.

Now we turn to the case of a general hypertoric variety $\M$.  Consider the cocharacter $\xi$ as an element in $\t_\R = \R^I/\mathfrak{k}_\R$. 
(Note that we have been using $T_\R$, $K_\R$ to denote the maximal compact subgroups of $T, K$, so we have
$\operatorname{Lie}(T_\R) = \sqrt{-1}\,\mathfrak t_\R$, and $\exp(t\xi) = \xi(e^t)$ for $t \in \R$.)

Let $\pi\colon \R^I \to \t_\R$ denote the natural quotient map.  Let $L_\xi\colon \bV \to \R$ be the linear function induced by $\xi$, so $f_\xi = L_\xi \bmu_T$.

\begin{lemma}\label{lem:infinitesimal action}
For any $x \in \M$, the set 
\begin{equation}\label{eq:action set}
(2|\bmu_T(x)| \circ \pi^{-1}(\xi)) \cap V_\R.
\end{equation}
contains a unique point $\alpha$; 
 we have
\[\xi_x := \left.\frac{d}{dt} \bmu_T(\exp(t\xi)\cdot x)\right|_{t=0} = (\alpha, 0,0).\]

We always have $L_{\xi*}(\xi_x) \ge 0$; if $L_{\xi*}(\xi_x)=0$, then $f_\xi(x) = 0$.
\end{lemma}

\begin{proof}  
As noted before, the map $\bmu_K^{-1}(0)/K_\R \to \mu_K^{-1}(0)\mod K$ is a homeomorphism.  A more precise statement is the following (see \cite[Section 3.1]{Ko07}).  A $K$-orbit $O$ contained in $\mu_K^{-1}(0)$ is closed if and only if it meets $\bmu_K^{-1}(0)$, in which case the intersection is a single $K_\R$-orbit.  The map $O \mapsto O\cap \bmu_K^{-1}(0)$ gives a 
bijection between closed $K$-orbits in $\mu_K^{-1}(0)$ and $K_\R$-orbits in $\bmu_K^{-1}(0)$.

We use this to transfer the action of the $1$-parameter group $\exp(t\xi)$ from $\mu_K^{-1}(0)\mod K$ to $\bmu_K^{-1}(0)/K_\R$, so that we can apply the hyperk\"ahler moment map.
Let $r$ denote the projection
$\bmu_K^{-1}(0) \to \bmu_K^{-1}(0)/K_\R$,
and choose any lifts $\hat{x} \in r^{-1}(x)$ and 
$X \in \pi^{-1}(\xi)$.

For any $t\in\R$, the orbit
$K(\exp(tX) \cdot \hat x) = \exp(tX)(K\cdot \hat x)$ is closed, so we have
\[\exp(t\xi)\cdot x = r(K(\exp (tX)\cdot \hat x) \cap \bmu_K^{-1}(0)).\]  
It follows that there is a $Y(t) \in \mathfrak{k}_\R$ so that
$r(\exp(Y(t) + tX)\cdot \hat x) = \exp(t\xi) \cdot x$.  It is unique modulo the Lie algebra of the stabilizer $K_{\hat x}$, and so
applying $\bmu_T$ and differentiating, we see that $\xi_x$ and $X_{\hat{x}}$ differ by a vector of the form $Y_{\hat{x}}$, $Y \in \mathfrak{k}_\R$.
Since $\pi^{-1}(\xi) = X + \mathfrak{k}_\R$, equation \eqref{eq:n-torus action} shows that $\alpha$ lies in the set \eqref{eq:action set}.  

To see that there is a unique such $\alpha$, suppose that $\alpha, \alpha'$ are both in 
\eqref{eq:action set}.
Since $V_\R = \mathfrak{k}_\R^\bot$, we have $\alpha-\alpha' \in (|\bmu_T(x)| \circ \mathfrak k_\R) \cap \mathfrak k_\R^\bot$.  But for any subspace $W \subset \R^I$, the set
\[(\R_{\ge 0} \circ W) \cap W^\bot\]
contains only $\{0\}$: if $a \in \R_{\ge 0}$ and $w \in W$ satisfy $a \circ w \in W^\bot$, then
\[\langle a \circ w, w\rangle = \sum_i a_iw_i^2 = 0,\]
so $a_i \ne 0$ implies $w_i = 0$, giving 
$a \circ w = 0$.  It follows that $\alpha = \alpha'$.

For the last part, fix $X \in \pi^{-1}(\xi)$ so that $\alpha = |\bmu_T(x)| \circ X \in \mathfrak{k}_\R^\bot$.  Then we have
\[L_{\xi*}(\xi_x) = \xi(\alpha) = \langle X, |\bmu_T(x)| \circ X \rangle \ge 0.\]
Finally, if $L_{\xi*}(\xi_x) = 0$, then the $i$th coordinate of $\bmu_I(\hat{x})$ vanishes for every $i$ such that $X_i \ne 0$.  It follows that $f_\xi(x) = 0$ and $x \in \M^\xi$.
\end{proof}

\begin{proof}[Proof of Lemma \ref{lem:orbit structure}]
From the previous lemma, it follows immediately that $t \mapsto \phi(t) := f_\xi(\xi(t)x)$ is constant or strictly increasing, and that if $\phi$ is constant, we have $f_\xi(x) = 0$.  

To see the remaining statements,
let $P$ be the preimage of $\xi(\C^*)$ under the quotient $(\C^*)^I \to (\C^*)^I/K = T$, 
fix a lift $\hat{x} \in \bmu_K^{-1}(0) \subset T^*\C^I$ of $x$.  Then $O := P\hat{x}$ is the open orbit in the (possibly non-normal) affine toric variety $\overline{O}$ with torus $\bar{P} := P/P_{\hat{x}}$.

Let $\eta \colon \R^I \to \mathfrak{p}_\R^*$ be the quotient map induced by the inclusion $P_\R \subset (S^1)^I$. 
Then the image of $\phi$ can be identified with 
\[\eta(\mu_{I, \R}(O) \cap V_\R) = \eta(\mu_{I,\R}(O))\cap \eta(V_\R),\]
under the identification of $\eta(V_\R)$ with 
$\R = \operatorname{Lie}(S^1)^*$ induced by $\xi$.

The map $\eta\circ \mu_{I, \R}|_{X}$ is a symplectic moment map for $\overline{O}$:
i.e., it is a function $\overline{O} \to \operatorname{Lie}(\bar P_\R)^*$ of the form 
\[y \mapsto \sum_{\gamma \in \Gamma_0} |f_\gamma(y)|^2 \gamma\]
where $\Gamma_0$ is a finite set of generators for the semigroup $\Gamma$ of regular functions $f_\gamma\colon \overline{O} \to \C$ for which $\bar p \mapsto f_\gamma(\bar p\cdot \hat x)$ is a character of $\bar{P}$.
This is clear when the group $P$ is all of $(\C^*)^I$; the functions $f_\gamma$ can be taken to be the coordinates $z_i$, $w_i$ which are not constant on $O$.  In the general case, the functions 
$f_\gamma$ can be taken to be the restrictions of these functions to $\overline{O}$.

The image of such a map is contained in the convex polyhedral cone $\sum_{\gamma \in \Gamma} \R \gamma$, and the image of $O$ is the relative interior of this cone.   A proof of this fact can be adapted from the proof in the projective case; see \cite[12.2.5]{CLS11}.

It follows that the image of $\phi$ is 
$\R$, $\R_{\ge 0}$ or $\R_{\le 0}$, and that if $0$ is in the image, it is the image of a fixed point in the closure of $\xi(\R_{>0})x$, completing the proof of the Lemma.
\end{proof}

\begin{lemma}\label{lem:distance estimate}
There is a constant $C_0>0$ so that 
 \[d(0\ha v, v) \le C_0|L_\xi(v)|\]
 for any $v \in \bV$.
\end{lemma}

\begin{proof}

Let $\Xi = \pi^{-1}(\R\xi)$. Then we have $L_\xi^{-1}(0) = \Xi^\bot \times V_\R \times V_\R$.  The vectors $\alpha$ which appear in Lemma \ref{lem:infinitesimal action} belong to $\R_{\ge 0}^I \circ \Xi$.  

In the proof of Lemma \ref{lem:infinitesimal action} we saw that $(\R_{\ge 0}^I \circ \Xi) \cap \Xi^\bot = \{0\}$. In fact the angle between nonzero vectors in $\R_{\ge 0}^I \circ \Xi$ and $\Xi^\bot$ is bounded away from $0$.  
To see this, note that 
for any $\xi \in \Xi$, the set $\R_{\ge 0}^I \circ \xi$ is closed, and as $\xi$ varies in $\Xi$, there are only finitely many different sets that appear.  Thus $\R_{\ge 0}^I \circ \Xi$ is closed in $\R^I$.  This together with the fact that $\Xi$ and $\R^I_{\ge 0} \circ \Xi$ are conical implies that the required angle bound exists.

If $x\in \bmu_T^{-1}(v)$, then for $t > 0$ we have 
\[\frac{d}{dt} t\ha v = 
\frac{d}{dt} \bmu_T(t \ast x) = \frac{|L_\xi(v)|}{|L_{\xi*}(\xi_x)|}\xi_x.
\]
Our angle bound gives a uniform upper bound on $|\xi_x|/|L_{\xi*}(\xi_x)|$, which in turn says that the length of 
tangent vectors to the curve $t \mapsto t \ha v$, $0 < t \le 1$, are at most a constant times $|L_\xi(v)|$.  The result follows.
\end{proof}

We can now show that our action $\ha$ on $\bV$ is continuous. 
We first prove that the action is continuous at points $(a, v)$ with $a > 0$ and $L_\xi(v) \ne 0$.  Without loss of generality we can assume $L_\xi(v) > 0$.  On the set $\R_{>0}\times L_\xi^{-1}(\R_{>0})$ the action is given by
\[(a,v) \mapsto s(aL_\xi(v), v) \hcd v,\]
where for any $t > 0$, $s(t,v)$ is the unique number in $\R_{>0}$ so that $L_\xi(s(t,v)\hcd v) = t$.  Therefore it is enough to show that $s$ is continuous.  But the graph of $s$ is the graph of $(a,v) \mapsto L_\xi(a \hcd v)$ restricted to $\R_{>0}\times L_\xi^{-1}(\R_{>0})$, so by the closed graph theorem  it is enough to show that $s$ is locally bounded.

Suppose that it is not locally bounded; then we can find compact sets $I \subset \R_{>0}$ and $C\subset L_\xi^{-1}(\R_{>0})$ so that $s(I\times C)$ is unbounded.  Thus there is a sequence $\{v_n\}$ in $C$ and an unbounded sequence $\{t_n\}$ in $\R$ such that 
$L_\xi(t_n\hcd v_n)$ lies in $I$ for all $n$.  Passing to a subsequence, we can assume that $\{v_n\}$ converges to  $v \in C$, and $L_\xi(t_n\hcd v_n)$ converges to $a \in I$.  But then we have $L_\xi(t_n\hcd v) \to a$, contradicting the unboundedness of $\{t_n\}$.  Thus $s$ is locally bounded, as desired.

To see that $\ha$ is continuous at a point $(a,v) \in \R_{>0} \times L_\xi^{-1}(0)$, we take  $u \in \bV$, $b \ge 0$.  We have
\begin{align*}
d(a \ha v, b\ha u) & = d(v, b\ha u) \le d(v,u) + d(u, 0\ha u) + d(0\ha u, b \ha u)\\
& \le d(v,u) + C_0(b+1)|L_\xi(u)|,
\end{align*}
using Lemma \ref{lem:distance estimate} twice.  
This can be made arbitrarily small by taking $(b,u)$ close to $(a,v)$.

Using this estimate, we see that $\ha$ is uniformly continuous on sets of the form $(0,a] \times C$, where $C \subset \bV$ is compact.
It follows that $\ha$ is continuous on all of $\R_{\ge 0} \times \bV$.

To see that the action ``upstairs" on $\M$ is continuous,
we use the fact that $\xi(\R_{>0})$ acts trivially on the angle coordinates.  Consider the non-Hausdorff space $\Theta = \C/\R_{>0}$; then we have a $(\C^*)^I$-equivariant embedding
\[T^*\C^I \hookrightarrow (\R^I \otimes \R^3) \times \Theta^I \times \Theta^I,\]
where the first coordinate is $\bmu_I$, and the second and third coordinates are the reductions of the $z_i$ and $w_i$ coordinates.  The compact torus $K_\R \subset K$ acts only on the second and third factors, so we have an embedding
\[\M \hookrightarrow \bV \times (\Theta^I \times \Theta^I)/K_\R.\]
Then the $\ast$-action of $\R_{>0}$ on $\M$ comes from the $\ha$-action on $\bV$ and the trivial action on $(\Theta^I \times \Theta^I)/K_\R$.  The continuity of the action at points of $\{0\} \times\M$ is easy to check.

This completes the proof of Lemma \ref{lem:continuous action} and thus of Proposition \ref{prop:hyperloc equals vanishing}.

\subsection{Relation to a conjecture of Finkelberg and Kubrak}

In this last section we show that Proposition \ref{prop:hyperloc equals vanishing} and the proof of Proposition \ref{prop:vanishing cycles iso} implies that a conjecture of Finkelberg and Kubrak \cite{FiKu} relating vanishing cycles and hyperbolic restriction is true for singular hypertoric varieties.

The Finkelberg-Kubrak conjecture applies to a class of Poisson singularities satisfying additional hypotheses.  It was noted without proof in \cite{FiKu} that affine hypertoric varieties satisfy these conditions; for completeness we will explain 
the details here.
 
Let $\M$ be the affine hypertoric variety defined by 
$V \subset \Z^I$.  Its coordinate ring is 
\[\C[\M]=\C[T^*\C^I]^K/\langle \mu_I^*(\mathfrak k_\C)\rangle,\]
where 
\[\mu_I^*\colon \C^I = \operatorname{Lie}(\C^*)^I \to \C[T^*\C^I], \; \mu_I^*(e_i)= z_iw_i,\]
is the pull-back by the 
moment map induced by the action of $(\C^*)^I$ on $T^*\C^I$.  This ring has a Poisson bracket induced by the bracket on $\C[T^*\C^I]=\C[z_i, w_i]_{i\in I}$ which satisfies $\{z_i,z_j\} = \{w_i, w_j\} = 0$, $\{z_i, w_j\} = \delta_{ij}$.

Under this bracket, $\M$ has finitely many Poisson leaves, namely the coarse strata $S_F$, $F \in \cF$.
Give $\M$ the $\C^*$-action induced by the action on $\C[T^*\C^I]$ for which $z_i$, $w_i$ have weight $1$ for all $i \in I$.  Under this action the Poisson bracket on  $R = \C[\M]$ has weight two, meaning that $\{R_k, R_\ell\} \subset R_{k+\ell-2}$.
Furthermore, $R$ is positively graded, and we have
$R_0 = \C$ and $R_1 = 0$.  

There is a Hamiltonian action of a group $G$ on $\M$, extending the action of $T$ and a moment map
\[\mu_G^*\colon \mathfrak{g} \to \C[\M]\]
whose image is the degree two part $\C[\M]_2$.  To see this, consider the decomposition
\[\C^I = V_1 \oplus V_2 \oplus \dots \oplus V_r\]
of $\C^I$ into its $K$-isotypic components.  Then 
$G = [GL(V_1) \times GL(V_2) \times \dots \times GL(V_r)]/K$ is the required group; the moment map 
$\mu^*_G$ is induced by the natural maps $\mathfrak{gl}(V_j) \to \Sym(V_j\oplus V_j^*) \subset \C[T^*\C^I]$. 

The proof of the main result of \cite{FiKu} shows that there is a $T$-invariant vector $f$ in the Zariski cotangent space $T^*_0\M$ which is regular, meaning that it represents a smooth point of the conormal variety to the stratification of $\M$ by Poisson leaves, under an embedding $\M \subset T_0\M$.  This means that for a perverse sheaf $P$ on 
$\M$ which is constructible with respect to the stratification, the vanishing cycles $\phi_f(\M)$ 
is a perverse sheaf supported at $0$ (at least in a neighborhood of $0$).  Thus $P \mapsto H^0(\phi_f(P)|_0)$ is an exact functor to vector spaces.  By \cite[Theorem 2.4]{FiKu}, the dimension of this vector space is the Euler characteristic of the stalk $P|_0$.  

It follows that $H^0(\phi_f(P)|_0)$ and $\Morse_I(P)$ have the same dimension.  But Proposition \ref{prop:hyperloc equals vanishing} gives a canonical isomorphism between these vector spaces, since by an easy calculation the $T$-invariant part of $T^*_0\M$ is spanned by $\mu^*_G(\mathfrak{t})$, so that $f$ is represented 
by a linear combination of the monomials $z_iw_i$, i.e.\ it is the complex moment map $\mu_T$ composed with a linear function.  This can be extended to an isomorphism between (trivial) local systems on the 
regular part of the image of $\mathfrak{t}$ in $T^*_0\M$.  (Note that the local system on the whole regular part of $T^*_0\M$ need not be trivial, however.)

\bibliography{./refs.bib}
\bibliographystyle{amsalpha}

\end{document}